\documentclass[noinfoline]{imsart}

\usepackage[a4paper,top=3cm, bottom=3cm, left=2.5cm, right=2.5cm]{geometry}
\usepackage{fancyhdr}

\RequirePackage{amsthm,amsmath,amsfonts,amssymb}
\usepackage{dsfont} 

\allowdisplaybreaks
\RequirePackage[colorlinks,citecolor=blue,urlcolor=blue]{hyperref}
\usepackage[capitalize]{cleveref}
\usepackage{mathrsfs} 

\usepackage{graphicx}
\graphicspath{{Images/}}
\usepackage{float}
\usepackage{textcomp}
\usepackage[figuresright]{rotating} 
\usepackage{pdfpages}
\usepackage{wrapfig}

\usepackage{bbm}

\usepackage{array}
\usepackage{booktabs} 
\usepackage{enumerate}
\usepackage{enumitem}
\usepackage{multicol}
\usepackage{tabularx}
\usepackage{tikz}
\usetikzlibrary{shapes.geometric, arrows.meta, positioning, shadows, backgrounds, fit, calc}
\usepackage{algorithm}
\usepackage[noend]{algpseudocode}
\makeatletter
\def\algbackskip{\hskip-\ALG@thistlm}
\makeatother

\renewcommand{\b}[1]{\mathbb{#1}} 
\newcommand{\m}[1]{\mathbf{#1}} 
\newcommand{\s}[1]{\mathscr{#1}} 
\newcommand{\f}[1]{\mathfrak{#1}} 
\renewcommand{\c}[1]{\mathcal{#1}} 
\newcommand{\ds}[1]{\mathds{#1}} 
\newcommand{\mb}[1]{\boldsymbol{#1}} 

\newcommand{\rd}{\mathrm{d}} 

\newcommand{\tr}[1]{\operatorname{Tr}\left[#1\right]}

\newcommand{\vertiii}[1]{{\left\vert\kern-0.25ex\left\vert\kern-0.25ex\left\vert #1 
    \right\vert\kern-0.25ex\right\vert\kern-0.25ex\right\vert}} 
\newcommand{\innpr}[2]{{\left\langle {#1},{#2}\right\rangle}} 
\newcommand{\QMD}{\operatorname{QMD}}

\newcommand{\Diag}{\operatorname{Diag}}
\newcommand{\Cov}{\operatorname{Cov}}
\renewcommand{\vec}{\operatorname{vec}}

\DeclareMathOperator*{\spann}{span}

\DeclareMathOperator*{\argmin}{\arg\min}
\DeclareMathOperator*{\argmax}{\arg\max}

\theoremstyle{plain}
\newtheorem{theorem}{Theorem}[section]
\newtheorem{lemma}[theorem]{Lemma}
\newtheorem{corollary}[theorem]{Corollary}
\newtheorem{proposition}[theorem]{Proposition}
\newtheorem{definition}[theorem]{Definition}
\newtheorem{remark}[theorem]{Remark}
\newtheorem{example}[theorem]{Example}

\begin{document}

\pagestyle{fancy}
\fancyhead{}
\renewcommand{\headrulewidth}{0pt} 
\fancyhead[CE]{P.~N. Mayer \& H. Yun}
\fancyhead[CO]{POVM Embedding with Its Application to QST}

\begin{frontmatter}
\title{{\large Kernel Embedding for Operator-Valued Measures and \\
Its Application to Quantum Tomography}}

\begin{aug}
\author[A]{\fnms{Philipp N.}~\snm{Mayer}\ead[label=e1]{philipp.mayer@epfl.ch}\orcid{0009-0004-1466-5000}}
\and
\author[A]{\fnms{Ho}~\snm{Yun}\ead[label=e2]{ho.yun@epfl.ch}\orcid{0009-0001-9976-6545}}
\thankstext{t1}{Equally contributed.}
\address[A]{Ecole Polytechnique F\'ed\'erale de Lausanne, \printead{e1,e2}}
\end{aug}

\begin{abstract}
This paper introduces the Quantum Covariance Embedding, which embeds Positive Operator-Valued Measures into a tensor product of a Reproducing Kernel Hilbert Space and the quantum state space via a tensorized Bochner integral. This construction induces the Quantum Maximum Discrepancy that metrizes the space of quantum measurements. Applying this framework to Quantum State Tomography, we reformulate density estimation as a tensorized kernel regression, enabling optimal inference without the basis-dependent sparsity constraints that restrict existing methods. We develop a unified geometric design theory for quantum Gram superoperators, establishing that Unitary Designs are strictly E-optimal experimental designs and thus statistically superior to Pauli observables. For general structure-free estimation, we derive the exact minimax lower bound and prove that our tensorized estimators achieve this optimal rate. Furthermore, we introduce the QUAntum Regression with Kernels (QUARK) estimator to accommodate the spectral geometry of physical implementations, deriving central limit theorem and concentration inequalities. To facilitate practical estimation, we establish the exactness of trace-preserving projections and demonstrate efficient estimation under mutually unbiased bases via the fast Walsh-Hadamard transform.
\end{abstract}

\begin{keyword}[class=MSC]
\kwd[Primary ]{81P50}   
\kwd{46E22} 
\kwd[; secondary ]{47N50}   
\kwd{62G05} 
\end{keyword}

\begin{keyword}
\kwd{Quantum state tomography}
\kwd{Positive operator-valued measure}
\kwd{Reproducing kernel Hilbert space}
\kwd{Optimal experimental design}
\kwd{Tensorized regression}
\end{keyword}
\end{frontmatter}

\maketitle

\section{Introduction}
Quantum technology has matured from theoretical curiosity to experimental reality, a shift underscored by the 2022 Nobel Prize in Physics. As quantum platforms scale, verifying and characterizing these systems has emerged as a central challenge. Fundamentally, this is a statistical problem: latent parameters must be inferred through macroscopic evidence, and one must rigorously quantify the associated uncertainties \cite{d2002quantum, christandl2012reliable, cramer2010efficient}. To address this, this paper develops the theory of kernel embeddings for quantum mechanics. By mapping the measurement process itself into an infinite-dimensional feature space, we provide a unified methodology for quantum state estimation that bridges nonparametric statistics with quantum physics.

Whereas classical statistics views data as samples directly generated from a probability distribution, the probability law driving the data-generating mechanism in quantum mechanics is determined by the interaction of a quantum state with a physical apparatus \cite{holevo2011probabilistic,cohen2019quantum,nielsen2010quantum}. Mathematically, the physical apparatus is modeled by an \emph{operator-valued} Borel measure $\mb{\nu}$ that acts upon an unobservable state of the system $\mb{\rho}$. The physical measurement process fuses these two objects together via the Born rule, reducing the operator-valued measure into a classical, scalar probability distribution $\mb{\nu}_{\mb{\rho}}$ over the measurement outcomes \cite{hall2013quantum,talagrand2022quantum}. The statistician is then presented with a finite set of independent and identically distributed (i.i.d.) observations drawn from $\mb{\nu}_{\mb{\rho}}$. If the state is known and the apparatus is uncertain, the task is Quantum Detector Tomography. On the other hand, if the apparatus is known and the state is uncertain, the task is Quantum State Tomography (QST).

The problem of Quantum State Tomography has attracted a vast literature, establishing minimax estimation rates \cite{gutaminmax, dongxiaminmax}, exploring compressed sensing for low-rank states \cite{gross2010quantum, gross2011recovering, flammia2012quantum}, and developing estimators under a range of structural assumptions \cite{artiles2005invitation, wang2013asymptotic, verdeil2022pure, gross2011recovering, acharya2025pauli, d2002quantum, christandl2012reliable, cramer2010efficient, cai2016optimal, georges2025pauli}. In particular, the predominant framework centers on tensorized Pauli observables \cite{acharya2025pauli, cai2016optimal, georges2025pauli}, where its sharpest statistical guarantees hold under sparsity in the Pauli basis. Yet, this restriction carries a severe fundamental consequence: by the Gottesman-Knill theorem \cite{gottesman1998heisenberg}, quantum states that are sparse in the Pauli basis admit efficient simulation also on classical computers, offering no true quantum advantage. Consequently, existing sparsity-based estimators may suffer from substantial bias precisely in the regimes where quantum computation is genuinely powerful. Furthermore, least-squares approaches treat measurement outcomes as abstract categorical labels, ignoring the underlying spectral geometry of the physical hardware  \cite{arecchi1972atomic,pegg1989phase,gottesman2001encoding}.

This paper takes a different route based on the theory of kernel embeddings for scalar probability measures \cite{fukumizu2004dimensionality, sriperumbudur2008injective, sriperumbudur2010hilbert,bach2022information}. Rather than imposing basis-dependent structural constraints on the unobservable state, we place the measurement mechanism itself at the center of the statistical analysis. Yet, extending classical kernel methods to quantum mechanics is non-trivial, as standard arguments for scalar measures do not readily apply to the non-commutative nature of the Positive Operator-Valued Measure (POVM) that governs the measurement process. To resolve this, our central construction introduces the \textbf{Quantum Covariance Embedding (QCE)}. By considering a tensorization of a Reproducing Kernel Hilbert Space (RKHS) and the quantum Hilbert space, the QCE rigorously embeds the POVM into an infinite-dimensional feature space. Consequently, this construction provides a unified view for both classical kernel methods and quantum physics, as illustrated in \cref{fig:pipeline}.

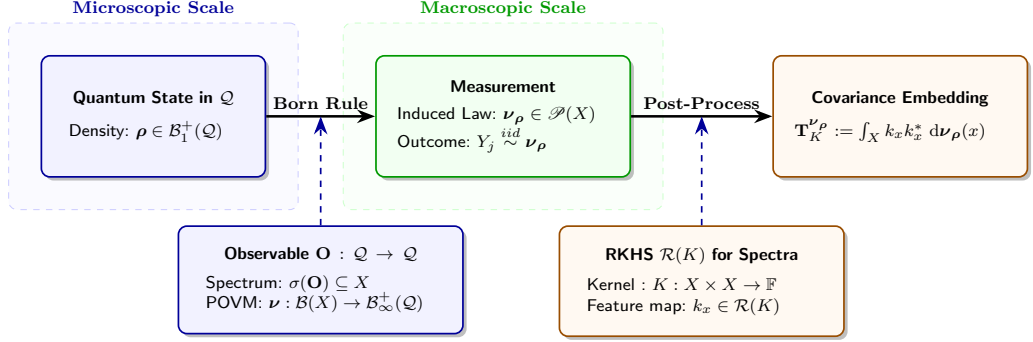
\begin{figure}[H]
    \centering
    \scalebox{0.8}{
        \begin{tikzpicture}[
    font=\sffamily,
    node distance=1.5cm and 2cm,
    >={Stealth[length=3mm]},
    base/.style={
        draw, 
        thick, 
        rounded corners, 
        align=center, 
        drop shadow, 
        inner sep=10pt,
    },
    quantum/.style={base, fill=blue!5, draw=blue!60!black},
    prob/.style={base, fill=green!5, draw=green!60!black},
    kernel/.style={base, fill=orange!5, draw=orange!60!black},
    title/.style={font=\bfseries\large, inner sep=5pt},
    maths/.style={font=\small, align=left}
]

    \node[quantum, text width=3cm, minimum height=2cm] (state) {
        \textbf{Quantum State in $\c{Q}$} \\
        \vspace{0.2cm}
        \begin{minipage}{2.8cm}
            \footnotesize
            Density: $\mb{\rho} \in \c{B}_1^+(\c{Q})$
        \end{minipage}
    };

    \begin{scope}[on background layer]
        \node[fit=(state), draw=blue!20, dashed, inner sep=15pt, rounded corners, fill=blue!2, label={[blue!60!black, font=\bfseries]above:Microscopic Scale}] (q_group) {};
    \end{scope}

    \node[prob, right=1.8cm of state, text width=3.5cm, minimum height=2cm] (law) {
        \textbf{Measurement} \\
        \vspace{0.2cm}
        \begin{minipage}{3.8cm}
            \footnotesize
            Induced Law: $\mb{\nu}_{\mb{\rho}} \in \s{P}(X)$ \\
            Outcome: $Y_{j} \stackrel{iid}{\sim} \mb{\nu}_{\mb{\rho}}$
        \end{minipage}
    };
    
    \begin{scope}[on background layer]
        \node[fit=(law), draw=green!20, dashed, inner sep=15pt, rounded corners, fill=green!2, label={[green!60!black, font=\bfseries]above:Macroscopic Scale}] (macro_group) {};
    \end{scope}

    \draw[->, very thick] (state.east) -- (law.west) node[midway, above, font=\bfseries] (born) {Born Rule};

    \node[quantum, below=1.8cm of born, text width=4cm] (povm) {
        \textbf{Observable $\m{O}: \c{Q} \to \c{Q}$} \\
        \vspace{0.2cm}
        \begin{minipage}{3.8cm}
            \footnotesize
            Spectrum: $\sigma(\m{O}) \subseteq X$ \\
            POVM: $\mb{\nu}: \c{B}(X) \to \c{B}_\infty^+(\c{Q})$
        \end{minipage}
    };

    \draw[->, thick, dashed, blue!60!black] (povm.north) -- (born.south);

    \node[kernel, right=1.8cm of macro_group.east, text width=3.5cm, minimum height=2cm] (cov) {
        \textbf{Covariance Embedding} \\
        \vspace{0.2cm}
        \begin{minipage}{3.5cm}
            \footnotesize
            $\m{T}_{K}^{\mb{\nu}_{\mb{\rho}}} := \int_{X} k_{x} k_{x}^{*} \ \rd \mb{\nu}_{\mb{\rho}} (x)$
        \end{minipage}
    };

    \draw[->, very thick] (law.east) -- (cov.west) node[midway, above, font=\bfseries] (embed) {Post-Process};

    \node[kernel, below=1.8cm of embed, text width=4cm] (rkhs) {
        \textbf{RKHS $\c{R}(K)$ for Spectra} \\
        \vspace{0.2cm}
        \begin{minipage}{3.8cm}
            \footnotesize
            Kernel : $K: X \times X \to \b{F}$ \\
            Feature map: $k_{x} \in \c{R}(K)$
        \end{minipage}
    };

    \draw[->, thick, dashed, blue!60!black] (rkhs.north) -- (embed.south);

\end{tikzpicture}
}
\caption{Pipeline of Quantum Covariance Embedding}
\label{fig:pipeline}
\end{figure}
A profound consequence of this embedding is the derivation of the \textbf{Quantum Maximum Discrepancy (QMD)}. Extending the Maximum Discrepancy of kernel methods \cite{gretton2012kernel, sriperumbudur2010hilbert, bach2022information} to the operator level, the QMD establishes a novel metrization on the space of POVMs. This provides a statistical tool to quantify the distance between distinct quantum measurement processes --- or between a theoretical apparatus and its finite-sample empirical counterpart --- independent of the underlying quantum state. 
Applied to QST, this embedding recasts tomography as a kernel regression. However, because our estimand is a density operator rather than a vector, this regression is inherently \emph{tensorized}. This lifts the problem to a tensor-operator equation, aware of the spectral geometry of the measurement design, enabling estimation without imposing specific classical simulability assumptions such as Pauli sparsity.
By bridging quantum theory with tensorized kernel regression, this paper makes three primary contributions in the context of QST:
\begin{enumerate}[leftmargin = *]
\item \textbf{Unified Design Theory:} We develop a basis-independent formulation for QST, characterizing statistical identifiability through the design tensor. By analyzing the action of quantum Gram superoperators via representation theory \cite{faraut2008analysis}, we establish that \textbf{Unitary Designs} (including Mutually Unbiased Bases \cite{klappenecker2005mutually}), which serve as the quantum analogue of an orthogonal design matrix, are E-optimal experimental designs \cite{pukelsheim2006optimal}. We quantify the fundamental loss of information caused by eigenvalue multiplicities via an explicit damping coefficient, proving analytically why entangled, isotropic measurements are statistically superior to local measurements such as tensorized Pauli observables.

\item \textbf{Minimax Optimality:} Moving away from low-rank and basis-sparse assumptions, we address the fundamental limit of general density estimation. Under an identifiable design with $n$ observables and $r$ independent shots per observable, we derive the exact minimax lower bound for the squared-error risk, establishing that the information-theoretic limit scales as $O(q^{2}/(nr))$ for a state space of dimension $q$. We further prove that our tensorized estimators achieve this optimal parametric rate in the dense regime where classical sparsity-based methods fail.

\item \textbf{Kernel Regression:} We introduce the \textbf{QUAntum Regression with Kernels (QUARK)} estimator, which encodes the metric-aware topology between spectra to accommodate the geometric reality of physical implementations. With a Dirac 0--1 kernel, QUARK natively recovers the standard unconstrained least-squares estimator; with a generic smooth kernel, it enforces spectral smoothness, yielding an estimator analogous to operator-level ridge regression. We establish a Quantum Representer Theorem that bounds the Schatten-norm discrepancy of kernel embeddings, and we derive a formal Central Limit Theorem, along with Bennett-type concentration inequalities.
\end{enumerate}

To ensure our theoretical estimators are practically realizable, we provide algorithmic implementations. We prove that applying the computationally efficient Trace-Preserving Projection algorithm \cite{smolin2012efficient} to our relaxed estimator yields the exact, globally optimal solution to the fully constrained physical problem. Additionally, we demonstrate how the least-squares estimate under Mutually Unbiased Bases can be computed in $O(q \log q)$ time via fast Walsh-Hadamard transforms \cite{fino1976unified}, in contrast to Pauli tomography $O(q^{2})$.

The remainder of the main body is organized as follows. \cref{sec:background} establishes the necessary background on the operator formalism of quantum mechanics, alongside a literature review of QST and kernel embeddings. \cref{sec:RKHS} constructs the RKHS framework, embedding POVMs as operators within a tensorization of the RKHS and the quantum Hilbert space. \cref{sec:QKT} applies this framework to Quantum State Tomography, deriving the tensorized regression, the minimax lower bounds, and the asymptotic theory for the QUARK estimator. Finally, \cref{sec:simstudy} presents numerical simulations confirming our theoretical scaling laws.

\section{Preliminaries}\label{sec:background}
\subsection{Notations}
$\b{F}$ denotes the underlying scalar field, either $\b{R}$ or $\b{C}$. All Hilbert spaces are assumed separable. We adopt the sesquilinear convention that the inner product $\innpr{\cdot}{\cdot}$ is linear in the first argument and conjugate-linear in the second. The symbol $^{*}$ denotes the adjoint operator, and $^{\dagger}$ denotes the Moore-Penrose pseudoinverse. The indicator function for a set $A$ is denoted by $\ds{1}_{A}$. To clarify the hierarchy of mathematical objects, we employ 
\textbf{calligraphic} fonts (e.g., $\c{Q}, \c{H}, \c{B}, \c{R}$) to denote Hilbert or Banach spaces, and \textbf{uppercase roman} (e.g., $\m{O}$) to denote linear operators. 

Given a generic $\b{F}$-Hilbert space $\c{H}$, $\m{\Pi}_{\c{V}}$ denotes the orthogonal projection onto a closed subspace $\c{V} \subset \c{H}$. We utilize the following Banach spaces of operators on $\c{H}$ \cite{conway2019functional}:
\begin{itemize}
    \item $\c{B}_{\infty}(\c{H})$: The space of bounded operators, equipped with the operator norm $\vertiii{\cdot}_{\infty}$.
    \item $\c{B}_{0}(\c{H})$: The space of compact operators, equipped with the same norm $\vertiii{\cdot}_{\infty}$.
    \item $\c{B}_{s}(\c{H})$ (for $s \ge 1$): The Schatten $s$-class, equipped with the norm $\vertiii{\cdot}_{s}$. Specifically, $\c{B}_{1}(\c{H})$ denotes trace-class operators, and $\c{B}_{2}(\c{H})$ denotes Hilbert-Schmidt operators.
\end{itemize}
An operator $\m{T}$ is called self-adjoint if $\m{T} = \m{T}^{*}$, and non-negative definite (n.n.d.) if it is self-adjoint and satisfies $\innpr{\m{T} f}{f} \ge 0$ for all $f$, denoted by $\m{T} \succeq \m{0}$. The trace of $\m{T} \in \c{B}_{1}(\c{H})$ is defined as $\tr{\m{T}} := \sum_{l} \innpr{\m{T} e_{l}}{e_{l}}$ for any orthonormal basis $\{e_{l}\}_{l} \subset \c{H}$. We denote the convex cone of n.n.d. bounded operators by $\c{B}_{\infty}^{+}(\c{H})$, and similarly define $\c{B}_{s}^{+}(\c{H})$. Note that $\tr{\m{T}} = \vertiii{\m{T}}_{1}$ if and only if $\m{T} \in \c{B}_{1}^{+}(\c{H})$.

\subsection{Quantum Measurements}\label{ssec:quant:msr}
To model the probabilistic nature of microscopic phenomena, quantum theory describes the states and observables of a system through operator algebra. We review the minimal axiomatic framework of quantum mechanics required to develop our statistical theory; for more extensive treatments, we refer to \cite{holevo2011probabilistic,hall2013quantum,talagrand2022quantum, cohen2019quantum,nielsen2010quantum}.

In (non-relativistic) quantum mechanics, the state space of a physical system is modeled over a complex separable Hilbert space $\c{Q}$, which we call the \textbf{Quantum Hilbert Space} (QHS). A \textbf{pure state} corresponds to a one-dimensional subspace of $\c{Q}$, uniquely represented by a rank-one orthogonal projection $\mb{\rho} = \psi \psi^{*}$, where $\psi \in \c{Q}$ is a unit vector. 
To account for mixtures of pure states, this framework generalizes to \textbf{density operators}, defined as non-negative definite, unit-trace operators $\mb{\rho} \in \b{S}(\c{B}_{1}^{+}(\c{Q}))$. Geometrically, the space of density operators $\b{S}(\c{B}_{1}^{+}(\c{Q}))$ forms a closed convex set whose extreme points are precisely the pure states \cite{holevo2011probabilistic}. Consequently, if $\mb{\rho}$ has a rank strictly greater than one, it cannot be an extreme point and is termed a \textbf{mixed state}. Notable examples of the underlying QHS include:
\begin{itemize}[leftmargin = *]
    \item $\c{Q} = \b{C}^{q}$ for a $q$-level system, called a qudit. It is called a qubit (e.g., the spin of an electron) when $q=2$. Throughout the work, we assume $q \ge 2$ for non-triviality.
    \item $\c{Q} = \c{L}_{2}(\b{R}, \b{C})$ for wavefunctions $\psi: \b{R} \to \b{C}$ (e.g., the spatial wavefunction of a massive particle on a real line).
\end{itemize}

Crucially, the quantum state itself resides in the microscopic regime and is not directly observable. Information can only be extracted through its interaction with a macroscopic apparatus, which constitutes the physical definition of measurement. 
An \textbf{observable} is a physical quantity measured by an experimental apparatus, associated with a self-adjoint operator $\m{O}$ on $\c{Q}$.
As a motivating example, consider an observable $\m{O} \in \b{C}^{q \times q}$ for the $q$-level system $\c{Q} = \b{C}^{q}$. Assuming a non-degenerate spectrum, $\m{O}$ admits the spectral decomposition:
\begin{align*}
    \m{O} = \sum_{k=1}^{q} \lambda_{k} \mb{\Pi}_{k}, \quad \text{where} \quad \sigma(\m{O}) = \{\lambda_{k} : k =1, \cdots, q \} \subset \b{R}, \quad \mb{\Pi}_{k} = \psi_{k} \psi_{k}^{*}.
\end{align*}
The measurement outcome $Y$ is a random variable taking values in the spectrum $\sigma(\m{O})$. The \textbf{Born rule} dictates that its probability mass, conditioned on the apparatus $\m{O}$ and the input state $\mb{\rho}$, is governed by the trace:
\begin{align*}
    \mb{\nu}^{\m{O}}_{\mb{\rho}}(\{\lambda_{k}\}) := \b{P}(Y = \lambda_{k} \vert \mb{\rho}) = \innpr{\mb{\rho} \psi_{k}}{\psi_{k}}_{\c{Q}} = \tr{\mb{\rho} \mb{\Pi}_{k}}.
\end{align*}
This formulation guarantees a legitimate probability law since $\sum_{k=1}^{q} \mb{\nu}^{\m{O}}_{\mb{\rho}}(\{\lambda_{k}\}) = \tr{\mb{\rho}} = 1$. It also follows that the conditional expectation of the measurement outcome is $\b{E}(Y \vert \mb{\rho}) = \tr{\mb{\rho} \m{O}}$. Throughout the work, all probabilistic statements regarding $Y$ are \emph{conditional} on the input state $\mb{\rho}$.
Note that the possible values of the measurement outcome are solely determined by the apparatus $\m{O}$, while the density operator $\mb{\rho}$ determines the probability mass. To decouple the input state from the apparatus, we consider the operator-valued measure 
\begin{align*}
    \mb{\nu}^{\m{O}} = \sum_{k=1}^{q} \delta_{\lambda_{k}} \mb{\Pi}_{k} : \s{B}(\sigma(\m{O})) \to \b{C}^{q \times q},
\end{align*}
that acts on the density $\mb{\rho} \in \b{S}(\c{B}_{1}^{+}(\c{Q}))$ via $\mb{\nu}^{\m{O}}_{\mb{\rho}}(B) = \tr{\mb{\rho} \mb{\nu}^{\m{O}}(B)}$ for any Borel set $B \subseteq \sigma(\m{O})$. Mathematically, this measure maps subsets of the spectrum to orthogonal projections onto the corresponding eigenspaces.

While this logic extends naturally to compact operators as they possess discrete spectra, physical observables such as position and momentum are represented by operators with continuous spectra. Consequently, the notion of a \textbf{Projection-Valued Measure} (PVM) serves as a necessary generalization, as evaluating an atomless measure at individual points trivially yields zero \cite{hall2013quantum}. A PVM is a Borel measure $\mb{\nu}: \s{B}(X) \to \c{B}_{\infty}^{+} (\c{Q})$ over a locally compact Hausdorff (LCH) space $X$, satisfying $\mb{\nu}(X) = \m{I}_{\c{Q}}$ and the idempotence property $\mb{\nu}(B)^{2} = \mb{\nu}(B)$ for any Borel set $B$. Idempotency implies orthogonality, meaning $\mb{\nu}(B) \mb{\nu}(B') = \m{0}$ whenever $B \cap B' = \emptyset$. By the \textbf{Spectral Theorem}, there is a one-to-one correspondence between self-adjoint observables $\m{O}$ and PVMs $\mb{\nu}^{\m{O}}$ over $X = \sigma(\m{O}) \subseteq \b{R}$ in the following sense:
\begin{align*}
    \mb{\nu}^{\m{O}}_{\mb{\rho}}(B) := \b{P}(Y \in B \vert \mb{\rho}) = \tr{\mb{\rho} \mb{\nu}^{\m{O}}(B)}, \quad B \in \s{B}(\sigma(\m{O})), \quad \mb{\rho} \in \b{S}(\c{B}_{1}^{+}(\c{Q})).
\end{align*}
When there is no risk of confusion regarding the observable $\m{O}$, we will simply denote the measure as $\mb{\nu}$ and omit the superscript.

In practice, measurement processes are frequently subject to noise or involve interactions with a larger environmental system (See \cref{ssec:uq:povm} for concrete examples). This necessitates the following definition, which relaxes the idempotence requirement while preserving the validity of the induced probability distributions.

\begin{definition}[POVM]\label{def:POVM}
Let $\c{Q}$ be a QHS and $X$ be an LCH space equipped with the Borel $\sigma$-algebra $\s{B}(X)$. A set function $\mb{\nu}: \s{B}(X) \to \c{B}_{\infty}^{+} (\c{Q})$ is called a \textbf{Positive Operator-Valued Measure} (POVM) if:
\begin{enumerate}
\item $\mb{\nu}(\emptyset) = \m{0}$ and $\mb{\nu}(X) = \m{I}_{\c{Q}}$.
\item For all countable collections $\{B_{k}\}_{k=1}^{\infty}$ of mutually disjoint Borel sets,
\begin{equation*}
    \mb{\nu} \left(\bigcup_{k=1}^{\infty} B_{k} \right) = \sum_{k=1}^{\infty} \mb{\nu} (B_{k}),
\end{equation*}
where the convergence is with respect to the Strong Operator Topology (SOT) \cite{conway2019functional}.\footnote{Due to Vigier's theorem \cite{conway2025course}, the Weak Operator Topology (WOT) induces the equivalent convergence.}
\end{enumerate}
We denote the space of POVMs by $\s{P}(X, \c{Q})$ and omit the dependence on $\c{Q}$ when $\c{Q} = \b{F}$, as this reduces to the space of classical probability measures $\s{P}(X)$. 
\end{definition}

Given a POVM $\mb{\nu} \in \s{P}(X, \c{Q})$, its action on a density operator $\mb{\rho} \in \b{S}(\c{B}_{1}^{+}(\c{Q}))$ induces a classical probability measure $\mb{\nu}_{\mb{\rho}} \in \s{P}(X)$ via:
\begin{equation}\label{eq:POVM2prob}
    \mb{\nu}_{\mb{\rho}} (B) := \tr{\mb{\rho} \mb{\nu}(B)}, \quad B \in \s{B}(X).
\end{equation}
Additionally, for any $\phi, \psi \in \c{Q}$, we define an $\b{F}$-valued measure:
\begin{equation*}
    \mb{\nu}_{\phi, \psi} (B) := \innpr{\mb{\nu}(B) \phi}{\psi}_{\c{Q}}, \quad B \in \s{B}(X).
\end{equation*}
Note that $\mb{\nu}_{\phi, \psi} = \overline{\mb{\nu}_{\psi, \phi}}$, and for pure states $\mb{\rho} = \phi \phi^{*}$, the induced probability measure \eqref{eq:POVM2prob} reduces to $\mb{\nu}_{\phi, \phi}(B) = \innpr{\mb{\nu}(B) \phi}{\phi}_{\c{Q}}$. This construction allows us to formalize the integration of scalar functions with respect to $\mb{\nu}$ via sesquilinear forms:

\begin{definition}[Functional Calculus]
Given a POVM $\mb{\nu} \in \s{P}(X, \c{Q})$ over an LCH space $X$, and a measurable function $f: X \to \b{F}$, define the domain of integration by
\begin{align*}
    \c{D}_{f} := \Big\{ \phi \in \c{Q} : \int_{X} |f(x)|^{2} \rd \mb\nu_{\phi, \phi}(x) < \infty \Big\},
\end{align*}
which forms a linear subspace. Provided that $\c{D}_{f}$ is dense in $\c{Q}$, we define the operator integral $\int_{X} f(x) \rd \mb{\nu} (x)$ as the unique (potentially unbounded) operator satisfying
\begin{align*}
    \innpr{\left( \int_{X} f(x) \rd \mb{\nu} (x) \right) \phi}{\psi}_{\c{Q}} = \int_{X} f(x) \rd \mb\nu_{\phi, \psi}(x), \quad \phi \in \c{D}_{f}, \quad \psi \in \c{Q}.
\end{align*}
\end{definition}

We remark two cases of immediate relevance. 
When $f$ is bounded ($\sup_{x \in X} |f(x)| < \infty$), $\c{D}_{f} = \c{Q}$ for any POVM $\mb{\nu}$, and $\int_{X} f \rd \mb{\nu} \in \c{B}_{\infty}(\c{Q})$ is uniquely determined by the Riesz representation theorem.
When $\mb{\nu} \in \s{P}(X, \c{Q})$ is a PVM, $\c{D}_{f}$ is unconditionally dense in $\c{Q}$ for any measurable function $f: X \to \b{F}$ \cite{hall2013quantum}. 
For any bounded measurable function $f$ (or unbounded $f$ provided the product $\mb{\rho} \int_{X} f \rd \mb{\nu}$ is trace-class), the quantum expectation matches the classical expectation: 
\begin{align*}
    \b{E}[f(Y) \vert \mb{\rho}] = \int_{X} f \rd \mb{\nu}_{\mb{\rho}} = \tr{\mb{\rho} \left( \int_{X} f \rd \mb{\nu} \right) }, \quad Y \mid \mb{\rho} \sim \mb{\mb{\nu}_{\rho}},
\end{align*}
and we recover \eqref{eq:POVM2prob} by applying the indicator function $f = \ds{1}_{B}$.

\subsection{Related Work}

The POVM framework presents a largely unexplored opportunity for statistical generalization. While there is an abundant literature on kernel \emph{mean} embeddings for scalar-valued probability measures \cite{fukumizu2004dimensionality, muandet2017kernel,smola2007hilbert, sriperumbudur2008injective, sriperumbudur2010hilbert, carmeli2010vector, gretton2012kernel}, a rigorous kernel embedding framework for operator-valued measures is virtually non-existent. 
In contrast to a classical probability measure, a quantum state $\mb{\rho} \in \b{S}(\c{B}_{1}^{+}(\c{Q}))$ is a unit-trace, non-negative definite operator, rendering it structurally analogous to a covariance matrix rather than a mean vector. Consequently, extending kernel methods to the quantum setting fundamentally requires \emph{covariance embedding} proposed in \cite{bach2022information}.

While \cite{bach2022information} leverages quantum information theory tools, 
that analysis assumes the underlying QHS itself possesses a reproducing kernel. Yet, canonical QHSs (such as $\c{Q} = \c{L}_{2}(\b{R}, \b{C})$) generally lack this RKHS structure, and our work addresses a critical theoretical gap by formulating covariance embeddings that capture the non-commutative interplay between general operator-valued measures $\mb{\nu}$ and density operators $\mb{\rho}$ without requiring the QHS to be an RKHS.
Developing this operator-valued embedding theory is our primary focus. We apply this framework to Quantum State Tomography (QST), focusing on the statistical reconstruction of an unknown state $\mb{\rho}$ given a known measurement apparatus $\mb{\nu}$. We select QST because this state-estimation problem naturally parallels classical statistical inference, distinguishing it from dual tasks like quantum detector tomography (estimating an unknown apparatus $\mb{\nu}$), which we leave for future work.

QST is an ill-posed inverse problem. Since individual quantum states are unobservable, reconstruction requires repeated measurements on an ensemble of identically prepared systems. Moreover, a single observable is mathematically insufficient: for $\c{Q} = \b{C}^{q}$, the real affine dimension of $\b{S}(\c{B}_{1}^{+}(\c{Q}))$ is $q^{2}-1$, yet one observable yields at most $q$ distinct spectral values. 
To resolve this dimensional deficit, the quantum physics literature has identified specific measurement schemes, such as Pauli measurements \cite{nielsen2010quantum,cai2016optimal} and Mutually Unbiased Bases \cite{klappenecker2005mutually}. To reconstruct the state from these fixed designs, specialized statistical approaches have been developed, such as nuclear-norm penalized compressed sensing to promote low-rankness \cite{gross2010quantum}, or Bayesian regression utilizing unitarily invariant priors \cite{blume2010optimal, lukens2020practical}. 
However, we caution against the application of standard classical high-dimensional regression techniques (such as LASSO) to this inverse problem. Because quantum mechanics is unitarily invariant, imposing $\ell_1$-parsimony on the coefficients of a fixed design (e.g., assuming the state is sparse in the Pauli basis \cite{cai2016optimal}) may lack physical justification unless the specific environment inherently breaks that symmetry. 
To provide a rigorous foundation free of basis-dependent artifacts, we devote \cref{ssec:identi:dsgn} to translating the physics concept of informational completeness into the formal statistical language of identifiability and design tensors (see \cref{prop:sphe:design}). Furthermore, \cref{thm:minimax} establishes the minimax lower bound applicable to any identifiable design, delineating the fundamental information-theoretic limits in its own right.

\section{Quantum Covariance Embedding}\label{sec:RKHS}

\subsection{RKHS Framework}
\begin{definition}[RKHS]
Let $\c{R}$ be a Hilbert space of $\b{F}$-valued functions defined on a set $X$. A bivariate function $K:X \times X \to \b{F}$ is called a reproducing kernel for $\c{R}$ if:
\begin{enumerate}
    \item For any $x \in X$, the feature map $k_{x}:= K(\cdot, x) : X \to \b{F}$ belongs to $\c{R}$.
    \item For any $f \in \c{R}$ and $x \in X$, the reproducing property holds, i.e., $f(x)=\langle f, k_{x} \rangle_{\c{R}}$.
\end{enumerate}
Such a Hilbert space is called a Reproducing Kernel Hilbert Space (RKHS). 
\end{definition}

The Moore–Aronszajn theorem \cite{aronszajn1950theory} establishes a one-to-one correspondence between positive definite kernels and RKHSs through $K(x,y)= k_{y}(x) = \innpr{k_{y}}{k_{x}}_{\c{R}}$, justifying the notation $\c{R}=\c{R}(K)$ \cite{paulsen2016introduction,hsing2015theoretical}.
The following theorem establishes the fundamental operator that will serve as our embedding by proving its existence and uniqueness via sesquilinear forms.

\begin{theorem}\label{thm:quant:cov:scalar}
Let $X$ be an LCH space and $K$ be a $\c{C}_{0}$-kernel. Then, for any POVM $\mb{\nu} \in \s{P}(X, \c{Q})$, there exists a unique bounded operator $\m{T}_{K}^{\mb{\nu}} \in \c{B}_{\infty}^{+}(\c{R}(K) \otimes \c{Q})$ satisfying
\begin{equation*}
    \innpr{\m{T}_{K}^{\mb{\nu}} (f \otimes \phi)}{g \otimes \psi}_{\c{R}(K) \otimes \c{Q}} = \int_{X} f(x) \overline{g(x)} \ \rd \mb{\nu}_{\phi, \psi}(x), \quad \forall f, g \in \c{R}(K), \, \phi,\psi \in \c{Q}.
\end{equation*}
It holds that $\vertiii{\m{T}_{K}^{\mb{\nu}}}_{\infty} \le \sup_{x \in X} K(x, x)$. Additionally, $\m{T}_{K}^{\mb{\nu}} \in \c{B}_{1}^{+}(\c{R}(K) \otimes \c{Q})$ if and only if $\int_{X} K(x, x) \ \rd \mb{\nu}(x) \in \c{B}_{1}^{+}(\c{Q})$. In this case, the trace-norm is given by:
\begin{align*}
    \vertiii{\m{T}_{K}^{\mb{\nu}}}_{1} = \vertiii{\int_{X} K(x, x) \ \rd \mb{\nu}(x)}_{1}.
\end{align*}
\end{theorem}

The bounded operator in \cref{thm:quant:cov:scalar} can be equivalently defined as the tensorized Bochner integral $\m{T}_{K}^{\mb{\nu}} = \int_{X} k_{x} k_{x}^{*} \otimes \rd \mb{\nu} (x)$, as discussed in \cref{sec:tens:boch}.
Also, applying \cref{thm:quant:cov:scalar} to $\mb{\nu}_{\mb{\rho}} \in \s{P}(X)$ guarantees the existence of $\m{T}_{K}^{\mb{\nu}_{\mb{\rho}}} \in \c{B}_{1}^{+}(\c{R}(K))$, given by
\begin{equation}\label{eq:tr:contra:embed:scalar}
    \m{T}_{K}^{\mb{\nu}_{\mb{\rho}}} := \int_{X} k_{x} k_{x}^{*} \ \rd \mb{\nu}_{\mb{\rho}} (x) = \b{E}[k_{Y} k_{Y}^{*} | \mb{\rho}],
\end{equation}
which acts as an integral operator:
\begin{align*}
    &\m{T}_{K}^{\mb{\nu}_{\mb{\rho}}} f(x) = \b{E}[f(Y) \overline{k_{x}(Y)} | \mb{\rho}] = \int_{X} K(x, y) f(y) \ \rd \mb{\nu}_{\mb{\rho}}(y), \quad f \in \c{R}(K), \\
    &\innpr{\m{T}_{K}^{\mb{\nu}_{\mb{\rho}}} f}{g}_{\c{R}(K)} = \int_{X} f(x) \overline{g(x)} \ \rd \mb{\nu}_{\mb{\rho}}(x) = \b{E}[f(Y) \overline{g(Y)} | \mb{\rho}], \quad f, g \in \c{R}(K), \\
    &\vertiii{\m{T}_{K}^{\mb{\nu}_{\mb{\rho}}}}_{1} = \int_{X} K(x, x) \ \rd \mb{\nu}_{\mb{\rho}}(x) = \b{E}[K(Y, Y) | \mb{\rho}].
\end{align*}

Analogous to classical kernel embedding \cite{smola2007hilbert}, a quantum covariance embedding (QCE) can be defined naturally as follows:
\begin{definition}[Quantum Covariance Embedding]
Let $X$ be LCH and $K:X \times X \to \b{F}$ be a $\c{C}_{0}$-kernel. We define the \textbf{Quantum Covariance Embedding} (QCE) as the map:
\begin{equation*}
    \mb{\nu} \in \s{P}(X, \c{Q}) \, \mapsto \m{T}_{K}^{\mb{\nu}} \in \c{B}_{\infty}^{+}(\c{R}(K) \otimes \c{Q}).
\end{equation*}
Furthermore, we define the \textbf{conditional QCE} 
given the density $\mb{\rho} \in \b{S}(\c{B}_{1}^{+}(\c{Q}))$ as the map:
\begin{align}\label{eq:contra:QCE:trace}
    \mb{\nu} \in \s{P}(X, \c{Q}) \, \mapsto \m{T}_{K}^{\mb{\nu}_{\mb{\rho}}} \in \c{B}_{1}^{+}(\c{R}(K)).
\end{align}
\end{definition}

It follows immediately that the conditional QCE is closed under convex combination of density operators: for $\mb{\rho}_{0}, \mb{\rho}_{1} \in \b{S}(\c{B}_{1}^{+}(\c{Q}))$ and $t \in [0, 1]$, we have:
\begin{equation}\label{eq:embed:cvx:combi}
    \m{T}_{K}^{\mb{\nu}_{\mb{\rho}_{t}}} = (1-t) \m{T}_{K}^{\mb{\nu}_{\mb{\rho}_{0}}} + t \m{T}_{K}^{\mb{\nu}_{\mb{\rho}_{1}}}, \quad \mb{\rho}_{t} := (1-t) \mb{\rho}_{0} + t \mb{\rho}_{1} \in \b{S}(\c{B}_{1}^{+}(\c{Q})).
\end{equation}
We also remark that when the PVM $\mb{\nu}^{\m{O}} \in \s{P}(X, \c{Q})$ is associated with a compact observable $\m{O} = \sum_{k=1}^{\infty} \lambda_{k} \mb{\Pi}_{k} \in \c{B}_{0}(\c{Q})$, then both QCEs reduce to
\begin{align*}
    \m{T}_{K}^{\mb{\nu}^{\m{O}}} = \sum_{k=1}^{\infty} (k_{\lambda_{k}} k_{\lambda_{k}}^{*}) \otimes \m{\Pi}_{k}, \quad \m{T}_{K}^{\mb{\nu}^{\m{O}}_{\mb{\rho}}} = \sum_{k=1}^{\infty} k_{\lambda_{k}} k_{\lambda_{k}}^{*} \tr{\mb{\rho} \m{\Pi}_{k}},
\end{align*}
where the convergence is with respect to the SOT due to Vigier's theorem. 
As a guiding example, we provide below a matrix represenation of QCEs for a finite-dimensional case:

\begin{example}[Leveled Quanta]
Consider a qudit $\c{Q} = \b{F}^{q}$ and an observable $\m{O} = \sum_{k=1}^{l} \lambda_{k} \m{\Pi}_{k} \in \b{F}^{q \times q}$,
where $X = \{\lambda_k\}_{k=1}^l \subset \b{R}$ are distinct eigenvalues ($l < q$ in case of multiplicities) and $\m{\Pi}_{k}$ are the associated eigenprojections. The corresponding PVM $\mb{\nu} = \sum_{k=1}^{l} \delta_{\lambda_{k}} \m{\Pi}_{k}$ can be viewed as an $l \times q \times q$ tensor with $\mb{\nu}_{kij} = [\m{\Pi}_{k}]_{ij}$.
Let $K(x, y) = \ds{1} (x = y)$ be the 0--1 kernel so that the evaluation map $f \in \c{R}(K) \cong \m{f} \in \b{F}^{l}$ is isometrically isomorphic, where $\m{f}_{k} := f(\lambda_{k})$ \cite{paulsen2016introduction}. This induces an isomorphism for the tensor product space, equipped with the Frobenius inner product: $f \otimes \phi \in \c{R}(K) \otimes \c{Q} \cong \phi \overline{\m{f}}^{*} \in \b{F}^{q \times l}$.
Under this identification, the QCE acts as a column-wise projection on $\b{F}^{q \times l}$:
\begin{equation*}
    \m{T}_{K}^{\mb{\nu}^{\m{O}}} : [\phi_{1} \vert \dots \vert \phi_{l}] \in \b{F}^{q \times l} \, \mapsto \, [\m{\Pi}_{1} \phi_{1} \vert \dots \vert \m{\Pi}_{l} \phi_{l}] \in \b{F}^{q \times l}.
\end{equation*}
To see this, note that for any $f, g \in \c{R}(K)$ and $\phi, \psi \in \b{F}^{q}$,
\begin{align*}
    \innpr{\m{T}_{K}^{\mb{\nu}^{\m{O}}} (f \otimes \phi)}{g \otimes \psi}_{\c{R}(K) \otimes \c{Q}} 
    &= \sum_{k=1}^{l} \m{f}_{k} \overline{\m{g}_{k}} \innpr{\m{\Pi}_{k} \phi}{\psi}_{\b{F}^{m}} = \int_{X} f(x) \overline{g(x)} \ \rd \mb{\nu}_{\phi, \psi}(x).
\end{align*}
Similarly, one can show that the conditional QCE behaves as a diagonal matrix:
\begin{equation*}
    \m{T}_{K}^{\mb{\nu}^{\m{O}}_{\mb{\rho}}} = \Diag(\tr{\mb{\rho} \m{\Pi}_{1}}, \dots, \tr{\mb{\rho} \m{\Pi}_{l}}) \in \b{F}^{l \times l} \cong \c{B}_{\infty}(\c{R}(K)), \quad \mb{\rho} \in \b{S}(\b{F}_{+}^{q \times q}),
\end{equation*}
of which the trace norm satisfies:
\begin{equation*}
    \vertiii{\m{T}_{K}^{\mb{\nu}^{\m{O}}_{\mb{\rho}}}}_{1} = \sum_{k=1}^l \tr{\mb{\rho} \m{\Pi}_{k}} = \tr{\mb{\rho}} = 1 < q = \sum_{k=1}^l \tr{ \m{\Pi}_{k}} = \vertiii{ \m{T}_{K}^{\mb{\nu}^{\m{O}}} }_{1},
\end{equation*}
consistent with \cref{thm:quant:cov:scalar}. The operator norms satisfy:
\begin{equation*}
    \vertiii{\m{T}_{K}^{\mb{\nu}^{\m{O}}_{\mb{\rho}}}}_{\infty} = \max_{1 \le k \le l} \tr{\mb{\rho} \m{\Pi}_{k}} \le 1 = \max_{\phi_{k} \in \b{F}^{q} \backslash \{0\}} \frac{\sum_{k=1}^{l} \| \m{\Pi}_{k} \phi_{k} \|^{2}}{\sum_{k=1}^{l} \| \phi_{k} \|^{2}}  = \vertiii{\m{T}_{K}^{\mb{\nu}^{\m{O}}}}_{\infty},
\end{equation*}
where equality holds if and only if $\mb{\rho}$ lies in the convex hull of the pure eigenstates of $\m{O}$.
\end{example}
For a general kernel other than the 0--1 kernel, we refer to \cref{ex:disc:kern}.

\subsection{Characteristic Kernel}
Unlike QST where the goal is to estimate $\mb{\rho}$, the focus of this subsection lies in identifying the measurement processes $\mb{\nu}$ (or their marginalizations $\mb{\nu}_{\mb{\rho}}$).

\begin{definition}[Characteristicity]
Let $X$ be an LCH space and $\c{Q}$ be a QHS. A $\c{C}_{0}$-kernel $K$ is called \textbf{characteristic} if the associated QCE is injective. That is, for any two POVMs $\mb{\mu}, \mb{\nu} \in \s{P}(X, \c{Q})$, we have $\m{T}_{K}^{\mb{\mu}} = \m{T}_{K}^{\mb{\nu}}$ if and only if $\mb{\mu} = \mb{\nu}$.
\end{definition}

The following theorem establishes that characteristicity in the quantum setting is equivalent to the usual definition for classical probability measures \cite{fukumizu2004dimensionality, sriperumbudur2010hilbert}.

\begin{theorem}\label{thm:char:ker}
Let $X$ be an LCH space and $K:X \times X \to \b{F}$ be a $\c{C}_{0}$-kernel. 
\begin{enumerate}[leftmargin = *]
\item Given POVMs $\mb{\mu}, \mb{\nu} \in \s{P}(X, \c{Q})$, their QCEs are identical if and only if their conditional QCEs are identical for every density operator $\mb{\rho} \in \b{S}(\c{B}_{1}^{+}(\c{Q}))$, i.e.,
\begin{align*}
    \m{T}_{K}^{\mb{\mu}} = \m{T}_{K}^{\mb{\nu}} \quad \iff \quad \m{T}_{K}^{\mb{\mu}_{\mb{\rho}}} = \m{T}_{K}^{\mb{\nu}_{\mb{\rho}}}, \quad \forall \mb{\rho} \in \b{S}(\c{B}_{1}^{+}(\c{Q})).
\end{align*}
\item $K$ is characteristic for $\s{P}(X, \c{Q})$ if and only if it is characteristic for $\s{P}(X)$.  
\end{enumerate}
\end{theorem}

\cref{thm:char:ker,thm:quant:cov:scalar} naturally motivates us to metrize the space $\s{P}(X, \c{Q})$ via the discrepancy of QCEs in the operator norm:
\begin{definition}[Quantum Maximum Discrepancy]
Let $X$ be LCH, $\c{Q}$ be a QHS, and $K:X \times X \to \b{F}$ be a $\c{C}_{0}$-kernel. For any POVMs $\mb{\mu}, \mb{\nu} \in \s{P}(X, \c{Q})$, we define their \textbf{Quantum Maximum Discrepancy} (QMD) as:
\begin{equation*}
    \QMD_{K}(\mb{\mu}, \mb{\nu}) := \vertiii{\m{T}_{K}^{\mb{\mu}} - \m{T}_{K}^{\mb{\nu}}}_{\infty}.
\end{equation*}
For a fixed density operator $\mb{\rho} \in \b{S}(\c{B}_{1}^{+}(\c{Q}))$, we define the \textbf{conditional QMD} as:
\begin{equation*}
    \QMD_{K}^{\mb{\rho}}(\mb{\mu}, \mb{\nu}) := \vertiii{\m{T}_{K}^{\mb{\mu}_{\mb{\rho}}} - \m{T}_{K}^{\mb{\nu}_{\mb{\rho}}}}_{\infty}.
\end{equation*}
\end{definition}

It follows immediately from \cref{thm:char:ker} that the QMD induces a valid metric on $\s{P}(X, \c{Q})$ if and only if $K$ is characteristic. However, even in this case, the conditional QMD for a fixed $\mb{\rho}$ may not be a metric in general, as $\mb{\mu}_{\mb{\rho}} = \mb{\nu}_{\mb{\rho}}$ does not imply $\mb{\mu} = \mb{\nu}$.

\begin{proposition}\label{prop:ineq:QCE}
Let $X$ be LCH and $\c{Q}$ be a QHS. Then, for any $\c{C}_{0}$-kernel $K$ and POVMs $\mb{\mu}, \mb{\nu} \in \s{P}(X, \c{Q})$, the following inequality holds:
\begin{equation*}
    \QMD_{K}(\mb{\mu}, \mb{\nu}) \ge \sup_{\mb{\rho} \in \b{S}(\c{B}_{1}^{+}(\c{Q}))} \QMD_{K}^{\mb{\rho}}(\mb{\mu}, \mb{\nu}).
\end{equation*}
\end{proposition}

Crucially, while the full QMD provides a universal upper bound, we focus on the conditional QMD for two primary reasons. First, it carries far greater statistical significance: $\mb{\mu}_{\mb{\rho}}$ and $\mb{\nu}_{\mb{\rho}}$ represent the actual physical probability distributions governing the observed macroscopic data given an input state $\mb{\rho}$. Second, from a technical standpoint, the full QMD presents analytical challenges as it operates over the larger tensor product space.
Furthermore, the conditional QMD provides a natural framework to address a practical question: \emph{given two physical apparatuses represented by POVMs $\mb{\mu}$ and $\mb{\nu}$, which input state $\mb{\rho}$ maximizes the discrepancy between their induced probability measures?} Because the map $\mb{\rho} \mapsto \m{T}_{K}^{\mb{\mu}_{\mb{\rho}}} - \m{T}_{K}^{\mb{\nu}_{\mb{\rho}}}$ is affine (see \eqref{eq:embed:cvx:combi}), its composition with the operator norm yields a convex function. Consequently, the supremum on the right-hand side of \cref{prop:ineq:QCE} is necessarily achieved at an extreme point of the convex state space $\b{S}(\c{B}_{1}^{+}(\c{Q}))$, which corresponds precisely to a pure state \cite{holevo2011probabilistic}. Thus, there always exists a \textbf{maximally distinguishing} pure state that maximizes the conditional discrepancy, i.e.,
\begin{align*}
    \exists \, \phi \in \c{Q} \quad \text{such that} \quad \|\phi\|_{\c{Q}}=1, \quad \phi \phi^{*} \in \argmax \{\QMD_{K}^{\mb{\rho}}(\mb{\mu}, \mb{\nu}) : \mb{\rho} \in \b{S}(\c{B}_{1}^{+}(\c{Q})) \}.
\end{align*}
\begin{remark}\label{rmk:QMD:schatten}
For later applications in quantum state tomography, we remark that the conditional QMD can be generalized to any $s$-Schatten norm with $s \in [1, \infty)$ instead of the operator norm, since $\m{T}_{K}^{\mb{\mu}_{\mb{\rho}}}, \m{T}_{K}^{\mb{\nu}_{\mb{\rho}}} \in \c{B}_{1}^{+}(\c{R}(K))$ due to \eqref{eq:contra:QCE:trace}. Hence, if another norm is more convenient for statistical inference, one may substitute the operator norm accordingly, e.g., adopting the Hilbert-Schmidt norm with $s=2$ in \cref{thm:loss:max:eigen}. Indeed, because any norm is a convex function, the guarantee that there always exists a maximally distinguishing pure state remains valid under this generalization.
\end{remark}

\subsubsection{Pauli Measurements for Qubits}
A fundamental orthogonal basis for the space of Hermitian matrices on a single qubit $\c{Q} = \b{C}^{2}$ is formed by the identity matrix $\m{I}_{2}$ and the three Pauli matrices \cite{nielsen2010quantum}:
\begin{equation*}
    \setlength{\arraycolsep}{0.3em} 
    \mb{\sigma}_{x} = 
    \begin{pmatrix}
    0 & 1 \\
    1 & 0 \vphantom{-1}
    \end{pmatrix}, \quad
    \mb{\sigma}_{y} = 
    \begin{pmatrix}
    0 & -\imath \\
    \imath & 0
    \end{pmatrix}, \quad
    \mb{\sigma}_{z} = 
    \begin{pmatrix}
    1 & 0 \\
    0 & -1
    \end{pmatrix} \in \b{C}^{2 \times 2}.
\end{equation*}
These matrices are traceless, unitary, and involutory ($\mb{\sigma}_{x}^{2}=\mb{\sigma}_{y}^{2}= \mb{\sigma}_{z}^{2}= \m{I}_{2}$), with spectra $X = \{\pm 1 \}$. 
Consequently, any unit-trace Hermitian matrix admits a unique decomposition known as the \emph{Bloch sphere representation}:
\begin{align}\label{eq:Bloch:repre}
    \mb{\rho} = \frac{1}{2} (\m{I}_{2} + \m{a} \cdot \vec{\mb{\sigma}})
    = \frac{1}{2} (\m{I}_{2} + a_{x} \mb{\sigma}_{x} + a_{y} \mb{\sigma}_{y} + a_{z} \mb{\sigma}_{z}), \quad \m{a} \in \b{R}^{3}.
\end{align}
This establishes a correspondence between the unit-trace Hermitian matrix and a real vector $\m{a} = (a_{x}, a_{y}, a_{z})^{\top} \in \b{R}^{3}$, called the \textbf{Bloch vector}. Since the eigenvalues of $\mb{\rho}$ are $\lambda_{\pm} = (1 \pm \|\m{a}\|_{\b{R}^{3}})/2$, the n.n.d. condition for the density matrix $\mb{\rho} \in \b{S}(\b{C}_{+}^{2 \times 2})$ restricts the Bloch vector $\m{a}$ to the unit ball $\b{B}^{3}$ \cite{sakurai2017modern}. Furthermore, it holds that:
\begin{align*}
    (\m{a} \cdot \vec{\mb{\sigma}}) (\m{b} \cdot \vec{\mb{\sigma}}) = (\m{a} \cdot \m{b}) \m{I}_{2} + \imath (\m{a} \times \m{b}) \cdot \vec{\mb{\sigma}}, \quad \m{a}, \m{b} \in \b{R}^{3},
\end{align*}
thus $\mb{\rho} \in \b{S}(\b{C}_{+}^{2 \times 2})$ represents a pure state if and only if $1= \tr{\mb{\rho}^{2}} = (1 + \|\m{a}\|_{\b{R}^{3}}^{2})/2$, i.e., $\m{a}$ lies on the surface $\b{S}^{2}$. In short, the norm of the Bloch vector $\m{a}$  for any density matrix satisfies $\|\m{a}\|_{\b{R}^{3}} \le 1$ with equality if and only if it represents a pure state. In particular, the maximally mixed state $\mb{\rho} = \m{I}_{2}/2$ corresponds to the origin $\m{a} = \m{0}$.

We now demonstrate a decoupling phenomenon specific to single qubits ($\c{Q} = \b{C}^{2}$). While the absolute magnitude of the conditional QMD is linked to the chosen kernel geometry, the maximally distinguishing pure states are intrinsic to the Pauli measurements and independent of the kernel choice. This invariance arises due to the binary nature of their spectra. Because the measurement outcomes are strictly two-dimensional, the optimal input states that maximize the statistical discrepancy are governed solely by the algebraic difference of the Pauli observables themselves:

\begin{theorem}\label{thm:pauli:QMD}
Let $\c{Q}= \b{C}^{2}$, $X = \{\pm 1 \}$, and $K:X \times X \to \b{C}$ be a kernel with associated Gram matrix $\m{K} \in \b{C}_{+}^{2 \times 2}$. For any distinct pair of Pauli indices $i \ne j \in \{x, y, z\}$, the conditional QMD between the associated PVMs $\mb{\nu}_{i}$ and $\mb{\nu}_{j}$ is bounded by:
\begin{align*}
    \QMD_{K}^{\mb{\rho}}(\mb{\nu}_{i}, \mb{\nu}_{j}) = \frac{\vertiii{\m{K}^{1/2} \mb{\sigma}_{z} \m{K}^{1/2}}_{\infty}}{2} |\tr{\mb{\rho} (\mb{\sigma}_{i} - \mb{\sigma}_{j})}| \le \frac{\vertiii{\m{K}^{1/2} \mb{\sigma}_{z} \m{K}^{1/2}}_{\infty}}{\sqrt{2}}.
\end{align*}
Equality holds (i.e., the state is maximally distinguishing) if and only if $\mb{\rho}$ is either of the two pure states:
\begin{align*}
    \mb{\rho}_{\pm} = \frac{1}{2}\left(\m{I}_{2} \pm \frac{\mb{\sigma}_{i} - \mb{\sigma}_{j}}{\sqrt{2}}\right)\in \b{S}(\b{C}_{+}^{2 \times 2}).
\end{align*}
\end{theorem}

\begin{remark}[Indistinguishability and the Pauli-Z matrix]
\cref{thm:pauli:QMD} implies that the probability measures induced by Pauli matrices $\mb{\sigma}_{i}$ and $\mb{\sigma}_{j}$ are identical if and only if the Bloch vector $\m{a} \in \b{B}^{3}$ of the input state is orthogonal to the difference vector $\m{e}_{i} - \m{e}_{j} \in \b{R}^{3}$:
\begin{align*}
    \tr{\mb{\rho} (\mb{\sigma}_{i} - \mb{\sigma}_{j})} = \innpr{\m{a}}{\m{e}_{i} - \m{e}_{j}}_{\b{R}^{3}}, \quad \mb{\rho} = \frac{1}{2} (\m{I}_{2} + \m{a} \cdot \vec{\mb{\sigma}}).
\end{align*}
Furthermore, the appearance of the Pauli-Z matrix arises naturally from the binary nature of the spectra. In this context, $\mb{\sigma}_{z} = \Diag(1, -1)$ does not act on the quantum state, but rather computes the contrast $|f(+1)|^{2} - |f(-1)|^{2}$, thus capturing how effectively the kernel discriminates between the two possible measurement outcomes.
\end{remark}

\subsection{Noise Quantification of Quantum Measurements}\label{ssec:uq:povm}
As discussed in \cref{ssec:quant:msr}, physical measurement devices frequently deviate from idealized PVMs due to finite detector efficiency, thermal fluctuations, or environmental coupling. Consequently, the POVM framework is necessary to model the statistical corruption inherent in realistic apparatuses \cite{busch2016quantum, breuer2002theory}. To illustrate the utility of the Quantum Maximum Discrepancy (QMD), we demonstrate how this metric quantifies classical bit-flip noise in qubit systems. An extension of this analysis to discrete modular errors in general qudit systems $\c{Q} = \b{C}^{q}$ is provided in \cref{ssec:qudit:povm}.

\subsubsection{Bit-flip Noise in Qubits}\label{ssec:bit:flip}
Let $\m{O} \in \b{C}^{2 \times 2}$ be an observable in a qubit system $\c{Q} = \b{C}^{2}$. Without loss of generality, we assume the spectrum is $X = \sigma(\m{O}) = \{\pm 1\}$ (e.g., a Pauli matrix), or equivalently, $\m{O} = \m{u} \cdot \vec{\mb{\sigma}}$ for some unit vector $\m{u} \in \b{S}^{2}$. The eigenprojections associated with the eigenvalues $\pm 1$ are given by $\m{\Pi}_{\pm} = (\m{I}_{2} \pm \m{u} \cdot \vec{\mb{\sigma}})/2$.
Consider a scenario where the measurement outcome is subject to a symmetric bit-flip error, a standard noise model in quantum information \cite{nielsen2010quantum}. Let $Y$ denote the ideal measurement outcome for an input state $\mb{\rho} \in \b{S}(\b{C}^{2 \times 2})$, and let $\varepsilon$ be an independent multiplicative noise variable. The observed outcome is given by:
\begin{align*} 
    Y_{\eta} = Y \cdot \varepsilon, \quad \text{where } \quad 
    Y = \begin{cases} +1 & \text{w.p. } \tr{\mb{\rho} \m{\Pi}_{+}} \\ -1 & \text{w.p. } \tr{\mb{\rho} \m{\Pi}_{-}} \end{cases}, \quad 
    \varepsilon = \begin{cases} +1 & \text{w.p. } 1-\eta \\ -1 & \text{w.p. } \eta \end{cases}. 
\end{align*}
Here, $\eta \in [0, 1/2]$ represents the error rate (or flip probability). We restrict our attention to this regime because an error rate $\eta > 1/2$ merely corresponds to re-labeling the outcomes. The boundary case $\eta = 1/2$ represents a \emph{completely} uninformative measurement: the observed outcome $Y_{1/2}$ follows a uniform distribution over $\{\pm 1\}$ regardless of the input state $\mb{\rho}$, rendering any statistical inference on the underlying state fundamentally impossible.
The POVM governing the noisy outcome $Y_{\eta}$ is constructed by mixing the original projections:
\begin{align*}
    \mb{\nu}^{\m{O}, \eta} = \delta_{+1} [(1-\eta) \m{\Pi}_{+} + \eta \m{\Pi}_{-}] + \delta_{-1} [\eta \m{\Pi}_{+} + (1-\eta) \m{\Pi}_{-}].
\end{align*}
Crucially, it is \emph{not} a PVM as it violates the idempotence property. Furthermore, this noise process dampens the effective observable, defined as the first moment of the POVM:
\begin{align*}
    \m{O}_{\eta} := \int_{X} x \rd \mb{\nu}^{\m{O}, \eta} (x) 
    = [(1-\eta) \m{\Pi}_{+} + \eta \m{\Pi}_{-}] - [\eta \m{\Pi}_{+} + (1-\eta) \m{\Pi}_{-}] 
    = (1-2\eta) \m{O}.
\end{align*}
While $\m{O}_{\eta}$ shares the same eigenvectors as $\m{O}$, its spectrum is attenuated by a factor of $|1-2\eta|$. Hence, the POVM $\mb{\nu}^{\m{O}, \eta}$ is distinct from the spectral measure of $\m{O}_{\eta}$ (which is a PVM).

\begin{proposition}\label{prop:QMD:flip}
Given an observable $\m{O} \in \b{C}^{2 \times 2}$ with the spectrum $X = \sigma(\m{O}) = \{\pm 1\}$, the conditional QMD between the ideal PVM $\mb{\nu}^{\m{O}}$ and the noisy POVM $\mb{\nu}^{\m{O}, \eta}$ with error rate $\eta \in [0, 1]$ is bounded by
\begin{align*}
    \QMD_{K}^{\mb{\rho}}(\mb{\nu}^{\m{O}}, \mb{\nu}^{\m{O}, \eta}) = \eta \vertiii{\m{K}^{1/2} \mb{\sigma}_{z} \m{K}^{1/2}}_{\infty} |\tr{\mb{\rho} \m{O}}| \le \eta \vertiii{\m{K}^{1/2} \mb{\sigma}_{z} \m{K}^{1/2}}_{\infty}.
\end{align*}
Moreover, the maximally distinguishing states correspond to the eigenprojections $\m{\Pi}_{\pm}$ of $\m{O}$.
\end{proposition}

The theorem below unifies \cref{thm:pauli:QMD,prop:QMD:flip}:

\begin{theorem}\label{thm:two:bit:flip}
Let $\m{O}, \tilde{\m{O}} \in \b{C}^{2 \times 2}$ be two observables with spectrum $X = \{\pm 1\}$, i.e., $\m{O} = \m{u} \cdot \vec{\mb{\sigma}}$ and $\tilde{\m{O}} = \tilde{\m{u}} \cdot \vec{\mb{\sigma}}$ for some $\m{u}, \tilde{\m{u}} \in \b{S}^{2}$. Let $\mb{\nu}_{\eta}$ and $\mb{\tilde{\nu}}_{\tilde{\eta}}$ denote the corresponding noisy POVMs subject to error rates $\eta, \tilde{\eta} \in [0, 1]$, respectively. 
Define $\m{v} := (1-2\eta) \m{u} - (1-2 \tilde{\eta}) \tilde{\m{u}} \in \b{R}^{3}$. The conditional QMD between these measurements is given by:
\begin{align}\label{eq:QMD:flip:two}
    \QMD_{K}^{\mb{\rho}}(\mb{\nu}_{\eta}, \mb{\tilde{\nu}}_{\tilde{\eta}}) 
    &= \vertiii{\m{K}^{1/2} \mb{\sigma}_{z} \m{K}^{1/2}}_{\infty} \left|\tr {\mb{\rho} \frac{(1-2 \eta) \m{O} - (1-2 \tilde{\eta}) \tilde{\m{O}}}{2}} \right| \\
    &\le \vertiii{\m{K}^{1/2} \mb{\sigma}_{z} \m{K}^{1/2}}_{\infty} \frac{\|\m{v}\|_{\b{R}^{3}}}{2}, \nonumber 
\end{align}
provided $\m{v} \neq \m{0}$. In this case, the maximally distinguishing pure states are given by:
\begin{align*}
    \mb{\rho}_{\pm} = \frac{1}{2} \left(\m{I}_{2} \pm \frac{\m{v}}{\|\m{v}\|_{\b{R}^{3}}} \cdot \vec{\mb{\sigma}} \right).
\end{align*}
\end{theorem}

Note that $\m{v} = \m{0}$ if and only if two noisy measurement processes are statistically identical ($\mb{\nu}_{\eta} = \mb{\tilde{\nu}}_{\tilde{\eta}}$). Consequently, the QMD is identically zero and the bound vanishes, meaning no quantum state is capable of distinguishing between the two apparatuses. 

\begin{remark}[Discrepancy vs. Uncertainty]
Heisenberg's uncertainty principle characterizes the physical consequences of measurement ordering (non-commutativity), quantified by the Lie bracket $[\m{O}, \tilde{\m{O}}] := \m{O} \tilde{\m{O}} - \tilde{\m{O}} \m{O} = 2 \imath (\m{u} \times \tilde{\m{u}}) \cdot \vec{\mb{\sigma}}$ \cite{talagrand2022quantum}. In contrast, statistical inference for QST rather relies on distinguishability: the capacity to differentiate data-generating processes, which our QMD explicitly captures. To see that these are profoundly different notions, consider the ideal setting without apparatus error ($\eta = \tilde{\eta} = 0$). Let $\theta = \angle(\m{u}, \tilde{\m{u}})$ denote the angle between the unit Bloch vectors. The statistical discrepancy and the quantum uncertainty metrics diverge strictly as a function of this angle:
\begin{align*}
    \text{Conditional QMD Bound} &\propto \|\m{u} - \tilde{\m{u}}\|_{\b{R}^{3}} = 2 |\sin(\theta/2)|, \\
    \text{Uncertainty Bound} &= \sup_{\mb{\rho} \in \b{S}(\b{C}_{+}^{2 \times 2})} \left|\tr{\mb{\rho} [\m{O}, \tilde{\m{O}}]} \right| = 2 \|\m{u} \times \tilde{\m{u}} \|_{\b{R}^{3}} = 2 |\sin(\theta)|.
\end{align*}
While both quantities vanish for identical observables ($\theta=0$), their behaviors diverge elsewhere. Uncertainty peaks at $\theta = \pi/2$ (e.g., Pauli-X vs. Pauli-Z), whereas the QMD peaks for anti-parallel observables ($\theta = \pi$). Because anti-parallel observables strictly commute ($[\m{O}, \tilde{\m{O}}] = \m{0}$) yet yield perfectly disjoint measurement outcomes, this demonstrates that statistical discrepancy is maximized precisely where quantum uncertainty vanishes.
\end{remark}

\section{Quantum State Tomography}\label{sec:QKT}
Building upon the kernel embedding framework in \cref{sec:RKHS}, we now address the statistical inverse problem of estimating an unknown density operator from measurements performed on an ensemble of identically prepared quantum systems - a process known as \textbf{Quantum State Tomography} (QST).
For the remainder of this work, we restrict our attention to finite-dimensional QHSs $\c{Q} = \b{F}^{q}$ with $q \ge 2$. To streamline our statistical derivations, we adopt the following finite-dimensional conventions. The space of linear operators on $\c{Q}$ (which coincides across all Schatten classes $\c{B}_{p}$ in finite dimensions) is simply the space of $q \times q$ matrices, denoted $\b{F}^{q \times q}$. We denote the real vector space of self-adjoint matrices (symmetric if $\b{F}=\b{R}$, Hermitian if $\b{F}=\b{C}$) by $\b{F}_{sa}^{q \times q}$. We further denote its traceless linear subspace by $\b{F}_{sa, 0}^{q \times q}$, and the affine hyperplane of unit-trace matrices by $\b{F}_{sa, 1}^{q \times q}$. The convex cone of n.n.d. matrices is denoted $\b{F}_{+}^{q \times q} \subsetneq \b{F}_{sa}^{q \times q}$. Consequently, the quantum state space is precisely $\b{S}(\b{F}_{+}^{q \times q}) = \b{F}_{+}^{q \times q} \cap \b{F}_{sa, 1}^{q \times q}$.
Finally, we employ \emph{fraktur} fonts (e.g., $\f{D}, \f{H}$) to denote operators acting on matrix spaces, commonly called \textbf{superoperators} \cite{scott2006tight}.

\subsection{Measurement Design}\label{ssec:identi:dsgn}
Our goal is to estimate the unknown density matrix. Since a single observable is mathematically insufficient for full reconstruction, we consider a design consisting of a collection of $n$ self-adjoint observables $\s{O} = \{ \m{O}_{i} \}_{i=1}^n \subset \b{F}_{sa}^{q \times q}$.
For each observable $\m{O}_{i}$, let its spectral decomposition and associated PVM be given by:
\begin{equation*}
    \m{O}_{i} = \sum_{k=1}^{q_{i}} \lambda_{ik} \m{\Pi}_{ik}, \quad \mb{\nu}_{i} =  \sum_{k=1}^{q_{i}} \delta_{\lambda_{ik}} \m{\Pi}_{ik},
\end{equation*}
where $\{\lambda_{ik}\}_{k=1}^{q_i}$ are distinct eigenvalues and $\m{\Pi}_{ik}$ are the corresponding eigenprojections with rank (multiplicity) $m_{ik} = \tr{\m{\Pi}_{ik}} \in \b{N}$. Throughout, we assume $q_{i} \ge 2$ (otherwise $\m{\Pi}_{i1} = \m{I}_{q}$, yielding no information). We note the following orthogonality and trace properties:
\begin{align*}
    \tr{\m{\Pi}_{ik} \m{\Pi}_{ik'}} = m_{ik} \ds{1} (k=k'), \quad \sum_{k=1}^{q_{i}} \sum_{k'=1}^{q_{i'}} \tr{\m{\Pi}_{ik} \m{\Pi}_{i'k'}} = q.
\end{align*}
To handle eigenvalue multiplicities, we define the index map $k_i : \{1, \dots, q\} \to \{1, \dots, q_i\}$, which maps the global flat index $\ell$ to the distinct eigenspace index $k$:
\begin{align*}
    k = k_{i}(\ell) \in \{1, \cdots, q_{i} \} \quad \Longleftrightarrow \quad \sum_{k' =1}^{k-1} m_{ik'} < \ell \le \sum_{k' =1}^{k} m_{ik'}, \quad \ell \in \{1, \cdots, q \}.
\end{align*}

The standard tomography protocol proceeds as follows. For each $\m{O}_{i}$, an ensemble of $r$ independent systems is prepared in the state $\mb{\rho}$. The measurement outcomes $\{Y_{ij}\}_{j=1}^{r}$ are i.i.d. random variables governed by the Born rule:
\begin{equation*}
    \b{P}(Y_{ij} = \lambda_{ik}\mid\mb\rho) = \mb{\nu}_{i}(\{\lambda_{ik}\}) = \tr{\mb{\rho} \m{\Pi}_{ik}}, \quad 1 \le i \le n, \, 1 \le j \le r, \, 1 \le k \le q_{i}.
\end{equation*}
Since the convex cone of density matrices $\b{S}(\b{F}_{+}^{q \times q})$ is not a linear space, we consider a linear operator by extending the domain:
\begin{align*}
    d_{ik} : \b{F}_{sa}^{q \times q} \to \b{R}, \quad \m{S} \mapsto d_{ik}[\m{S}] := \tr{\m{S} \m{\Pi}_{ik}}, \quad 1 \le i \le n, \, 1 \le k \le q_{i}.
\end{align*}
Since $\b{F}_{sa}^{q \times q}$ under the Frobenius inner product is a real Hilbert space, the functional $d_{ik}$ can be identified with the projection $\m{\Pi}_{ik}$ by the Riesz representation theorem.
Because the maximally mixed state $\m{I}_{q}/q \in \b{S}(\b{F}_{+}^{q \times q})$ produces entirely uninformative measurement statistics ($m_{ik}^{-1} d_{ik} [\m{I}_{q}/q] \equiv 1/q$), it acts as a non-directional baseline. In other words, $\mb{\rho} \in \b{S}(\b{F}_{+}^{q \times q})$ is identifiable if and only if the traceless component $\mb{\rho} - \m{I}_{q}/q \in \b{F}_{sa, 0}^{q \times q}$ is identifiable via variations in measurement statistics. This naturally motivates the \textbf{trace-centering} decomposition to isolate the statistically identifiable signal:
\begin{equation}\label{eq:tr:center}
    \m{S} = \tr{\m{S}} \frac{\m{I}_{q}}{q} \oplus \left[ \m{S} - \tr{\m{S}} \frac{\m{I}_{q}}{q} \right] \quad \in \quad \b{F}_{sa}^{q \times q} = \spann\{\m{I}_{q}\} \oplus \b{F}_{sa, 0}^{q \times q}.
\end{equation}

\begin{definition}[Complete Design] \label{def:complete_design}
A collection of observables $\s{O} = \{\m{O}_{i} \in \b{F}_{sa}^{q \times q}: i = 1, \cdots, n\}$ is said to be a \emph{complete design} \cite{busch1991informationally}, if for any traceless self-adjoint $\mb{\Delta} \in \b{F}_{sa, 0}^{q \times q}$:
\begin{equation*}
    d_{ik}[\mb{\Delta}] = 0 \, \text{ for all } \, 1 \le i \le n, \, 1 \le k \le q_{i} \quad \iff \quad \mb{\Delta} = \m{0}.
\end{equation*}
\end{definition}

\emph{Informational completeness} is the physical equivalent of statistical identifiability:
\begin{proposition}[Identifiability]\label{prop:identifi}
Let $\mb{\rho}, \tilde{\mb{\rho}} \in \b{S}(\b{F}_{+}^{q \times q})$ be two density matrices. Then the following conditions are equivalent:
\begin{enumerate}
\item $\s{O} = \{\m{O}_{i} \in \b{F}_{sa}^{q \times q}: i = 1, \cdots, n\}$ is a complete design.
\item $d_{ik}[\mb{\rho}] = d_{ik}[\tilde{\mb{\rho}}]$ for all $i = 1, \cdots, n$, and $k = 1, \cdots, q_{i}$, implies $\mb{\rho} = \tilde{\mb{\rho}}$.
\end{enumerate}
\end{proposition}

Analogous to the design matrix in linear regression, we define the \emph{quantum design tensor} $\f{D}: \b{F}_{sa}^{q \times q} \to \b{R}^{n \times q}$ as the linear map: for a self-adjoint matrix $\m{S} \in \b{F}_{sa}^{q \times q}$,
\begin{equation}\label{eq:quant:des:op}
    \f{D}[\m{S}] = 
    \begin{pmatrix}
        \m{d}_{1}[\m{S}] \vert \m{d}_{2}[\m{S}] \vert   \cdots \vert \m{d}_{n}[\m{S}]
    \end{pmatrix}^{\top}, \quad \m{d}_{i}[\m{S}] := \left[d_{i k_{i}(\ell)}[\m{S}]/m_{i k_{i}(\ell)}  \right]_{1 \le \ell \le q}, 
\end{equation}
where each vector $\m{d}_{i}[\m{S}] \in \b{R}^{q}$ is formed by taking the weighted expectation $d_{ik}[\m{S}]/m_{ik}$ and repeating it exactly $m_{ik}$ times for each eigenspace index $k = 1, \dots, q_{i}$. We now characterize the complete design:

\begin{proposition}\label{prop:sphe:design}
Let $\c{Q} = \b{F}^{q}$. The adjoint of the quantum design tensor is given by:
\begin{equation*}
    \f{D}^{*}: \b{R}^{n \times q} \to \b{F}_{sa}^{q \times q}, \, \m{A} \mapsto \sum_{i=1}^{n} \sum_{k=1}^{q_{i}} \bar{a}_{ik} \m{\Pi}_{ik}, \quad \bar{a}_{ik} = \frac{1}{m_{ik}} \sum_{\{\ell : k_{i}(\ell) = k\}} a_{i\ell}, 
\end{equation*}
i.e., $\bar{a}_{ik}$ is the average of the $m_{ik}$ entries in the $i$-th row of $\m{A}$ corresponding to the eigenspace $\m{\Pi}_{ik}$. Additionally, the following are equivalent:
\begin{enumerate}
\item $\s{O} = \{\m{O}_{i} \in \b{F}_{sa}^{q \times q}: i = 1, \cdots, n\}$ is a complete design.
\item $\spann \{\m{\Pi}_{ik} : i = 1, \cdots, n, \, k = 1, \cdots, q_{i} \} = \b{F}_{sa}^{q \times q}$.
\item The quantum design tensor $\f{D}: \b{F}_{sa}^{q \times q} \to \b{R}^{n \times q}$ is injective.
\end{enumerate}
\end{proposition}

\cref{prop:identifi,prop:sphe:design} imply that QST is ill-posed if $\dim \c{Q} = \infty$ without structural restrictions (e.g., low rank), justifying our finite-dimensional focus. 

\begin{remark}
Given a fixed design $\s{O}$, the unidentifiable directions are given by
\begin{align}\label{eq:null:mat}
    \ker(\f{D}) = \spann \{\m{\Pi}_{ik} : i = 1, \cdots, n, \, k = 1, \cdots, q_{i} \}^{\perp}  \subset \b{F}_{sa, 0}^{q \times q},
\end{align}
For instance, measuring only Pauli-X and Y on $\c{Q} = \b{C}^{2}$ leaves $\mb{\sigma}_{z}$ unidentifiable because $\ker(\f{D}) = \spann \{\m{\Pi}_{x, \pm}, \m{\Pi}_{y, \pm} \}^{\perp} = \spann \{\mb{\sigma}_{z}\}$.
\end{remark}

Consider the following map, which we call the \textbf{Gram superoperator}:
\begin{align}\label{eq:design2spd}
    \frac{\f{D}^{*} \f{D}}{n}: \b{F}_{sa}^{q \times q} \to \b{F}_{sa}^{q \times q}, \,  \m{S} \mapsto \frac{1}{n} \sum_{i=1}^{n} \sum_{k=1}^{q_{i}} \frac{\tr{\m{S} \m{\Pi}_{ik}}}{m_{ik}} \m{\Pi}_{ik},
\end{align}
Regardless of the design of observables, $\b{F}_{sa, 0}^{q \times q}$ and $\spann\{\m{I}_{q}\}$ form invariant subspaces:
\begin{align*}
    \frac{\f{D}^{*} \f{D}}{n} : \m{S} = \omega \frac{\m{I}_{q}}{q} \oplus \mb{\Delta} \ \mapsto \ \omega \frac{\m{I}_{q}}{q} \oplus \frac{\f{D}^{*} \f{D}}{n} \mb{\Delta}, \quad \omega \in \b{R}, \, \mb{\Delta} \in \b{F}_{sa, 0}^{q \times q}.
\end{align*}
Roughly speaking, a complete design amounts to a full-rank matrix in linear regression, where $\m{I}_{q}/q$ acts as the intercept term. In this regard, we now introduce the quantum analogue of an orthogonal design matrix so that $\b{F}_{sa, 0}^{q \times q}$ and $\spann\{\m{I}_{q}\}$ become irreducible subspaces.

\begin{definition}[Unitary Design]
Let $\s{O} = \{\m{O}_{i} \in \b{F}_{sa}^{q \times q}: i = 1, \cdots, n\}$ be a collection of observables on $\c{Q} = \b{F}^{q}$. The collection is said to be an $\alpha$-\emph{unitary design} if there exists a constant $\alpha \in \b{R}$ such that $\left(\tfrac{1}{n}\f{D}^{*} \f{D}\right) \mb{\Delta} = \alpha \mb{\Delta}$
for any $\mb{\Delta} \in \b{F}_{sa, 0}^{q \times q}$.
\end{definition}

Unitary designs can be characterized as follows. 

\begin{theorem}[Invariance]\label{thm:const:unit:dsgn}
Let $\s{O} = \{\m{O}_{i} \in \b{F}_{sa}^{q \times q}: i = 1, \cdots, n\}$.
\begin{enumerate}[leftmargin = *]
\item $\s{O}$ is an $\alpha$-unitary design if and only if the rotated collection $\{\m{V} \m{O}_{i} \m{V}^{*} \in \b{F}_{sa}^{q \times q}: i = 1, \cdots, n\}$ is also an $\alpha$-unitary design for any matrix $\m{V} \in \b{F}^{q \times q}$ with $\m{V} \m{V}^{*} = \m{I}_{q}$ (i.e., orthogonal if $\b{F} = \b{R}$ and unitary if $\b{F} = \b{C}$).
\item The constant $\alpha$ is exclusively determined by the QHS $\c{Q}$ and the average number of distinct eigenvalues $\bar{q} := \sum_{i=1}^{n} q_i / n$:
\begin{align*}
    \alpha = \eta_{\s{O}} \cdot \alpha_{\c{Q}}, \quad \text{where} \quad \eta_{\s{O}} = \frac{\bar{q} - 1}{q-1}, \quad \alpha_{\c{Q}} = \begin{cases} (q/2+1)^{-1} & \b{F}=\b{R} \\ (q+1)^{-1} & \b{F}=\b{C} \end{cases}.
\end{align*}
\end{enumerate}
\end{theorem}
Note that $\alpha > 0$ since $\eta_{\s{O}} \ge 1/(q-1)$. For rank-1 measurements ($q_{i} \equiv q$), $\eta_{\s{O}} = 1$ and $\alpha_{\c{Q}}$ recovers the known frame bounds for projective 2-designs in the context of the Clifford algebra \cite{scott2006tight, hoggar1982t}. \cref{thm:const:unit:dsgn} generalizes this, isolating the attenuation of signal detection via the \emph{multiplicity damping factor} $\eta_{\s{O}}$ caused by utilizing observables with degenerate eigenvalues.

\begin{remark}[Mutually Unbiased Bases]\label{rmk:MUB}
For $\c{Q} = \b{C}^{q}$, a collection of $n =  q+1$ orthonormal bases $\{\phi_{ik} : 1 \le k \le q\}_{i=1}^{n}$ satisfying the overlap condition:
\begin{align*}
    |\innpr{\phi_{ik}}{\phi_{i'k'}}|^2 = \begin{cases} 
    \delta_{kk'}  & \quad i=i' \\
    1/q & \quad i \ne i',
    \end{cases}
\end{align*}
is known as a complete set of Mutually Unbiased Bases (MUBs) \cite{klappenecker2005mutually}. \footnote{While the existence of MUBs for prime powers ($q=p^{k}$) is known, its existence for arbitrary natural number is an open problem, e.g., Zauner's conjecture for $q=6$ \cite{zauner1999grundzuge}.}
Letting $\m{\Pi}_{ik} = \phi_{ik} \phi_{ik}^{*}$ (thus $m_{ik} \equiv 1, \bar{q} = q$), we obtain the trace inner products $\tr{\m{\Pi}_{ik} \m{\Pi}_{i'k'}} = \delta_{kk'} \delta_{ii'} + (1-\delta_{ii'})/q$. Using a property for the sandwich form of complete MUBs:
\begin{align*} 
    \frac{\f{D}^{*} \f{D}}{n} \m{S} = \frac{1}{q+1} \sum_{i=1}^{q+1} \sum_{k=1}^{q} \m{\Pi}_{ik} \m{S} \m{\Pi}_{ik} = \alpha_{\c{Q}} \m{S} + (1-\alpha_{\c{Q}}) \tr{\m{S}} \frac{\m{I}_{q}}{q} , \quad \m{S} \in \b{F}_{sa}^{q \times q},
\end{align*}  
demonstrating that MUBs form a $\alpha_{\c{Q}}$-unitary design. In particular, for $\c{Q} = \b{C}^{2}$, the Pauli measurements $\{\mb{\sigma}_{x}, \mb{\sigma}_{y}, \mb{\sigma}_{z}\}$ is a $1/3$-unitary design, as their eigenprojections $\{\m{\Pi}_{i, \pm}: i \in \{x, y, z\}\}$ satisfy this MUB condition. 
Consequently, $\s{O} = \{\m{I}_{2}, \mb{\sigma}_{x}, \mb{\sigma}_{y}, \mb{\sigma}_{z}\}$ forms a $1/4$-unitary design as $\eta_{\s{O}} = 3/4$, which corresponds to the setting in \cite{cai2016optimal,gross2010quantum} when $q=2$.
\end{remark}

In \cref{ssec:mub}, we develop the theory of our tensorized LSE estimator under MUB observables. We characterize the complete and unitary designs for a single qubit:
\begin{theorem}\label{thm:supop:qubit}
Let $\s{O} = \{\m{O}_{i} \in \b{C}_{sa}^{2 \times 2}: i = 1, \cdots, n\}$ be observables on a qubit system with $q_{i} \equiv 2$. Consider their Bloch representations
\begin{align*}
    \m{O}_{i} = w_{i} \m{I}_{2} + \m{v}_{i} \cdot \vec{\mb{\sigma}}, \quad w_{i} \in \b{R}, \, \m{v}_{i} \in \b{R}^{3} - \{\m{0}\},
\end{align*}
and define $\m{u}_{i} = \m{v}_{i} / \|\m{v}_{i}\|_{\b{R}^{3}} \in \b{S}^{2}$.
Then, the action of the Gram superoperator over traceless self-adjoint matrices is given by:
\begin{align}\label{eq:supop:qubit}
    \frac{\f{D}^{*} \f{D}}{n}: \b{C}_{sa, 0}^{2 \times 2} \to \b{C}_{sa, 0}^{2 \times 2}, \,  \m{a} \cdot \vec{\mb{\sigma}} \mapsto \left[\frac{1}{n} \sum_{i=1}^{n} \m{u}_{i} \m{u}_{i}^{\top} \m{a} \right] \cdot \vec{\mb{\sigma}}, \quad \m{a} \in \b{R}^{3}.
\end{align}
Consequently, the design $\s{O}$ is complete if and only if $\spann \{\m{v}_{i}\}_{i=1}^{n} = \b{R}^{3}$, and is $\alpha$-unitary if and only if $\alpha =1/3$ and $\frac{1}{3} \m{I}_{3} = \frac{1}{n} \sum_{i=1}^{n} \m{u}_{i} \m{u}_{i}^{\top} \in \b{R}^{3 \times 3}$. \footnote{When the bloch vectors of detectors are arranged as a regular tetrahedron on $\b{S}^{2}$, it is called SIC-POVM \cite{renes2004symmetric}.}
\end{theorem}

\begin{remark}[Approximate Unitary Design]\label{rmk:approx:unit}
For arbitrary dimensions $n$ and $q$, constructing an exact unitary design can be theoretically daunting and practically infeasible. However, the proof of \cref{thm:const:unit:dsgn} implies that a unitary design can be efficiently approximated via random sampling.
If the design $\s{O} = \{\m{V}_{i} \m{O} \m{V}_{i}^{*}: i = 1, \cdots, n\}$ is generated by applying random Haar-distributed unitaries $\m{V}_{i}$ to a fixed seed observable $\m{O} \in \b{F}_{sa}^{q \times q}$ that has $q$ distinct eigenvalues, the Strong Law of Large Numbers guarantees that the empirical Gram superoperator converges almost surely to the unitary design condition:
\begin{align*}
    \lim_{n \to \infty} \frac{\f{D}^{*} \f{D}}{n} \m{S} = \alpha_{\c{Q}} \m{S} + (1-\alpha_{\c{Q}}) \tr{\m{S}} \frac{\m{I}_{q}}{q} , \quad \m{S} \in \b{F}_{sa}^{q \times q}.
\end{align*}
In particular for a qubit system $\c{Q} = \b{C}^{2}$, it suffices to generate the i.i.d. unit bloch vectors $\m{u}_{i}$ from the uniform distribution on $\b{S}^{2}$ due to \cref{thm:supop:qubit} to approximate $1/3$-unitary design.
\end{remark}

We formally show that the unitary design condition is stronger than completeness:
\begin{proposition}\label{prop:unit:design}
Let $\alpha > 0$ and $\c{Q} = \b{F}^{q}$. A collection of observables $\{\m{O}_{i} \in \b{F}_{sa}^{q \times q}: i = 1, \cdots, n\}$ is an $\alpha$-unitary design if and only if
\begin{align*}
    \frac{\f{D}^{*} \f{D}}{n} \m{S} = \alpha \m{S} + (1-\alpha) \tr{\m{S}} \frac{\m{I}_{q}}{q} , \quad \m{S} \in \b{F}_{sa}^{q \times q}.
\end{align*}
In this case, the Gram superoperator is invertible (thus a complete design), with its inverse:
\begin{align*}
    \left( \frac{\f{D}^{*} \f{D}}{n} \right)^{-1} \m{S} = \frac{1}{\alpha} \left[ \m{S} - (1-\alpha) \tr{\m{S}} \frac{\m{I}_{q}}{q} \right], \quad \m{S} \in \b{F}_{sa}^{q \times q}.
\end{align*}
\end{proposition}

\cref{prop:unit:design} can be understood through the lens of Schur's Lemma \cite{faraut2008analysis} from representation theory. Because any irreducible representation on a finite-dimensional space is proportional to the identity, the Gram superoperator acts as a scalar multiple on its irreducible invariant subspaces, yielding an eigenvalue of $1$ on the one-dimensional subspace $\spann\{\m{I}_{q}\}$ and an eigenvalue of $\alpha$ on the traceless subspace $\b{F}_{sa, 0}^{q \times q}$.

\begin{corollary}\label{prop:sq:sum:proj}
Let $\{\m{O}_{i} \in \b{F}_{sa}^{q \times q}: i = 1, \cdots, n\}$ be an $\alpha$-unitary design on $\c{Q} = \b{F}^{q}$ with $\alpha > 0$. Then for any $\m{S} \in \b{F}_{sa}^{q \times q}$, it holds that
\begin{equation*}
    \frac{1}{n} \vertiii{\f{D} [\m{S}]}_{2}^{2} = \frac{1}{n} \sum_{i=1}^{n} \sum_{k=1}^{q_{i}} \frac{d_{ik}[\m{S}]^{2}}{m_{ik}} = \alpha \tr{\m{S}^{2}} + \frac{1-\alpha}{q} \tr{\m{S}}^{2}.
\end{equation*}
\end{corollary}

For a density matrix $\mb{\rho} \in \b{S}(\b{F}_{+}^{q \times q})$, this quantity linearly maps to the purity $\tr{\mb{\rho}^2}$, reaching its minimum for the maximally mixed state and its maximum for pure states:
\begin{equation*}
    \frac{1}{n} \vertiii{\f{D} [\mb{\rho}]}_{2}^{2} = \alpha \left(\tr{\mb{\rho}^{2}} - \frac{1}{q} \right) + \frac{1}{q}.
\end{equation*}

\subsection{Tensorized Linear Regression}\label{ssec:est:sch}
While Maximum Likelihood Estimation is a valid choice for a complete design, it becomes computationally prohibitive as the system dimension $q$ grows; 8 qubits ($q=2^{8}$) can require weeks of computation \cite{gross2010quantum, haffner2005scalable}. Consequently, a version of Least Squares Estimation (LSE) \cite{smolin2012efficient} has emerged as a standard, tractable alternative in state tomography. We now demonstrate how our framework provides a rigorous, tensorized formulation of this classical approach. By expressing the inverse problem strictly through basis-independent linear superoperators, we seamlessly recover the standard unconstrained LSE while providing closed-form geometric insights into its covariance structure, mean-square error, and statistical optimality.

For each $i= 1, \cdots, n$, we denote the empirical frequency for the outcome $\lambda_{ik}$ by:
\begin{align*}
    p_{ik} := \frac{\sharp \{1 \le j \le r : Y_{ij} = \lambda_{ik} \}}{r} \ge 0, \quad k=1, 2, \cdots, q_{i}, \quad p_{i1} + \cdots + p_{iq_{i}} = 1.
\end{align*}
Consequently, the empirical distribution $\hat{\mb{\nu}}_{i, \mb{\rho}} = \frac{1}{r} \sum_{j=1}^{r} \delta_{Y_{ij}} = \sum_{k=1}^{q_{i}} p_{ik} \delta_{\lambda_{ik}}$
serves as a proxy for the true probability measure $\mb{\nu}_{i, \mb{\rho}}$ generated by the unknown state $\mb{\rho} \in \b{S}(\b{F}_{+}^{q \times q})$. The expectations satisfy $d_{ik}[\mb{\rho}] = \tr{\mb{\rho} \m{\Pi}_{ik}} = \b{E}[p_{ik} \vert \mb{\rho}]$. Since the counts $r \cdot (p_{i1}, \dots, p_{iq_{i}})$ follow  $\text{Multinomial}(r, (d_{i1}[\mb{\rho}], \dots, d_{iq_{i}}[\mb{\rho}]))$, the covariance structure is given by:
\begin{align}\label{eq:multinom:cov}
    \Cov (p_{ik}, p_{i'k'}  \vert \mb{\rho}) 
    &= \begin{cases}
        (d_{ik}[\mb{\rho}] - d_{ik}[\mb{\rho}]^{2})/r & \quad i=i', k=k' \\
        - d_{ik}[\mb{\rho}] d_{ik'}[\mb{\rho}]/r & \quad i=i', k \neq k' \\
        0 & \quad i \neq i'.
    \end{cases}
\end{align}

To formulate the estimation as a linear regression problem, we construct the data matrix $\m{P} \in \b{R}^{n \times q}$ incorporating the multiplicities, analogous to the design tensor \eqref{eq:quant:des:op}:
\begin{align*}
    \m{P} = 
    \begin{pmatrix}
        \m{p}_{1} \vert \m{p}_{2} \vert   \cdots \vert \m{p}_{n}
    \end{pmatrix}^{\top} \in \b{R}^{n \times q}, \quad \m{p}_{i} := \left[p_{i k_{i}(\ell)}/m_{i k_{i}(\ell)}  \right]_{1 \le \ell \le q}, 
\end{align*}
where each vector $\m{p}_{i} \in \b{R}^{q}$ is formed by taking the scaled empirical frequency $p_{ik}/m_{ik}$ and repeating it exactly $m_{ik}$ times for each eigenspace index $k = 1, \dots, q_{i}$. We then define the quadratic loss function:
\begin{align}\label{eq:def:loss}
    \mathrm{Loss}[\m{S}] := \vertiii{\f{D}[\m{S}] - \m{P}}_{2}^{2}, \quad \m{S} \in \b{F}_{sa}^{q \times q}.
\end{align}
The estimator is defined as the minimizer over the unit-trace affine hyperplane:
\begin{align}\label{eq:def:density:est}
    \hat{\mb{\rho}}^{\text{LSE}} \in \argmin \{\mathrm{Loss}[\m{S}] : \m{S} \in \b{F}_{sa, 1}^{q \times q}\}.
\end{align}
Note that we explicitly relax the n.n.d. constraint $\mb{\rho} \succeq 0$. If the resulting estimator $\hat{\mb{\rho}}^{\text{LSE}}$ lies outside the convex set $\b{S}(\b{F}_{+}^{q \times q})$, it can be projected back onto the set of valid density matrices, which can be computed straightforwardly using the Trace-Preserving Projection algorithm in \cref{sec:tpp}. Since this projection is a contraction in the Frobenius norm, it preserves the statistical consistency and asymptotic convergence rates of the unconstrained estimator $\hat{\mb{\rho}}^{\text{LSE}}$.

\begin{theorem}[Linear Regression]\label{thm:est:normal:eq}
Given the data matrix $\m{P}$ and a collection of observables $\s{O} = \{\m{O}_{i}: i = 1, \cdots, n\}$, the set of minimizers in $\b{F}_{sa, 1}^{q \times q}$ is given by the affine space:
\begin{equation*}
    \argmin \{\mathrm{Loss}[\m{S}] : \m{S} \in \b{F}_{sa, 1}^{q \times q}\} = \hat{\mb{\rho}}^{\text{LSE}} \oplus \ker(\f{D}),
\end{equation*}
where $\hat{\mb{\rho}}^{\text{LSE}} = (\f{D}^{*} \f{D})^{\dagger} \f{D}^{*} \m{P} \in \b{F}_{sa, 1}^{q \times q}$ is the minimum Frobenius norm solution, and the null space corresponds to the unobservable directions from \eqref{eq:null:mat}. If the design $\s{O}$ is complete, the minimizer is uniquely given by $\hat{\mb{\rho}}^{\text{LSE}} = (\f{D}^{*} \f{D})^{-1} \f{D}^{*} \m{P}$.
\end{theorem}

For a complete design $\s{O}$, the LSE is unbiased, i.e., $\b{E}[\hat{\mb{\rho}}^{\text{LSE}} \vert \mb{\rho}] = (\f{D}^{*} \f{D})^{-1} \f{D}^{*} \b{E}[\m{P} \vert \mb{\rho}] = (\f{D}^{*} \f{D})^{-1} \f{D}^{*} \f{D} [\mb{\rho}] = \mb{\rho}$ for any density matrix $\mb{\rho} \in \b{S}(\b{F}_{+}^{q \times q})$.
Surprisingly, the two-step procedure, unconstrained estimation followed by trace-preserving projection, is mathematically equivalent to solving the fully constrained optimization problem for unitary designs.

\begin{corollary}\label{cor:equiv:proj}
Let $\s{O}$ be an $\alpha$-unitary design on $\c{Q} = \b{F}^{q}$ with $\alpha > 0$. The unique projection of the relaxed estimator $\hat{\mb{\rho}}^{\text{LSE}} \in \b{F}_{sa, 1}^{q \times q}$ onto $\b{S}(\b{F}_{+}^{q \times q})$ is precisely the exact global solution to the constrained regression problem:
\begin{align*}
    \argmin \{\vertiii{\mb{\rho} - \hat{\mb{\rho}}^{\text{LSE}}}_{2} : \mb{\rho} \in \b{S}(\b{F}_{+}^{q \times q})\} = 
    \argmin \{\mathrm{Loss}[\mb{\rho}] : \mb{\rho} \in \b{S}(\b{F}_{+}^{q \times q})\}.
\end{align*}
\end{corollary}

While the trace-preserving projection in \cite{smolin2012efficient} was originally proposed as a computationally efficient proxy for the MLE, \cref{cor:equiv:proj} rigorously proves that it induces the exact constrained least-squares solution. For unitary designs, we can further derive the explicit expression for the mean-square error (MSE):
\begin{align*}
    \mathrm{MSE}[\hat{\mb{\rho}}^{\text{LSE}} \vert \mb{\rho}] := \b{E} \left[\vertiii{\mb{\rho} - \hat{\mb{\rho}}^{\text{LSE}}}_{2}^{2} \vert \mb{\rho} \right].
\end{align*}

\begin{theorem}\label{cor:unit:mse}
Let $\s{O}$ be an $\alpha$-unitary design with $\alpha > 0$. Then, the LSE simplifies to:
\begin{align}\label{eq:unit:des:LSE}
    \hat{\mb{\rho}}^{\text{LSE}} = \frac{\f{D}^{*} \m{P}}{n \alpha} - \frac{1-\alpha}{\alpha} \frac{\m{I}_{q}}{q}, \quad \text{where} \quad \f{D}^{*} \m{P} = \sum_{i=1}^{n} \sum_{k=1}^{q_{i}} \frac{p_{ik}}{m_{ik}} \m{\Pi}_{ik}.
\end{align}
If the design consists exclusively of rank-1 projections ($m_{ik} \equiv 1$), then $\alpha = \alpha_{\c{Q}}$ in \cref{thm:const:unit:dsgn}, and the mean-squared error for any density matrix $\mb{\rho} \in \b{S}(\b{F}_{+}^{q \times q})$ is given by:
\begin{align*}
    \mathrm{MSE}[\hat{\mb{\rho}}^{\text{LSE}} \vert \mb{\rho}] = \frac{1}{n r \alpha_{\c{Q}}^{2}} \left[ \left(1 - \frac{1}{q} \right) -  \alpha_{\c{Q}} \left(\tr{\mb{\rho}^{2}} - \frac{1}{q} \right) \right].
\end{align*}
\end{theorem}

Recall that $1/q \le \tr{\mb{\rho}^{2}} \le 1$, with equality holding if and only if $\mb{\rho}$ is the maximally mixed state ($\m{I}_{q}/q$) or a pure state, respectively. Hence, under the $\alpha_{\c{Q}}$-unitary design, the MSE strictly varies with the state purity:
\begin{align*}
    \frac{1-\alpha_{\c{Q}}}{n r \alpha_{\c{Q}}^{2}} \left(1 - \frac{1}{q} \right) \le \mathrm{MSE}[\hat{\mb{\rho}}^{\text{LSE}} \vert \mb{\rho}] \le \frac{1}{n r \alpha_{\c{Q}}^{2}} \left(1 - \frac{1}{q} \right).
\end{align*}  
This demonstrates that the MSE achieves the standard parametric rate $O((nr)^{-1})$ and scales with the dimension of the operator space as $O(q^{2})$. 
We briefly remark on the \emph{curse of dimensionality}: even with a unitary design where the measurement sensitivity is geometrically isotropic, the signal is inherently dampened by the factor $\alpha_{\c{Q}}$, causing the MSE to scale inversely with $\alpha_{\c{Q}}^{2} \asymp q^{2}$. Furthermore, the relative gap between the best-case (pure state) and worst-case (maximally mixed state) error is determined precisely by $1 - \alpha_{\c{Q}}$, a quantity which vanishes asymptotically as $q \to \infty$.

\begin{proposition}[Fixed Budget]\label{prop:fixed:bud}
Let $\s{O} = \{\m{O}_{i} \in \b{F}_{sa}^{q \times q}: i = 1, \cdots, n\}$ be a design on $\c{Q} = \b{F}^{q}$. The trace of the superoperator is determined solely by the average number of distinct eigenvalues $\tr{\tfrac{1}{n}\f{D}^{*} \f{D}} = \bar{q} = \tfrac{1}{n} \sum_{i=1}^{n} q_{i}$.
Consequently, the smallest eigenvalue $\tau_{\min}$ of the Gram superoperator satisfies the upper bound $\tau_{\min} \le \alpha = \eta_{\s{O}} \cdot \alpha_{\c{Q}}$, with equality holding if and only if $\s{O}$ is $\alpha$-unitary. 
\end{proposition}

We now establish the minimax lower bound of QST. The result below applies to any complete design and any unknown density, with the lower bound scaling precisely as $q^{2}/(nr)$. 

\begin{theorem}[Lower Bound]\label{thm:minimax}
Let $\{\m{O}_{i}\}_{i=1}^{n}$ be a complete design on $\c{Q} = \b{F}^{q}$. In the asymptotic regime where the sample size is sufficiently large relative to the dimension, specifically $r \ge q c_{0} \tr{(\f{D}^{*} \f{D} )^{-1}}$ for some universal constant $c_{0} > 0$, the minimax risk for estimating the density matrix satisfies:
\begin{align*}
    \inf_{\hat{\mb{\rho}} \in \b{F}^{q \times q}} \sup_{\mb{\rho} \in \b{S}(\b{F}_{+}^{q \times q})} \mathrm{MSE}[\hat{\mb{\rho}} \vert \mb{\rho}] \ge C \frac{q^{2}}{nr},
\end{align*}
where $C > 0$ is a universal constant independent of $q, n, r$.
\end{theorem}

A closer inspection of the proof reveals that the unitary design is optimal (up to constant), as it maximizes the separation of the packing set for a fixed budget $\tr{\f{D}^{*} \f{D}/n} = \bar{q}$ in \cref{prop:fixed:bud}. Specifically, the unitary design maximizes the minimum eigenvalue of the Gram superoperator, a property analogous to E-optimality in classical experimental design \cite{pukelsheim2006optimal}, despite the non-Gaussian nature of the quantum setting.
We also note that a similar minimax rate for the trace norm discrepancy has been considered in \cite{butucea2025sample} albeit without explicit dependency on the dimension of the QHS, $q$.

\begin{example}[Real Qubit System]
Although usual quantum mechanics uses $\b{F} = \b{C}$, we consider the lowest dimensional QHS, the real qubit (\emph{rebit}) system $\c{Q} = \b{R}^{2}$ to explicitly illustrate the geometry of the regression. We adopt the \emph{Bloch disk representation} to parametrize the space of unit-trace symmetric matrices:
\begin{align*}
    &\m{S} = \frac{1}{2} \m{I}_{2} + \frac{s}{2} \begin{pmatrix}
        \cos(\theta) & \sin(\theta) \\
        \sin(\theta) & -\cos(\theta)
    \end{pmatrix}  \in \b{R}_{sa, 1}^{2 \times 2}, \quad z := s e^{\imath \theta} \in \b{C},
\end{align*}
where $\theta \in [0, 2\pi)$ represents the orientation. Note that $\m{S} \succeq \m{0}$ is a valid density matrix if and only if $s \in [0, 1]$, where $s$ corresponds to the purity of the state ($s=1$ for pure states).
We construct a measurement design using $n$ rotation matrices:
\begin{equation*}
    \m{U}_{i} = \m{U}(\theta_{i}) := \begin{pmatrix}
        \cos (\theta_{i}/2) & \sin (\theta_{i}/2) \\
        -\sin (\theta_{i}/2) & \cos (\theta_{i}/2)
    \end{pmatrix} \in SO(2), \quad \theta_{i} \in [0, 2\pi), \quad i = 1, \dots, n.
\end{equation*}
Let the seed observable be $\m{O} = \Diag[+1, -1]$ with eigenvectors $\m{e}_{+} = [1, 0]^{\top}$ and $\m{e}_{-} = [0, 1]^{\top}$. The rotated observables are $\m{O}_{i} = \m{U}_{i}^{\top} \m{O} \m{U}_{i}$, which form a complete design provided the angles $\theta_{i}$ are not all identical. The expected outcome probabilities are given by:
\begin{align*}
    d_{i,\pm}[\m{S}] = \langle (\m{U}_{i}^{\top} \m{e}_{\pm}), \m{S} (\m{U}_{i}^{\top} \m{e}_{\pm}) \rangle_{\b{R}^{2}} = \frac{1 \pm s \cos (\theta - \theta_{i})}{2}.
\end{align*}
Due to the symmetry $|d_{i, +}[\m{S}] - p_{i, +}| = |d_{i, -}[\m{S}] - p_{i, -}|$, the loss function simplifies to:
\begin{align*}
    \mathrm{Loss}[\m{S}] &= 2 \sum_{i=1}^{n} |d_{i, +}[\m{S}] - p_{i, +}|^{2} = \sum_{i=1}^{n} \frac{\left|s \cos (\theta - \theta_{i}) - \left(2 p_{i, +} - 1\right) \right|^{2}}{2} .
\end{align*}

Using the Bloch disk representation, the loss can be rewritten in terms of two complex parameters $\bar{D} = \frac{1}{n} \sum_{i=1}^{n} e^{\imath 2\theta_{i}}$ and $\bar{P} = \frac{1}{n} \sum_{i=1}^{n} (2 p_{i, +}-1) e^{\imath \theta_{i}}$:
\begin{align*} 
    \text{Loss}(z) = \frac{n}{4} |z|_{\b{C}}^2 + \frac{1}{4} \Re(z^2 \bar{D}^*) - \Re(z \bar{P}^*) + \text{const}.
\end{align*}
Only for this example, we reserve the bar notation $\bar{\cdot}$ for sample averages to align with statistical convention; thus, we use $^{*}$ to denote the complex conjugate, and $\Re$ and $\Im$ for the real and imaginary parts. The data average variable $\bar{P} \in \b{C}$, being an average of complex numbers in the unit disk, also satisfies $|\bar{P}|_{\b{C}} \le 1$. 
Solving the loss function for $z$ yields the explicit LSE from \cref{thm:est:normal:eq}:
\begin{align*}
    \hat{z}^{\text{LSE}} = \frac{2(\bar{P} - \bar{D} \bar{P}^*)}{1 - |\bar{D}|^2} \in \b{C} \quad \iff \quad
    \hat{\mb{\rho}}^{\text{LSE}} = \frac{1}{2} \m{I}_{2} + \frac{1}{2} \begin{pmatrix}
        \Re (\hat{z}^{\text{LSE}}) & \Im (\hat{z}^{\text{LSE}}) \\
        \Im (\hat{z}^{\text{LSE}}) & - \Re (\hat{z}^{\text{LSE}})
    \end{pmatrix} \in \b{R}^{2 \times 2}.
\end{align*}
In particular, if the angles $\theta_{i} = 2\pi i/n$ are uniformly spaced, then $\bar{D} = 0$. This corresponds to a $1/2$-unitary design, and the LSE reduces to $\hat{z}^{\text{LSE}} = 2 \bar{P}$.
Therefore, all we need for the estimation procedure to compute is essentially $\bar{P}$, as displayed in the bottom left of \cref{fig:figure_rebit} for a single realization.
In this case, \cref{cor:unit:mse} gives:
\begin{equation*}
    \mathrm{MSE}[\hat{\mb{\rho}}^{\text{LSE}} \vert \mb{\rho}] = \frac{2}{nr} \left( 1 - \frac{s^{2}}{2} \right).
\end{equation*}
Note that the LSE violates the n.n.d. condition if and only if $|\hat{z}^{\text{LSE}}|_{\b{C}} > 1$ (or, equivalently, $|\bar{P}|_{\b{C}} > 1/2$). In this case, the exact constrained optimal solution is obtained by simple radial shrinkage due to \cref{cor:equiv:proj}: $\hat{z}_{+}^{\text{LSE}} = \hat{z}^{\text{LSE}}/|\hat{z}^{\text{LSE}}|_{\b{C}}$, which yields a pure state.

\begin{figure}[h!]
    \begin{minipage}[c]{0.35\linewidth}
        \includegraphics[width=\linewidth]{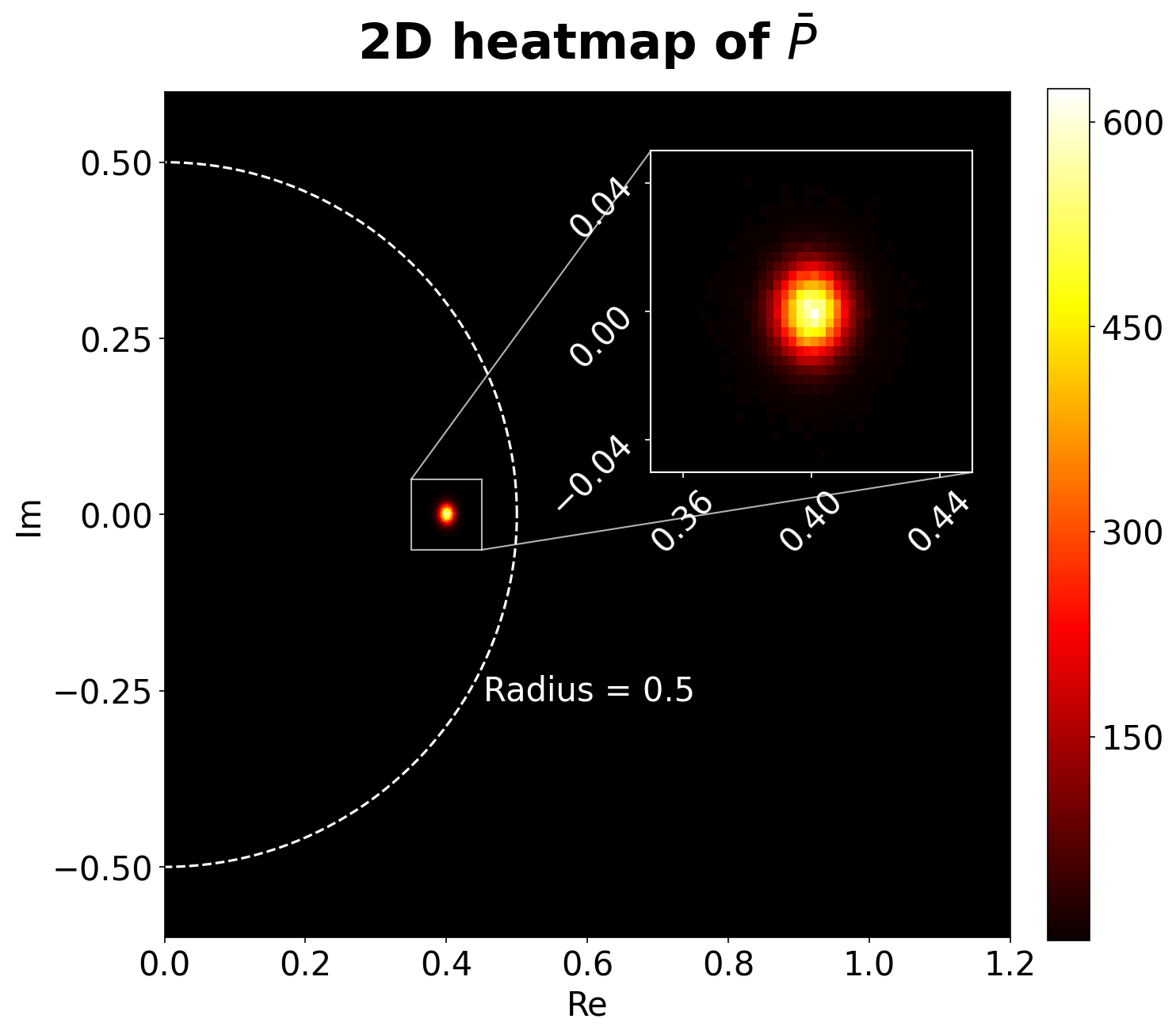}
        \vspace{2mm}
        \includegraphics[width=\linewidth]{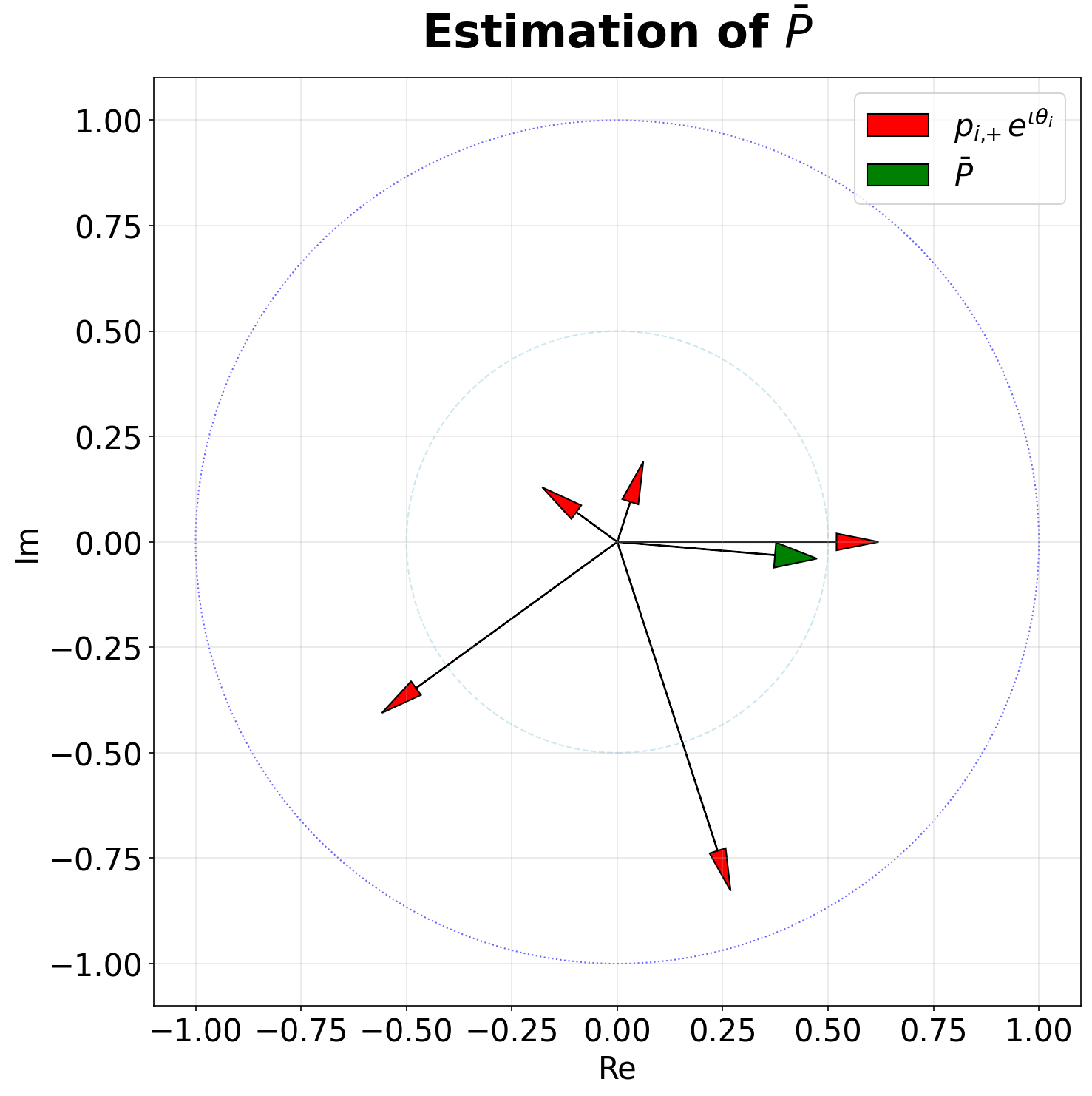}
    \end{minipage}%
    \hfill%
    \begin{minipage}[c]{0.63\linewidth}
        \includegraphics[width=\linewidth]{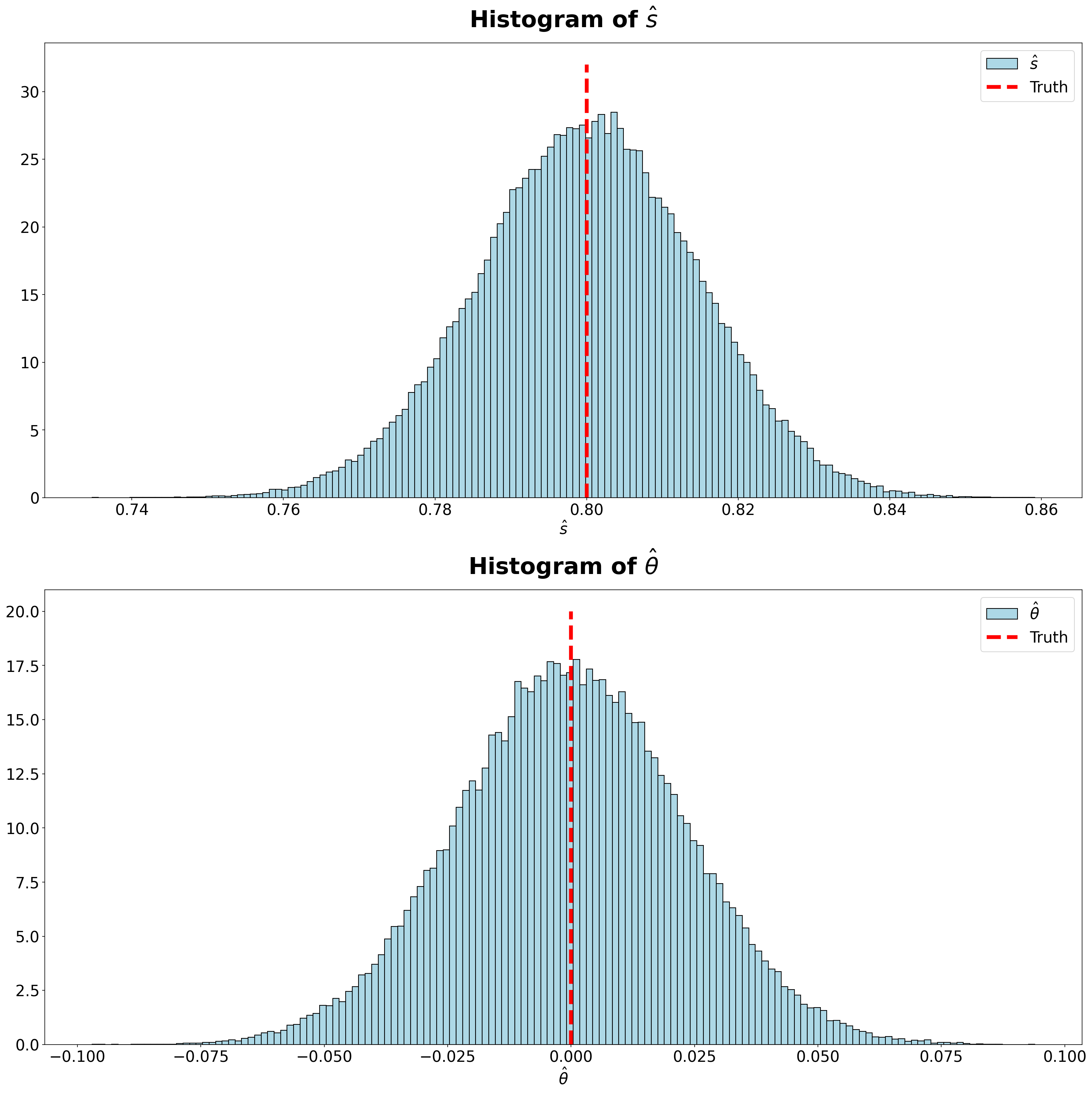}
    \end{minipage}
    \caption{
        Estimation behavior with $\mb{\rho}=\Diag[0.9,0.1]$.
        \textbf{Bottom left:} Single realization with $n=5$ observables and $r=50$ measurements per observable,
        showing $p_{i,+}e^{\imath\theta_i} \in \b{C}$ and the resulting estimator $\bar{P}$.
        \textbf{Top left:} Two-dimensional histogram of $\bar{P}$ over $10^5$ repetitions
        with $r=50$ measurements per $n=100$ observables.
        \textbf{Right:} Histograms of $\hat{s}$ and $\hat{\theta}$, respectively;
        dashed red lines indicate the true parameter values $s=0.8$ and $\theta=0$.}
    \label{fig:figure_rebit}
\end{figure}
We report the behavior of the Bloch disk representation of estimated density matrix $\hat{\mb{\rho}}_{+}^{\text{LSE}}$:
\begin{align*}
    \hat{z}_{+}^{\text{LSE}} = \hat{s} e^{\imath \hat{\theta}} \in \b{C}, \quad \hat{s} = \min\left(\lvert\hat{z}^{\text{LSE}}\rvert, 1\right), \quad \hat{\theta} = \arg\left(\hat{z}^{\text{LSE}}\right).
\end{align*}
Here, we conduct the Monte Carlo simulation $10^{5}$ times, where the data was generated with $n=100$ observables with uniformly spaced angles $\theta_{i} = 2\pi i/n$. To demonstrate that radial shrinkage is rarely required in practice, the underlying density matrix was chosen to be $\mb{\rho} = \Diag(0.9, 0.1)$ (thus $s= 0.8, \theta = 0$), closed to a pure state. Under this setup, for each observable $\m{O}_{i}$, we have generated the measurement outcomes $r=50$ times according to the Born rule, respectively. 
The right panels of \cref{fig:figure_rebit} display the empirical histograms of $\hat{s}$ and $\hat{\theta}$ obtained from the Monte Carlo simulation; both exhibit a strong concentration around their true values, resulting in an accurate performance of $\hat{\mb{\rho}}_{+}^{\text{LSE}}$. In particular, the edge value of the histogram for $\hat{s}$ is $[0.74, 0.86]$, meaning radial shrinkage is completely inactive.
\end{example}

In \cref{ex:comp:qubit:LSE}, we derive the Bloch vector representation of our tensorized LSE estimator for a complex qubit system and radial shrinkage for projection to yield a valid density matrix. Additionally, we also derive the MSE error in the presence of the bit-flip noise.

\subsection{Tensorized Kernel Regression}\label{ssec:est:sch:ker}
While standard linear regression in \cref{ssec:est:sch} suffices when spectra are treated as abstract logical labels, it fails to capture the geometric reality of physical implementations. In hardware based on continuous variables --- such as angular momentum \cite{arecchi1972atomic} or discrete phase \cite{pegg1989phase} --- measurement outcomes possess an intrinsic metric where errors manifest as local displacements rather than categorical flips \cite{gottesman2001encoding}. To reflect this structure, we introduce a reproducing kernel $K:\b{R} \times \b{R} \to \b{F}$ to encode the geometry over the spectrum. In this view, standard linear regression is recovered as the specific choice of the 0--1 kernel $K(x, y) = \ds{1} (x = y)$, corresponding to a ``metric-free'' topology over spectra.

We define the $i$-th Gram matrix associated with $\m{O}_{i}$ incorporating the multiplicities:
\begin{align}\label{eq:ith:Gram}
    \m{K}_{i} \in \b{F}_{+}^{q \times q}, \quad [\m{K}_{i}]_{\ell \ell'} := K(\lambda_{i k_{i}(\ell)}, \lambda_{i k_{i}(\ell')}), \quad 1 \le \ell, \ell' \le q.
\end{align}
Note that we index by $\ell \in \{1, \dots, q\}$ to account for eigenvalue multiplicities. We now relate the functional discrepancy of conditional QCE to a spectral matrix norm.

\begin{theorem}[Quantum Representer Theorem]\label{thm:loss:max:eigen}
Let $\c{Q} = \b{F}^{q}$. For each $i = 1, \dots, n$, let $\m{K}_{i}$ be the Gram matrix defined in \eqref{eq:ith:Gram}. Given any $\m{S} \in \b{F}_{sa}^{q \times q}$, define the diagonal matrices of expectations and empirical frequencies:
\begin{align*}
    \m{D}_{i}[\m{S}] := \Diag(\m{d}_{i}[\m{S}])  \in \b{R}^{q \times q}, \quad \m{P}_{i} := \Diag(\m{p}_{i}) \in \b{R}^{q \times q},
\end{align*}
where $\m{d}_{i}[\m{S}], \m{p}_{i} \in \b{R}^{q}$ are the $i$-th rows of the design tensor $\f{D}[\m{S}]$ and data matrix $\m{P}$ respectively as in \eqref{eq:quant:des:op}. Then, for any $s \in [1, \infty]$, the $s$-Schatten norm of the discrepancy between the population QCE $\m{T}_{K}^{\mb{\nu}_{i, \m{S}}}$ and the empirical QCE $\m{T}_{K}^{\hat{\mb{\nu}}_{i, \mb{\rho}}}$ is given by:
\begin{align*}
    \vertiii{\m{T}_{K}^{\mb{\nu}_{i, \m{S}}} - \m{T}_{K}^{\hat{\mb{\nu}}_{i, \mb{\rho}}}}_{s} = \vertiii{\m{K}_{i}^{1/2} (\m{D}_{i}[\m{S}] - \m{P}_{i}) \m{K}_{i}^{1/2}}_{s} < \infty.
\end{align*}
\end{theorem}

The finite-dimensional reduction effectively projects out the orthogonal complement, conceptually mirroring the classical representer theorem in kernel ridge regression \cite{paulsen2016introduction}. As noted previously in \cref{rmk:QMD:schatten}, we adopt the Hilbert-Schmidt norm ($s=2$) for the conditional Quantum Maximum Discrepancy to yield a tractable convex loss function:
\begin{align*}
    \mathrm{Loss}_{K}[\m{S}] := \sum_{i=1}^{n} \vertiii{\m{T}_{K}^{\mb{\nu}_{i, \m{S}}} - \m{T}_{K}^{\hat{\mb{\nu}}_{i, \mb{\rho}}}}_{2}^{2} = \sum_{i=1}^{n} \vertiii{\m{K}_{i}^{1/2} (\m{D}_{i}[\m{S}] - \m{P}_{i}) \m{K}_{i}^{1/2}}_{2}^{2}, \quad \m{S} \in \b{F}_{sa}^{q \times q}.
\end{align*}
We then define the \textbf{QUARK (QUAntum Regression with Kernels)} estimator by relaxing the n.n.d. constraint:
\begin{align}\label{eq:def:density:est:ker}
    \hat{\mb{\rho}}^{K} \in \argmin \{\mathrm{Loss}_{K}[\m{S}] : \m{S} \in \b{F}_{sa, 1}^{q \times q}\},  
\end{align}
which generalizes the tensorized linear regression derived in \cref{ssec:est:sch}. Specifically, if we select the 0--1 kernel, the loss reduces to the standard squared error provided $m_{ik} \equiv 1$:
\begin{align*}
    \mathrm{Loss}[\m{S}] = \sum_{i=1}^{n} \vertiii{ \m{D}_{i}[\m{S}] - \m{P}_{i} }_{2}^{2} = \sum_{i=1}^{n} \|\m{d}_{i}[\m{S}] - \m{p}_{i}\|_{\b{R}^{q}}^{2} = \vertiii{\f{D}[\m{S}] - \m{P}}_{2}^{2}.
\end{align*}

To derive the normal equation for a general kernel, recall the linear map $\m{S} \mapsto \m{d}_{i}[\m{S}]$. Unlike the tensorized linear regression, we will shortly see that the multiplicities $m_{ik}$ cancel out. Consider the following operator and its adjoint: 
\begin{align*}
    \f{D}_{i}: \b{F}_{sa}^{q \times q} \to \b{R}^{q_{i}}, \, \m{S} \mapsto 
    \left[d_{ik}[\m{S}]  \right]_{1 \le k \le q_{i}}, \quad
    \f{D}_{i}^{*}: \b{R}^{q_{i}} \to \b{F}_{sa}^{q \times q}, \, \m{a} \mapsto \sum_{k=1}^{q_{i}} a_{ik} \m{\Pi}_{ik}.
\end{align*}
For estimation within the RKHS framework, the squared kernel $|K|^2 : \b{R} \times \b{R} \to \b{R}$ naturally governs the estimation procedure. Letting $\mb{\Omega}_{i} \in \b{R}_{+}^{q_{i} \times q_{i}}$ be the Gram matrices evaluated at the spectra of observables $\m{O}_{i}$, i.e., $[\mb{\Omega}_{i}]_{kk'} = |K(\lambda_{ik}, \lambda_{ik'})|^{2}$, then the \textbf{kernel Gram superoperator} and the effective data matrix become:
\begin{align*}
    &\f{H}_{K} = \frac{1}{n} \sum_{i=1}^{n} \f{D}_{i}^* \mb{\Omega}_{i} \f{D}_{i} : \b{F}_{sa}^{q \times q} \to \b{F}_{sa}^{q \times q}, \ \m{S} \mapsto \frac{1}{n} \sum_{i=1}^{n} \sum_{k, k'=1}^{q_{i}} [\mb{\Omega}_i]_{kk'} \tr{\m{S} \m{\Pi}_{ik}} \m{\Pi}_{ik'}, \\ 
    &\m{P}_{K} = \frac{1}{n} \sum_{i=1}^{n} \f{D}_{i}^* \mb{\Omega}_{i} \left[p_{i k} \right]_{1 \le k \le q_{i}} = \frac{1}{n} \sum_{i=1}^{n} \sum_{k, k'=1}^{q_{i}} [\mb{\Omega}_i]_{kk'} p_{ik} \m{\Pi}_{ik'} \in \b{F}_{sa}^{q \times q}.
\end{align*}

\begin{proposition}\label{prop:superoperator_trace_norm}
Analogous to \cref{prop:fixed:bud}, the trace of the kernel Gram superoperator is determined solely by the eigenvalue multiplicities and the diagonal evaluation of the kernel, independent of the specific eigenprojections:
\begin{align*}
    \vertiii{\f{H}_{K}}_{1} = \frac{1}{n} \sum_{i=1}^{n} \sum_{k=1}^{q_{i}} m_{ik} [\mb{\Omega}_i]_{kk} .
\end{align*}
\end{proposition}

However, unlike the Gram superoperator for tensorized linear regression, the kernel Gram superoperator generally preserves neither the identity nor the trace, despite being a n.n.d. tensor. Consequently, the inversion inherently induces a trace shrinkage:
\begin{theorem}\label{thm:quark:supop:pd}
Let $\s{O} = \{\m{O}_{i} \in \b{F}_{sa}^{q \times q}: i = 1, \cdots, n\}$ be a complete design. If the squared kernel $|K|^2 : \b{R} \times \b{R} \to \b{R}$ is strictly positive definite (p.d.), then so is the kernel Gram superoperator, and the QUARK estimator in \eqref{eq:def:density:est:ker} is uniquely given by
\begin{align*}
    \hat{\mb{\rho}}^{K} = \f{H}_{K}^{-1}(\m{P}_{K}) + (1 - \tr{\f{H}_{K}^{-1}(\m{P}_{K})}) \frac{\f{H}_{K}^{-1}(\m{I}_{q})}{\tr{\f{H}_{K}^{-1}(\m{I}_{q})}} \in \b{F}_{sa, 1}^{q \times q}.
\end{align*}
\end{theorem}

As the kernel Gram superoperator is not trace-preserving, the QUARK estimator performs a version of shrinkage estimation at the matrix level:
\begin{align*}
    \min_{\m{S} \in \b{F}_{sa}^{q \times q}} \mathrm{Loss}_{K}[\m{S}] \quad \text{with respect to} \quad \tr{\m{S}}=1,
\end{align*}
and the penalty parameter in the Lagrange multiplier is explicitly fixed by the unit-trace constraint to achieve the shrinkage effect.
Indeed, the QUARK estimation procedure is homogeneous with respect to the kernel scaling; inference depends on the relative differences of the kernel values rather than their absolute magnitudes. To see this, consider a scalar multiple of a kernel $c K$ for some $c \in \b{F}$. Then, $\f{H}_{cK} = |c|^{2} \f{H}_{K}$ and $\m{P}_{cK} = |c|^{2} \m{P}_{K}$, thus $\hat{\mb{\rho}}^{c K} = \hat{\mb{\rho}}^{K}$.

\begin{corollary}[Unbiased]\label{cor:unbiased}
Under the assumptions of \cref{thm:quark:supop:pd}, the QUARK estimator in \eqref{eq:def:density:est:ker} is unbiased, i.e., $\b{E}[\hat{\mb{\rho}}^{K} \vert \mb{\rho}] = \mb{\rho}$ for any density matrix $\mb{\rho} \in \b{S}(\b{F}_{+}^{q \times q})$.
\end{corollary}

As in \cref{ssec:est:sch}, in the case where $\hat{\mb{\rho}}^{K} \in \b{F}_{sa, 1}^{q \times q}$ does not satisfy the n.n.d. constraint, one may apply the trace-preserving projection algorithm, which preserves consistency and asymptotic convergence rates.
To tackle the asymptotic behavior of the QUARK estimator under the fixed complete design, we introduce two superoperators with respect to the kernel:
\begin{proposition}\label{prop:ker:tr:centre}
Let $\{\m{O}_{i}\}_{i=1}^n$ be a complete design, and the squared kernel $|K|^2 : \b{R} \times \b{R} \to \b{R}$ be strictly p.d.. Define the \textbf{kernel trace-centering superoperator} by
\begin{align*}
    \f{C}_{K} : \b{F}_{sa}^{q \times q} \to \b{F}_{sa}^{q \times q}, \, \m{S} \mapsto \m{S} - \frac{ \tr{\m{S}}}{\tr{\f{H}_K^{-1}(\m{I}_{q})}} \f{H}_K^{-1}(\m{I}_{q}).
\end{align*}
Then, its adjoint is given by
\begin{align*}
    \f{C}_{K}^{*} : \b{F}_{sa}^{q \times q} \to \b{F}_{sa}^{q \times q}, \, \m{S} \mapsto \m{S} - \frac{ \tr{\f{H}_K^{-1}(\m{S})}}{\tr{\f{H}_K^{-1}(\m{I}_{q})}} \m{I}_{q}.
\end{align*}
Consequently, the \textbf{kernel centered-inversion superoperator} $\f{A}_{K}:= \f{H}_{K}^{-1} \f{C}_{K}^{*}$ is self-adjoint, and the QUARK estimator can be decomposed as:
\begin{align*}
    \hat{\mb{\rho}}^{K} =  \frac{\m{I}_{q}}{q} \oplus \f{A}_{K} \left[\m{P}_K - \frac{\m{I}_{q}}{q} \right] \quad \in \quad \b{F}_{sa}^{q \times q} = \spann\{\m{I}_{q}\} \oplus \b{F}_{sa, 0}^{q \times q}.
\end{align*}
\end{proposition}

To derive the Central Limit Theroem (CLT) of the QUARK estimator $\hat{\mb{\rho}}^{K}$ in the sample size per observable $r$, we first compute the covariance structure. In what follows, given a self-adjoint operator $\m{M} \in \b{F}_{sa}^{q \times q}$, we denote the superoperator via outer product by
\begin{align*}
    \m{M}^{\otimes 2} : \b{F}_{sa}^{q \times q} \to \b{F}_{sa}^{q \times q}, \, \m{S} \mapsto \innpr{\m{S}}{\m{M}}_{2} \m{M}.
\end{align*}

\begin{proposition}\label{prop:eff:data:cov}
Under the assumptions of \cref{thm:quark:supop:pd}, the mean of the effective data matrix $\m{P}_{K}$ is given by $\b{E}(\m{P}_{K} \vert \mb{\rho}) = \f{H}_{K} [\mb{\rho}]$, and its \emph{scaled} covariance superoperator $\f{S}_{K, \mb{\rho}} := r \cdot \b{E}[(\m{P}_{K} - \f{H}_{K} [\mb{\rho}])^{\otimes 2}] $
is given by
\begin{align*}
    \f{S}_{K, \mb{\rho}} = \frac{1}{n^{2}} \sum_{i=1}^{n} \Bigg[ \sum_{k=1}^{q_{i}} d_{ik}[\mb{\rho}] (\m{M}_{ik}^{K})^{\otimes 2} - \Big( \sum_{k=1}^{q_{i}} d_{ik}[\mb{\rho}] \m{M}_{ik}^{K} \Big)^{\otimes 2} \Bigg], 
\end{align*}
where $\m{M}_{ik}^{K} := \sum_{k'=1}^{q_i}[\mb{\Omega}_i]_{kk'}\mb{\Pi}_{ik'} \in \b{F}_{sa}^{q \times q}$.
\end{proposition}

It follows from \cref{prop:ker:tr:centre,prop:eff:data:cov} that the operator MSE of the QUARK estimator is given by:
\begin{align*}
    \mathrm{MSE}[\hat{\mb{\rho}}^{K} \vert \mb{\rho}] = \b{E} \left[\vertiii{\hat{\mb{\rho}}^{K} - \mb{\rho}}_{2}^{2} \vert \mb{\rho} \right] = \tr{\Cov(\hat{\mb{\rho}}^{K})} = \frac{\tr{\f{A}_{K} \f{S}_{K, \mb{\rho}} \f{A}_{K}}}{r}.
\end{align*}
Moreover, the following central limit theorem is trivial:
\begin{theorem}[CLT]\label{thm:quark:clt}
Under the assumptions of \cref{prop:eff:data:cov}, it holds that: 
\begin{align*}
    \sqrt{r}(\hat{\mb{\rho}}^{K}-\mb{\rho}) \xrightarrow{d}  N\left(\m{0}, \f{A}_{K} \f{S}_{K, \mb{\rho}} \f{A}_{K} \right),
\end{align*}
as $r\rightarrow\infty$, where $\f{A}_{K}$ is the kernel centered-inversion in \cref{prop:ker:tr:centre}.
\end{theorem}

Finally, we provide a Bennett-type non-asymptotic bound for the QUARK estimator:

\begin{theorem}[Directional Concentration Inequality]\label{thm:innpr:concen}
Assume the setup of \cref{thm:quark:supop:pd}. Then, for any $\m{S} \in \b{F}_{sa}^{q\times q}$ and $t>0$, the directional concentration inequality holds:
\begin{align*}
    \b{P} \left[\left. \innpr{\hat{\mb{\rho}}^{K} - \mb{\rho}}{\m{S}}_{2} >t \right\vert \mb{\rho}\right] \le \exp \left( 
    - \frac{\f M_{K,\mb \rho}[\m S]}{r} h \left( \frac{t}{\f M_{K,\mb \rho}[\m S]/ r} \right) 
    \right),
\end{align*}
where $h(x) := (1+x) \log(1+x) - x$ and $\f M_{K,\mb \rho}[\m S] = \innpr{\f{A}_{K} \f{S}_{K, \mb{\rho}} \f{A}_{K} \m{S}}{\m{S}}_{2}$.
\end{theorem}

Furthermore, using the inequality $h(\tfrac{t}{m_2})\geq \tfrac{t^2}{2(m_2+t/3)}$ for every $t>0$ and $m_2\geq 0$, the simpler Bernstein-type inequality follows from the previous result:
\begin{align*}
    \b{P} [ \innpr{\m{P}_{K} - \f{H}_{K} [\mb{\rho}]}{\m{S}}_{2}>t \vert \mb{\rho}] \leq \exp\left(-\frac{t^{2}}{2(m_{2} + \frac{t}{3})}\right),
\end{align*}
where
\begin{align*}
    m_{2} = \frac{1}{n^{2} r} \sum_{i=1}^{n} \Big[ \sum_{k=1}^{q_i} d_{ik}[\mb{\rho}]\,
\innpr{ \m{M}_{ik}^{K} }{\m{S}}_{2}^{2} - \big( \sum_{k=1}^{q_i} d_{ik}[\mb{\rho}]\,
\innpr{ \m{M}_{ik}^{K} }{\m{S}}_{2} \big)^{2} \Big].
\end{align*}

\section{Simulation Study}\label{sec:simstudy}
All simulations in this section are performed in the complex setting $\b{F}=\b{C}$ (see \cref{sec:supp_material_simulation_study} for $\b{F}=\b{R}$). 
Given a kernel $K$, the estimated eigenpairs for the LSE and QUARK estimators are respectively denoted by $(\hat \lambda_{j}^{\text{LSE}},\hat{\m{v}}_{j}^{\text{LSE}})$ and $(\hat \lambda_j^{K},\hat{\m v}_j^{K})$.

\subsection{Empirical Validation}\label{ssec:emp:val}
In this subsection, the QHS is fixed to $\c{Q}=\b{C}^{8}$, and the underlying density matrix is chosen to be diagonal:
$\mb{\rho} = \Diag (\lambda_{1}, \dots, \lambda_{8}) \in \b{S}(\b{C}_{+}^{8 \times 8})$, with eigenvalues 
$0.4 \geq 0.2 \geq 0.15 \geq 0.08 \geq 0.06 \geq 0.05 \geq 0.04 \geq 0.02$. Without loss of generality, the associated eigenvectors are the canonical basis vectors $\m{v}_{i}=(\delta_{ij})_{j=1}^{8}$. We sample a design of $n=100$ observables uniformly from $U(\c{Q})$ to approximate a unitary design (see \cref{rmk:approx:unit}). Throughout this study, we utilize the Gaussian radial basis function kernel $K_{c}(x,y) := \exp (-c \|x-y\|^{2})$ for $c>0$, and write $K:=K_{1}$ for brevity.

\vspace{1em} \noindent
\textbf{Spectral Behavior.}
We first investigate the estimation errors for the eigenvalues and eigenvectors based on $1000$ Monte Carlo repetitions, each with $r=50$ independent shots per observable.
\cref{fig:spectral_behavior_complex} reports the spectral estimation errors for both the LSE and QUARK estimator with respect to $K=K_{1}$. 
For both estimators, eigenvalue estimation errors remain of comparable magnitude across the spectrum. The same can be said for eigenvectors, with the notable exception that the leading eigenvector exhibits substantially reduced error and variance, similar to the Davis-Kahan theorem \cite{davis1970rotation,yu2015useful}: the largest eigenvalue possesses the largest spectral gap ($0.4-0.2=0.2$) compared to the second largest gap ($0.15-0.08=0.07$), rendering its associated eigenspace estimation significantly more stable. Comparing the two estimators, we observe that the LSE achieves steadier reconstruction error for the eigenvalues. Conversely, for the eigenvectors, the QUARK estimator achieves slightly lower variance compared to the LSE, while maintaining a comparable median error.

\begin{figure}[h!]
    \centering
    \begin{minipage}[t]{0.48\textwidth}
        \centering
        \includegraphics[width=\textwidth]{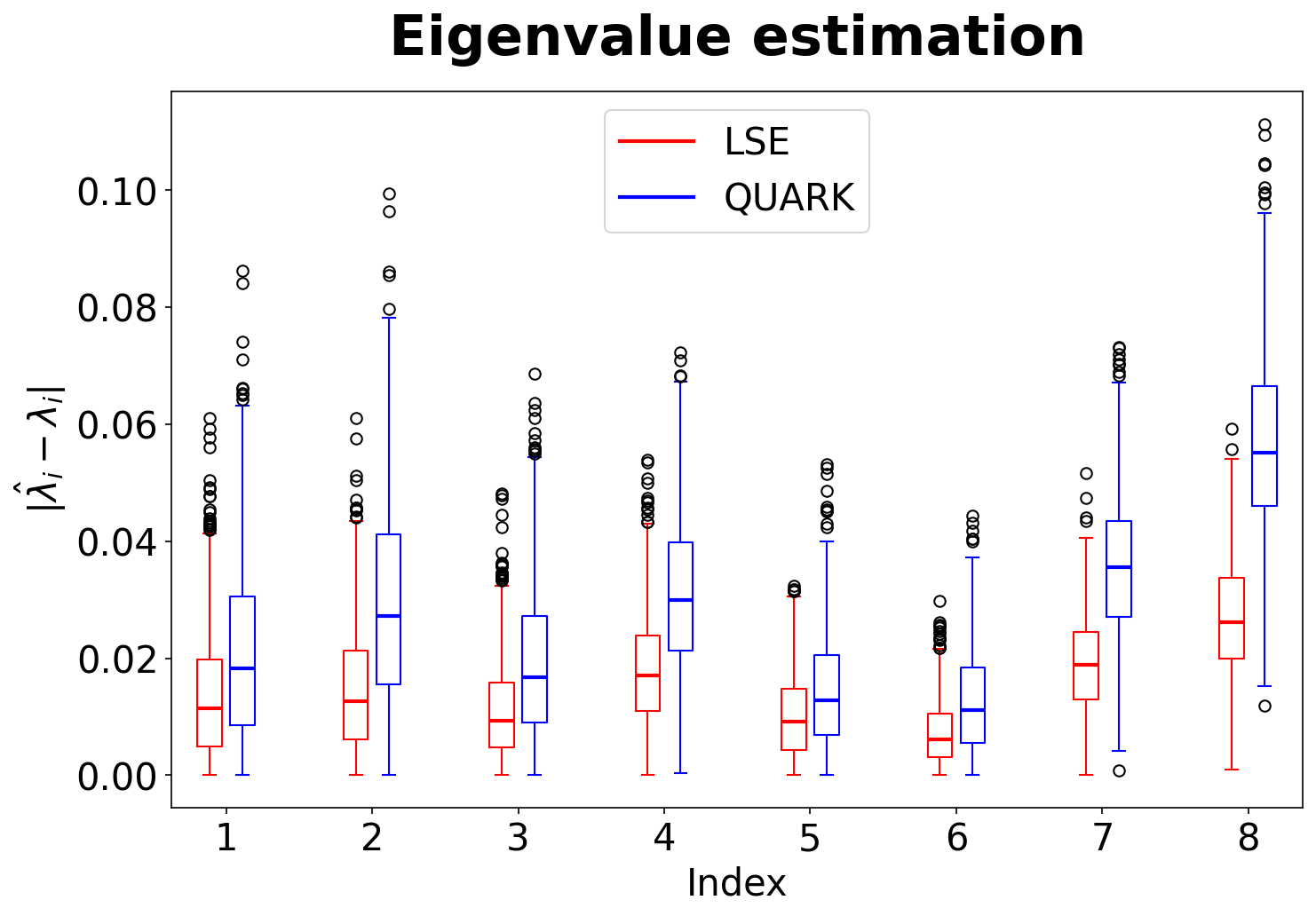}
    \end{minipage}
    \hfill
    \begin{minipage}[t]{0.48\textwidth}
        \centering
        \includegraphics[width=\textwidth]{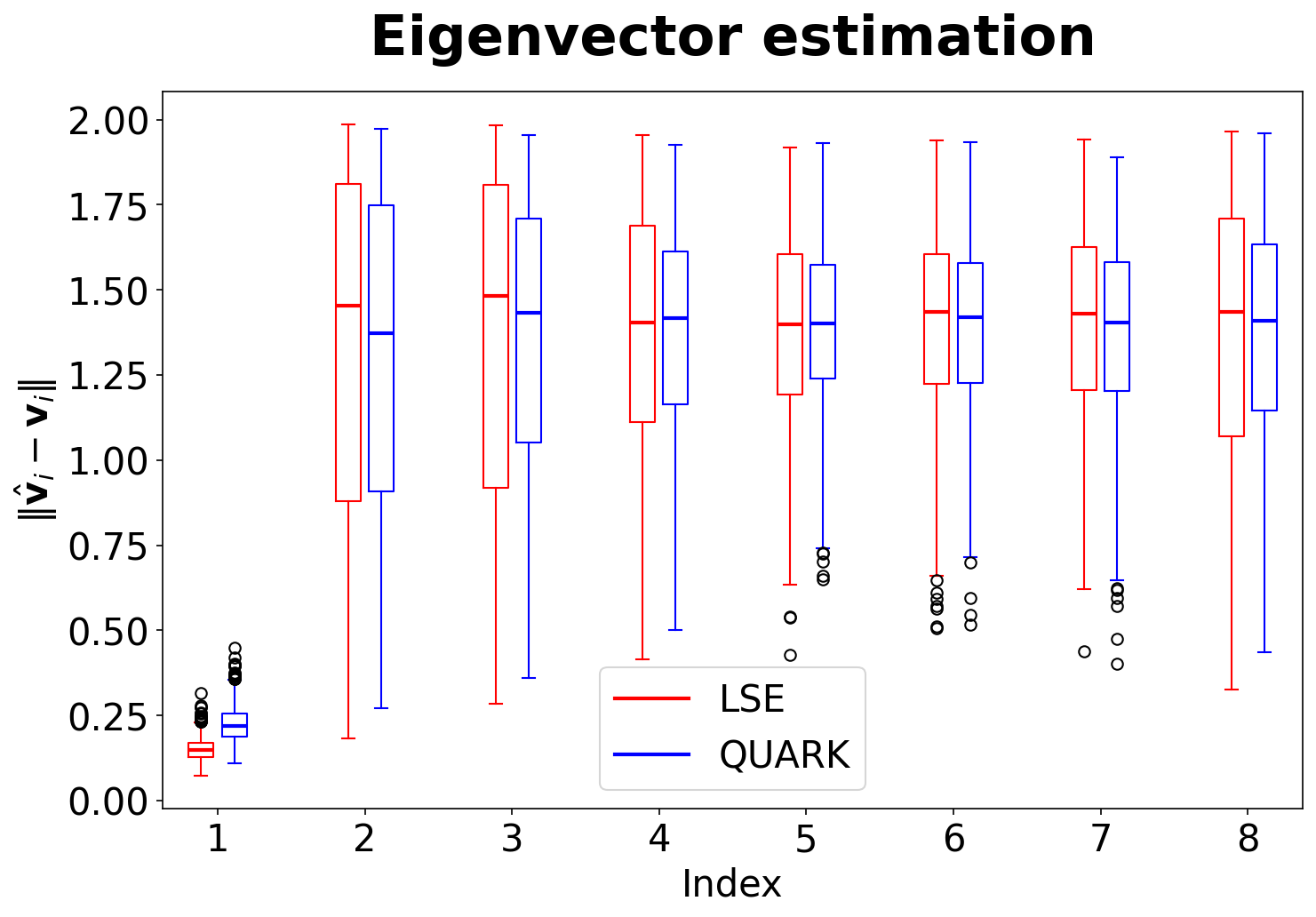}
    \end{minipage}
    \caption{Spectral behavior of the LSE and QUARK estimator with kernel $K$, based on $1000$ Monte Carlo repetitions. The boxplots are ordered by decreasing eigenvalue magnitude. \textbf{Left:} Absolute eigenvalue errors $|\hat{\lambda}^{\text{LSE}}_{i}-\lambda_{i}|$ (red) and $|\hat{\lambda}^{K}_{i}-\lambda_{i}|$ (blue). \textbf{Right:} Absolute eigenvector errors $\|\hat{\m{v}}^{\text{LSE}}_{i}-\m{v}_{i}\|$ (red) and $\| \hat{\m{v}}_{i}^{K}-\m{v}_{i}\|$ (blue).}
    \label{fig:spectral_behavior_complex}
\end{figure}

\noindent
\textbf{Asymptotic Performances.}
We study the asymptotic behavior of the estimators as the number of samples per observable $r$ increases, varying $r\in \{2^j:j=2,\dots,11\}$. 
We evaluate the reconstructions for the LSE (black), along with various QUARK estimators: $K$ (blue), $K_{100}$ (orange), $K_{0.01}$ (purple), and the polynomial kernel $K_{\text{poly}}(x,y)=(1+\langle x,y\rangle)^2$ (green). These empirical results are compared to the benchmark theoretical MSE under an exact unitary design (with damping factor $\alpha_{\c Q}=1/9$) established in \cref{cor:unit:mse}, denoted by the red dashed line. As shown in \cref{fig:mse_vs_r_complex}, the empirical behavior of the LSE closely tracks this theoretical scaling. While all QUARK estimators exhibit larger errors than the LSE, this gap diminishes as the bandwidth parameter $c$ increases. This reflects the fact that the Gaussian kernel approaches the $0$--$1$ kernel (and thus the LSE behavior) as $c \to \infty$. Nonetheless, \cref{fig:mse_vs_r_complex} validates the asymptotic decay rate predicted by the theory.

\begin{figure}[h!]
    \centering
    \includegraphics[width=\textwidth]{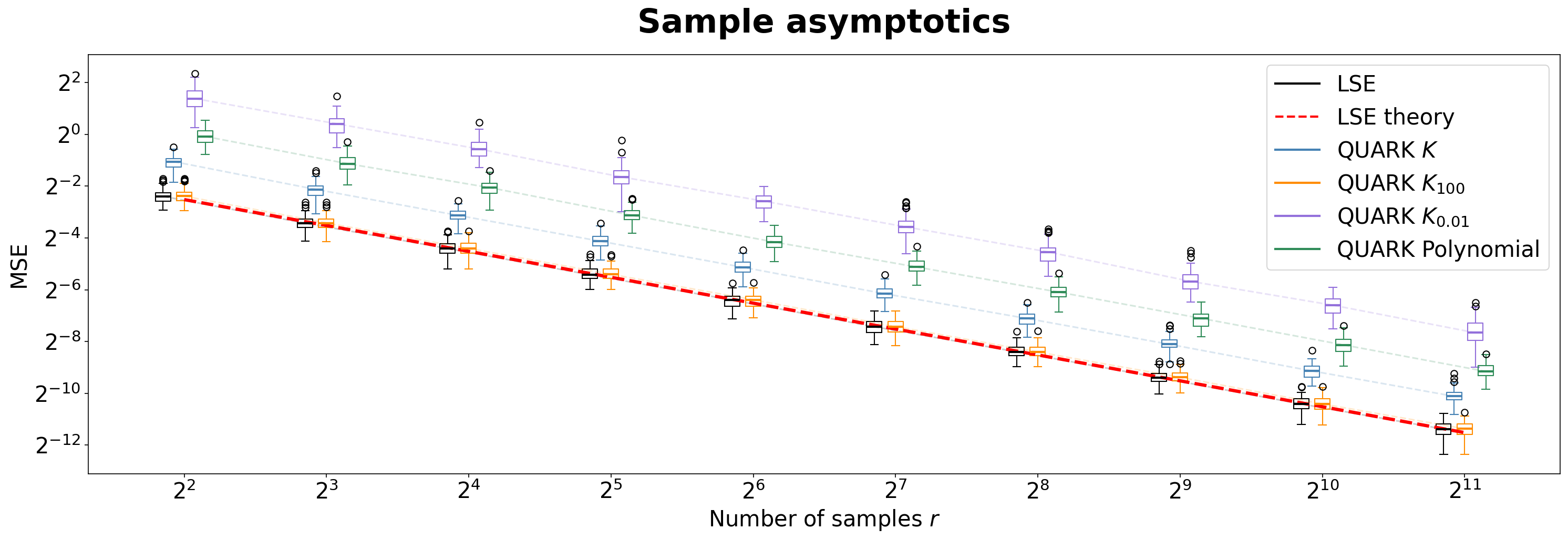}
    \caption{Asymptotic performances in terms of MSE of the LSE and various QUARK estimators.}
    \label{fig:mse_vs_r_complex}
\end{figure}

\subsection{Entangled MUB versus Local Pauli}\label{sec:simstudy_vary_q}
We compare the MUB design to Pauli design for a $k$-qubit system ($\c{Q} = \b{C}^{2^{k}}$). Standard approaches in the literature \cite{cai2016optimal, gross2010quantum, guctua2020fast} often employs the full set of $4^{k}$ tensorized Pauli observables:
\begin{align}\label{eq:tens:pauli}
    \mb{\sigma}_{(i_{1}, \dots, i_{k})} = \bigotimes_{j=1}^{k} \mb{\sigma}_{i_{j}}, \quad i_{j} \in \{0, x, y, z\}, \quad \mb{\sigma}_{0} = \m{I}_{2}.
\end{align}

We generate a true density matrix of low rank $3$, and perform $20$ Monte Carlo simulations per setup. The left panel of \cref{fig:comparison_to_other_estimators} fixes $k=6$ and compares the accuracy of our LSE estimator under a MUB design to the thresholding estimators under a Pauli design from \cite{cai2016optimal}, across varying sample sizes $r$. 
The right panel illustrates the scaling behavior when the number of qubits increases ($k=2, \cdots, 11$). For a fair comparison, we fix the total sample size $n \cdot r$ across setups. As shown, our LSE under a MUB design tightly adheres to the theoretical performance of an exact unitary design. It outperforms the Pauli-based estimators by orders of magnitude (note the log scale), while being computationally exponentially faster: $O(k \cdot 2^{k})$ via the fast WHT (\cref{ssec:mub}) versus $O(4^{k})$.

\begin{figure}[h!]
\begin{minipage}{0.48\textwidth}
\raggedright
\includegraphics[width=\linewidth]{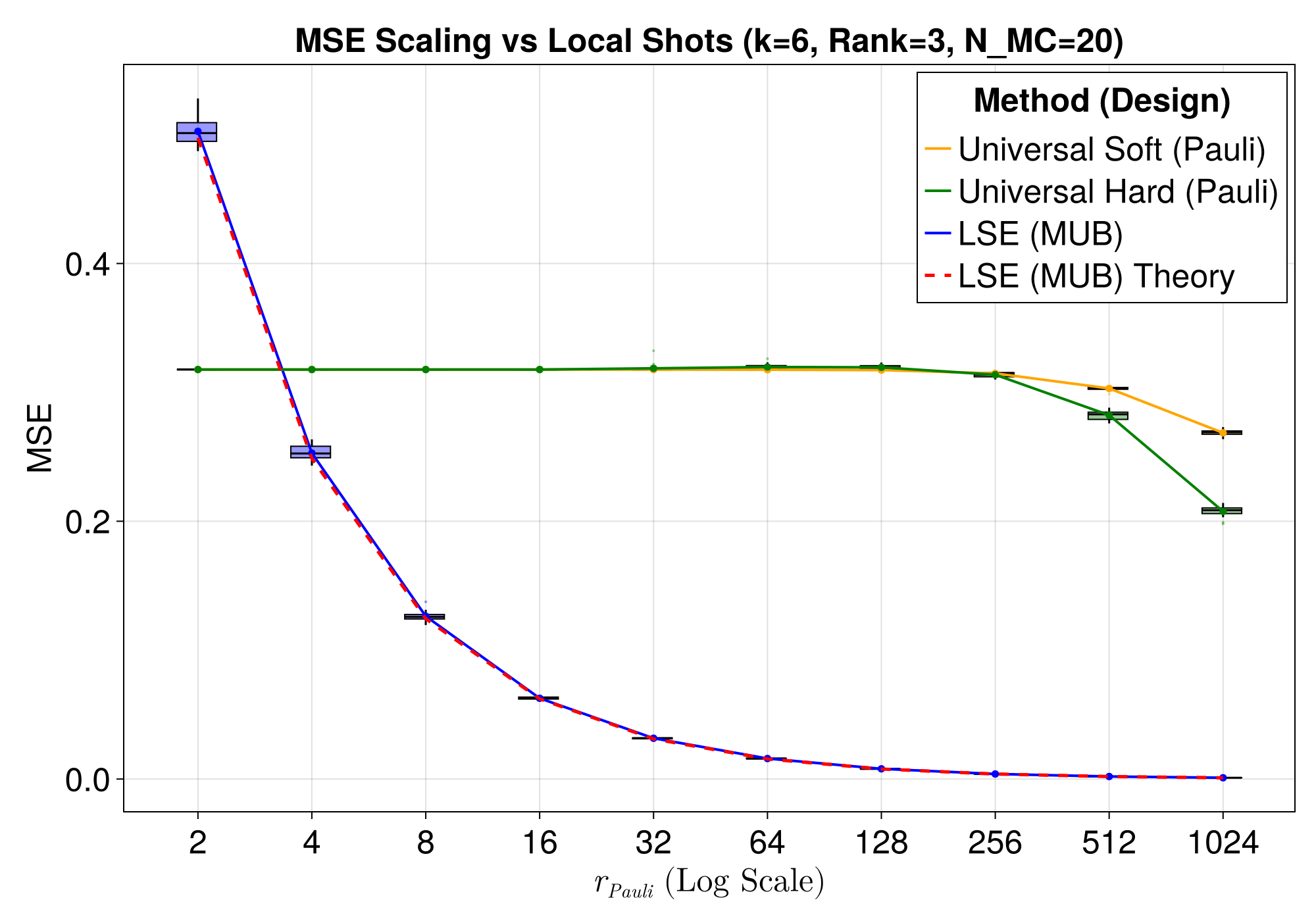}
\end{minipage}
\hfill
\begin{minipage}{0.48\textwidth}
\raggedleft
\includegraphics[width=\linewidth]{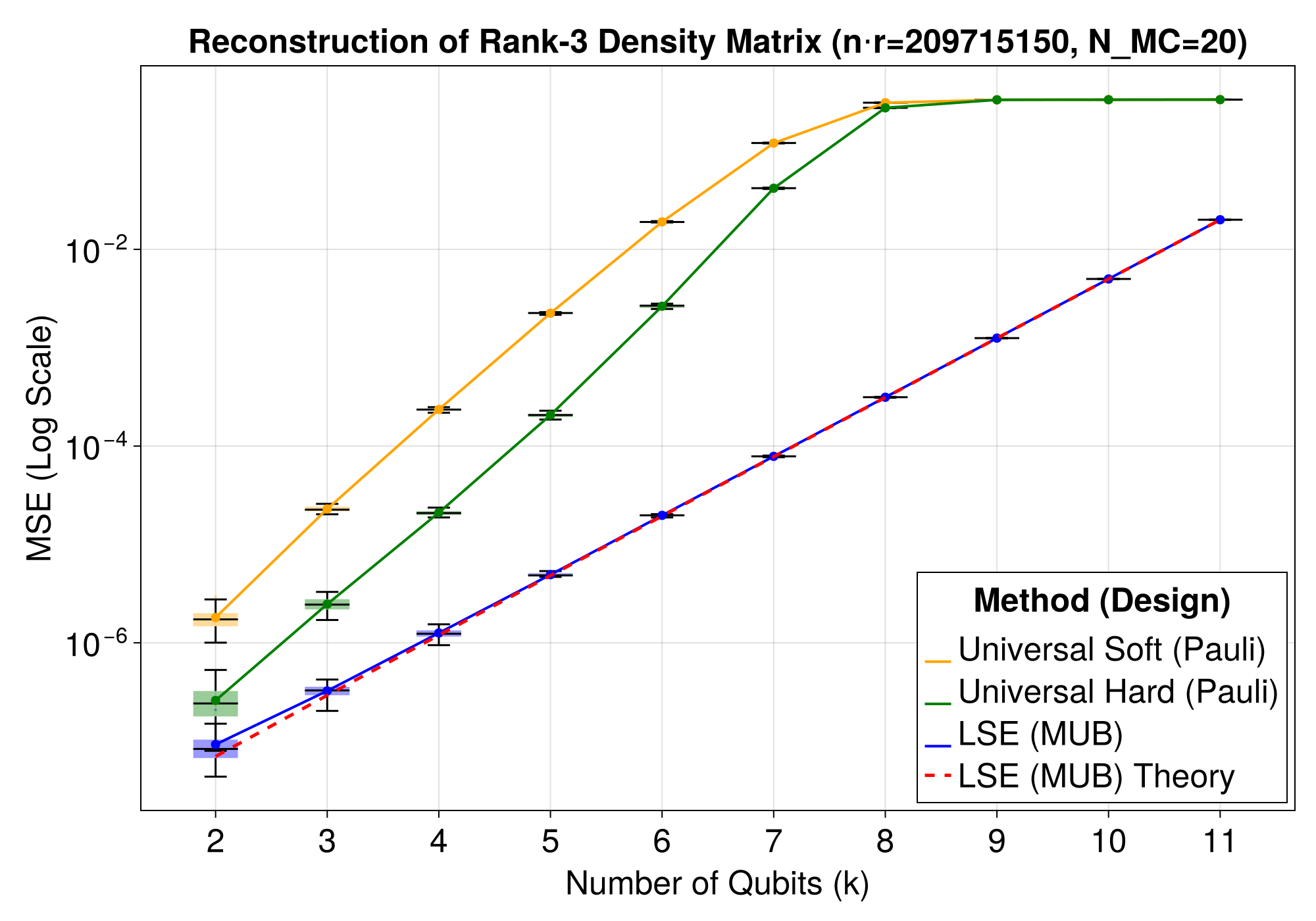}
\end{minipage}
\caption{Comparison between the proposed LSE estimator with MUB design and the Universal soft and hard thresholding estimators of \cite{cai2016optimal} with Pauli designs. Left: dependence on the number of samples per observable. Right: dependence on the dimension of the quantum Hilbert space.}
\label{fig:comparison_to_other_estimators}
\end{figure}

This performance gap is due to the high multiplicity inherent to Pauli observables: the identity yields a single trivial eigenvalue, while the remaining $4^{k}-1$ observables have spectra $\{\pm 1\}$ with multiplicity $2^{k-1}$. 
Crucially, for $k \ge 2$, the tensor product design in \eqref{eq:tens:pauli} performs only \emph{local} measurements and does not constitute a unitary design. To illustrate this, consider the trace-centering decomposition for $k=2$: the space $\b{C}_{sa, 0}^{4 \times 4}$ splits into two local sectors ($[\b{C}_{sa, 0}^{2 \times 2} \otimes \spann\{\m{I}_{2}\}]$ and $[\spann\{\m{I}_{2}\} \otimes \b{C}_{sa, 0}^{2 \times 2}]$) and one pure correlation sector ($[\b{C}_{sa, 0}^{2 \times 2} \otimes \b{C}_{sa, 0}^{2 \times 2}]$).
The Gram superoperator $\f{D}^{*} \f{D}/n$ acts diagonally on these subspaces with distinct eigenvalues. Since the local qubit design has $\alpha_{\b{C}^{2}} = 1/3$, the eigenvalues for the local sectors are $1/3$, whereas the eigenvalue for the correlation sector is the product $(1/3)^{2}$:
\begin{align*}
    \frac{\f{D}^{*} \f{D}}{n} \mb{\rho} = 1 \cdot \frac{\m{I}_{4}}{4} \oplus \frac{1}{3} \left[\mb{\Delta}_{A} \otimes \frac{\m{I}_{2}}{2}\right] \oplus \frac{1}{3} \left[\frac{\m{I}_{2}}{2} \otimes \mb{\Delta}_{B} \right] \oplus \frac{1}{9} \mb{\Delta}_{AB}, 
\end{align*}
where $\mb{\Delta}_{A} = \mathrm{Tr}_{B}[\mb{\rho}] - \m{I}_{2}/2$ and $\mb{\Delta}_{B} = \mathrm{Tr}_{A}[\mb{\rho}] - \m{I}_{2}/2$ are the traceless marginals, and $\mb{\Delta}_{AB}$ is the pure correlation tensor. This spectral gap widens exponentially with the number of qubits; the eigenvalue governing the $k$-body correlation sector is $(1/3)^{k}$. Consequently, repeating the analysis of \cref{cor:unit:mse} reveals that the MSE scales as $O(9^{k}/(nr))$ for a ``worst-case'' state dominated by global correlations.
In contrast, the MUB design is an $\alpha_{\c{Q}}$-unitary design with $\alpha_{\c{Q}}=(2^{k}+1)^{-1}$ achieves $O(4^{k}/(nr))$.

\section{Discussion}
This work establishes the kernel embedding theory to quantify how the geometry of the measurement process affects statistical efficiency. We demonstrate its utility for Quantum State Tomography, natively incorporating the least-squares approach under a metric-free topology. 

Our theory of unitary design and minimax analysis, providing an analytical bound for the penalty incurred by eigenvalue degeneracies, clarifies the statistical price of quantum state recovery. Crucially, our analysis sheds a light on the fundamental tension between hardware implementation and statistical optimality: while local measurement schemes (e.g., tensorized Pauli) may yield optimal accuracy given strict sparsity assumptions tailored to that design, they incur a severe, exponential variance penalty. Achieving fundamental statistical efficiency strictly requires entangled observables, such as Mutually Unbiased Bases.

In contrast, by leveraging the spectral geometry of the measurement process rather than design-dependent sparsity, the proposed QUARK estimator provides a minimax-optimal framework for density estimation even in the highly entangled regimes where genuine quantum supremacy is realized.

\begin{acks}[Acknowledgments]
Ho Yun is supported by a Swiss National Science Foundation granted to Victor Panaretos.
\end{acks}

\bibliographystyle{imsart-number}
\bibliography{bibliography}

\newpage

\begin{appendix}

\section{Tensorized Bochner integral}\label{sec:tens:boch}
The usual Bochner integral for a general Hilbert (or Banach) space with respect to a scalar probability measure is typically constructed via a limiting argument over step functions \cite{yosida2012functional, hsing2015theoretical}. Here, we develop a \emph{tensorized} Bochner integration where the measure is a Positive Operator-Valued Measure (POVM) and the Banach space of interest is the space of bounded operators on a Hilbert space, a formulation that natively supports our operator-theoretic framework.

Let $\c{H}_{1}$ and $\c{H}_{2}$ be Hilbert spaces. We define the inner product on elementary tensors by
\begin{equation*}
    \innpr{f_{1} \otimes f_{2}}{g_{1} \otimes g_{2}}_{\c{H}_{1} \otimes \c{H}_{2}} := \innpr{f_{1}}{g_{1}}_{\c{H}_{1}} \innpr{f_{2}}{g_{2}}_{\c{H}_{2}},  \quad f_{1}, \, g_{1} \in \c{H}_{1}, \ f_{2}, \, g_{2} \in \c{H}_{2}.
\end{equation*}
For $\m{T}_{1} \in \c{B}_{\infty}(\c{H}_{1})$ and $\m{T}_{2} \in \c{B}_{\infty}(\c{H}_{2})$, there exists a unique bounded operator, denoted by $\m{T}_{1} \otimes \m{T}_{2} \in \c{B}_{\infty}(\c{H}_{1} \otimes \c{H}_{2})$, such that
\begin{equation*}
    (\m{T}_{1} \otimes \m{T}_{2})(f_{1} \otimes f_{2}) = (\m{T}_{1} f_{1}) \otimes (\m{T}_{2} f_{2}),
\end{equation*}
and it holds that $\vertiii{\m{T}_{1} \otimes \m{T}_{2}}_{\infty} = \vertiii{\m{T}_{1}}_{\infty} \vertiii{\m{T}_{2}}_{\infty}$ \cite[Proposition A.67]{hall2013quantum}. Therefore, the quadratic form with respect to the elementary tensor entirely characterizes the tensorized operators. That is, for an operator $\m{T}: (\c{H}_{1} \otimes \c{H}_{2}) \to (\c{H}_{1} \otimes \c{H}_{2})$, it holds that $\m{T} = \m{T}_{1} \otimes \m{T}_{2}$ if and only if
\begin{align*}
    \innpr{(\m{T}_{1} \otimes \m{T}_{2}) (f_{1} \otimes f_{2})}{g_{1} \otimes g_{2}}_{\c{H}_{1} \otimes \c{H}_{2}} = \innpr{\m{T}_{1} f_{1}}{g_{1}}_{\c{H}_{1}} \innpr{\m{T}_{2} f_{2}}{g_{2}}_{\c{H}_{2}},
\end{align*}
for any $f_{1}, g_{1} \in \c{H}_{1}$ and $f_{2}, g_{2} \in \c{H}_{2}$. Hence, the formal definition below is well-posed:

\begin{definition}
Let $\c{Q}, \c{H}$ be Hilbert spaces. Let $(X, \s{B}(X))$ be an LCH space equipped with Borel $\sigma$-algebra, and $\mb{\nu} \in \s{P}(X, \c{Q})$ be a POVM. Consider the locally convex topological vector space $\c{B}_{\infty}(\c{H})$ equipped with the Weak Operator Topology (WOT). A measurable function $\m{F}: X \to \c{B}_{\infty}(\c{H})$ is said to be \textit{tensorized Bochner integrable} if there exists a bounded 
operator on $\c{H} \otimes \c{Q}$, denoted by $\int_{X} \m{F}(x) \otimes \rd \mb{\nu}(x)$, such that
\begin{align}\label{eq:tens:Boch}
    \innpr{\left( \int_{X} \m{F}(x) \otimes \rd \mb{\nu}(x) \right) (f \otimes \phi)}{g \otimes \psi}_{\c{H} \otimes \c{Q}} = \int_{X} \innpr{\m{F}(x) f}{g}_{\c{H}} \rd \mb{\nu}_{\phi, \psi} (x),
\end{align}
for any $f, g \in \c{H}$ and $\phi, \psi \in \c{Q}$. We call $\int_{X} \m{F}(x) \otimes \rd \mb{\nu}(x)$ the tensorized Bochner integral of $\m{F}$ with respect to $\mb{\nu}$.
\end{definition}

In the case where $\c{H} = \b{F}$ or $\c{Q} = \b{F}$, this naturally reduces to the standard Bochner integral and functional calculus, respectively, where it is customary to write $\int_{X} \m{F}(x) \rd \mb{\nu}(x)$ instead of $\int_{X} \m{F}(x) \otimes \rd \mb{\nu}(x)$.
Unlike the standard Bochner integral, we employ a topology strictly weaker than the norm topology. The WOT is sufficient to uniquely determine the integrated operator via its associated quadratic form. Furthermore, this choice allows us to leverage the Banach--Alaoglu theorem, which guarantees that the closed unit ball in $(\c{B}_{\infty}(\c{H}), \vertiii{\cdot}_{\infty})$ is compact with respect to the WOT. Consequently, if $\m{F}$ is uniformly bounded, i.e., $\sup_{x \in X} \vertiii{\m{F}(x)}_{\infty} < \infty$, then it is automatically tensorized Bochner integrable since \eqref{eq:tens:Boch} remains bounded for all $f, g \in \c{H}$ and $\phi, \psi \in \c{Q}$. In this bounded regime, we obtain the inequality: $\vertiii{\int_{X} \m{F} \otimes \rd \mb{\nu}}_{\infty} \le \sup_{x \in X} \vertiii{\m{F}(x)}_{\infty}$. For a measurable simple function
\begin{equation*}
    \m{S}(x) = \sum_{i=1}^{n} \ds{1}_{B_{i}}(x) \m{S}_{i}, \quad B_{i} \in \s{B}(X), \quad \m{S}_{i} \in \c{B}_{\infty}(\c{H}), \quad n \in \b{N},
\end{equation*}
its tensorized Bochner integral is trivially given by $\int_{X} \m{S}(x) \otimes \rd \mb{\nu}(x) = \sum_{i=1}^{n} \m{S}_{i} \otimes \mb{\nu}(B_{i}) \in \c{B}_{\infty}(\c{H} \otimes \c{Q})$, which is independent of the particular representation of the simple function $\m{S}$.
Alternatively, one can construct the tensorized Bochner integral globally using a standard limiting argument:

\begin{theorem}[Lebesgue Dominated Convergence]\label{thm:tens:lebesgue}
Let $\c{Q}$ be a Hilbert space, and $\mb{\nu} \in \s{P}(X, \c{Q})$. An operator-valued measurable function $\m{F}: X \to (\c{B}_{\infty}(\c{H}), \mathrm{WOT})$ is tensorized Bochner integrable if there exists a sequence of measurable simple functions $\{\m{S}_{n}: X \to (\c{B}_{\infty}(\c{H}), \mathrm{WOT}) \}_{n \in \b{N}}$ such that:
\begin{enumerate}
    \item There exists a scalar function $h \in \c{L}_{1}(X, \mb{\nu})$ (meaning $h$ is integrable with respect to the scalar positive measures $\mb{\nu}_{\phi, \phi}$ for all $\phi \in \c{Q}$) satisfying $\vertiii{\m{S}_{n}(x)}_{\infty} \le h(x)$ for any $n \in \b{N}$ and $x \in X$.
    \item For each $x \in X$, $\m{S}_{n}(x) \to \m{F}(x)$ as $n \to \infty$ with respect to the WOT in $\c{B}_{\infty}(\c{H})$.
\end{enumerate}
In this case, we have $\vertiii{\m{F}(x)}_{\infty} \le h(x)$ for any $x \in X$, and
\begin{align*}
    \int_{X} \m{F}(x) \otimes \rd \mb{\nu}(x) = \lim_{n \to \infty} \int_{X} \m{S}_{n}(x) \otimes \rd \mb{\nu}(x),
\end{align*}
where the limit converges with respect to the WOT in $\c{B}_{\infty}(\c{H} \otimes \c{Q})$.
\end{theorem}

\begin{proof}
First, it is trivial that $\vertiii{\m{F}(x)}_{\infty} \le h(x)$ for any $x \in X$, since for any $f, g \in \c{H}$,
\begin{equation*}
    |\innpr{\m{F}(x) f}{g}_{\c{H}}| = \lim_{n \to \infty} |\innpr{\m{S}_{n}(x) f}{g}_{\c{H}}| \le h(x) \|f\|_{\c{H}} \|g\|_{\c{H}}.
\end{equation*}
Fix $f, g \in \c{H}$ and $\phi, \psi \in \c{Q}$. Because $|\innpr{\m{S}_{n}(x) f}{g}_{\c{H}}| \le h(x) \|f\|_{\c{H}} \|g\|_{\c{H}}$ and $h \in \c{L}_{1}(X, \mb{\nu}_{\phi, \psi})$, we can apply the standard Lebesgue dominated convergence theorem to obtain:
\begin{align*}
    \int_{X} \innpr{\m{F}(x) f}{g}_{\c{H}} \rd \mb{\nu}_{\phi, \psi} (x) &= \lim_{n \to \infty} \int_{X} \innpr{\m{S}_{n}(x) f}{g}_{\c{H}} \rd \mb{\nu}_{\phi, \psi} (x) \\
    &= \lim_{n \to \infty} \innpr{\left( \int_{X} \m{S}_{n}(x) \otimes \rd \mb{\nu}(x) \right) (f \otimes \phi)}{g \otimes \psi}_{\c{H} \otimes \c{Q}}.
\end{align*}
It suffices to show that the bounding limit is well-behaved for all $f, g \in \c{H}$ and $\phi, \psi \in \c{Q}$. Define the bounded positive operator $\m{H} := \int_{X} h(x) \rd \mb{\nu}(x) \in \c{B}_{\infty}^{+}(\c{Q})$, which is well-defined by the assumption $h \in \c{L}_{1}(X, \mb{\nu})$. By the Cauchy-Schwarz inequality for POVMs, the total variation of the complex measure satisfies $\rd |\mb{\nu}_{\phi, \psi}|(x) \le \sqrt{\rd\mb{\nu}_{\phi, \phi}(x) \rd\mb{\nu}_{\psi, \psi}(x)}$. 

Therefore, applying the Cauchy-Schwarz inequality for integrals, for any $n \in \b{N}$ we have:
\begin{align*}
    &\left| \int_{X} \innpr{\m{S}_{n}(x) f}{g}_{\c{H}} \rd \mb{\nu}_{\phi, \psi} (x) \right| 
    \le \|f\|_{\c{H}} \|g\|_{\c{H}} \int_{X} h(x) \rd |\mb{\nu}_{\phi, \psi}| (x) \\ 
    &\le \|f\|_{\c{H}} \|g\|_{\c{H}} \left( \int_{X} h(x) \rd \mb{\nu}_{\phi, \phi} (x) \right)^{1/2} \left( \int_{X} h(x) \rd \mb{\nu}_{\psi, \psi} (x) \right)^{1/2} \\
    &= \|f\|_{\c{H}} \|g\|_{\c{H}} \innpr{\m{H} \phi}{\phi}_{\c{Q}}^{1/2} \innpr{\m{H} \psi}{\psi}_{\c{Q}}^{1/2} \le \|f\|_{\c{H}} \|g\|_{\c{H}} \vertiii{\m{H}}_{\infty} \|\phi\|_{\c{Q}} \|\psi\|_{\c{Q}}.
\end{align*}
This uniform bound over all $n$ strictly guarantees that the sesquilinear form defines a bounded operator on the tensor product space, completing the proof.
\end{proof}

\section{Algorithms for Quantum Tomography}

\subsection{Trace-Preserving Projection}\label{sec:tpp}

\begin{theorem}[Sum-Preserving Threshold Algorithm (SPTA)]\label{thm:SPTA}
Let $\m{a} = (a_{1}, \cdots, a_{q})$ be a non-increasing real-valued sequence with $a_{1} +  \cdots + a_{q} > 0$. Consider the following constrained optimization problem:
\begin{equation*}
    \min \sum_{k=1}^{q} (a_{k} - b_{k})^{2} \quad \text{ w.r.t. } \quad b_{1} \ge  \cdots \ge b_{q} \ge 0, \text{ and } b_{1} +  \cdots + b_{q} = a_{1} +  \cdots + a_{q}.
\end{equation*}
Define the truncation index and the truncated negative tail average by
\begin{align*}
    t := t(\m{a}) = \max \{ 1 \le k \le q : k a_{k} + a_{k+1} + \dots + a_{q} \ge 0 \}, \,
    v := v(\m{a}) = - \frac{a_{t+1} + \cdots + a_{q}}{t}.
\end{align*}
Then, we have $0 \le v \le a_{t}$, and the unique solution of the primal problem is achieved at $\mathrm{SPTA}(\m{a}):= \m{b}^{*} = (b^{*}_{1}, \cdots, b^{*}_{q})$ with
\begin{equation*}
    b^{*}_{k} = \begin{cases}
        a_{k} - v &, \quad k = 1, 2, \cdots, t, \\
        0 &, \quad k = t+1, t+2, \cdots, q.
    \end{cases}
\end{equation*}
\end{theorem}
\begin{proof}
We may assume that $a_{1} +  \cdots + a_{q} = 1$, as we can normalize the problem. We may further assume that $a_{q} < 0$, otherwise the result is trivial as $t = q$, $v = 0$ and $\m{b}^* = \m{a}$. Consider the sequence
\begin{equation*}
    s_{k} := k a_{k} + a_{k+1} + \dots + a_{q}, \quad k = 1, 2, \cdots, q.
\end{equation*}
Then $\{s_{k} : k = 1, 2, \cdots, q \}$ is a monotonically non-increasing sequence with $s_{1} = 1$, since $s_{k} - s_{k+1} = k (a_{k} - a_{k+1}) \ge 0$. This shows that $t := t(\m{a}) = \max \{ 1 \le k \le q : s_{k} \ge 0 \}$ is well-defined. Also, the definition of $t$ yields $v = - (a_{t+1} + \cdots + a_{q})/t \le a_{t}$. Note that 
\begin{equation*}
    (t+1) a_{t+1} + a_{t+2} + \dots + a_{q} < 0 \quad \iff \quad v = - \frac{a_{t+1} + \cdots + a_{q}}{t} > a_{t+1},
\end{equation*}
hence $v \ge 0$. Otherwise, we have $a_{t+1} < 0$, but the monotonicity of $\m{a}$ yields $v > 0$, which is a contradiction.

We now show that the sequence $\m{b}^{*}$ satisfies the KKT conditions for the dual problem \cite{boyd2004convex}:
\begin{align*}
    L(\m{b}, \m{u}, v) = \frac{1}{2} \sum_{k=1}^{q} (a_{k} - b_{k})^{2} + \sum_{k=1}^{q} u_{k}(-b_{k} + b_{k+1}) + v (\m{1}^{\top} \m{b} - 1),
\end{align*}
where $b_{q+1} := 0$. Let $v^{*} := v$, $u_{0} := 0$, and $u^{*}_{k} := k v - \sum_{i=1}^{k} (a_{i} - b^{*}_{i})$ for $k = 1, 2, \cdots, q$.

\begin{enumerate}
\item (Stationarity) For each $k = 1, \cdots, q$,
\begin{equation*}
    \frac{\partial L(\m{b}, \m{u}^{*}, v^{*})}{\partial b_{k}} \Big\vert_{b_{k} = b^{*}_{k}} = - (a_{k} - b^{*}_{k}) - (u^{*}_{k} - u^{*}_{k-1}) + v^{*} = 0.
\end{equation*}

\item (Primal Feasibility) We have
\begin{align*}
    b^{*}_{1} \ge \cdots \ge b^{*}_{t} = a_{t} - v \ge 0 = b^{*}_{t+1} = \cdots = b^{*}_{q}, \, \sum_{k=1}^{q} b^{*}_{k} = \sum_{k=1}^{t} a_{k} - tv = \sum_{k=1}^{q} a_{k} = 1.
\end{align*}
\item (Dual Feasibility) For $k = 1, \cdots, t$, $u^{*}_{k} = 0$. Also, we have $u^{*}_{q} = q v \ge 0$ and
\begin{align*}
    u^{*}_{t+1} = u_{t}^{*} + (v - a_{t+1} + b^{*}_{t+1}) = v - a_{t+1} = - \frac{(t+1) a_{t+1} + a_{t+2} + \cdots + a_{q}}{t} \ge 0,
\end{align*}
by the construction of $t$, which yields 
\begin{align*}
    u^{*}_{k} = k v - \sum_{i=1}^{k} (a_{i} - b^{*}_{i}) = (k-t)v - \sum_{i=t+1}^{k} a_{i} \ge (k-t) (v -  a_{t+1}) = (k-t) u^{*}_{t+1} \ge 0,
\end{align*}
for $k = t+2, \cdots, q-1$.

\item (Complementary Slackness) For $k = 1, \cdots, t$, we have $u^{*}_{k}(-b^{*}_{k} + b^{*}_{k+1}) = 0$ since $u^{*}_{k} = 0$. For $k = t+1, \cdots, q$, we also have $u^{*}_{k}(-b^{*}_{k} + b^{*}_{k+1}) = 0$ since $b^{*}_{k} = b^{*}_{k+1} = 0$.
\end{enumerate}
\end{proof}

\begin{algorithm}[h!]
\caption{Sum-Preserving Threshold Algorithm (SPTA)}\label{alg::spta}
\begin{algorithmic}[1]
\Require $\m{a} = (a_{1}, \cdots, a_{q}) \in \mathbb{R}^{q}$ with $a_{1} \ge \cdots \ge a_{q}$ and $a_{1} +  \cdots + a_{q} > 0$
\Ensure $\m{b}= (b_{1}, \cdots, b_{q}) \in \mathbb{R}^{q}$ with $b_{1} \ge \cdots \ge b_{q} \ge 0$ and $\m{1}^{\top} \m{b} = \m{1}^{\top} \m{a}$
\Procedure{SPTA}{$\m{a}$}
    \If{$a_q \ge 0$} \Comment{If all elements are non-negative, $\m{a}$ is the solution} 
        \State \textbf{return} $\m{a}$
    \EndIf

    \State $\text{Sum} \gets 0$ \Comment{Stores the running sum of the tail, $a_{k+1} + \dots + a_{q}$}
    \State $k \gets q$
    
    \Statex \Comment{\textbf{Phase 1:} Efficiently scan past the negative tail.}
    \While{$a_{k} \le 0$}
        \State $\text{Sum} \gets \text{Sum} + a_k$
        \State $k \gets k-1$
    \EndWhile

    \Statex \Comment{\textbf{Phase 2:} Find the true truncation index, starting from the first positive value.}
    \While{$k \cdot a_{k} + \text{Sum} < 0$}
        \State $\text{Sum} \gets \text{Sum} + a_{k}$
        \State $k \gets k-1$
    \EndWhile

    \State $\m{b} \gets (0, 0, \cdots, 0)$ \Comment{Initialize as a zero vector.}
    
    \State $\text{Shift} \gets \text{Sum}/k$ \Comment{Shift is $-v$, where $v$ is the corrective offset}
    \For{$i \gets 1$ \textbf{to} $k$}
        \State $b_{i} \gets a_{i} + \text{Shift}$ \Comment{Apply the shift to the head of the vector}
    \EndFor
    \State \textbf{return} $\m{b}$ \Comment{The tail of $\m{b}$ remains zero}
\EndProcedure
\end{algorithmic}
\end{algorithm}

In the case where $\m{a}$ contains only a small number of negative values, it is computationally more efficient to compute $s_{k}$ backwards, starting from the first non-negative value. The SPTA algorithm can readily be used to project a self-adjoint matrix onto the convex set of n.n.d. matrices with respect to the Frobenius norm:

\begin{table}[h!]
\centering
\caption{Execution trace of SPTA (\cref{alg::spta}) with $q=7$ and $\m{a} = (1.1, 0.3, 0.1, 0.1, -0.1, -0.2, -0.3)$.}
\label{tab:spta_optimized_trace}
\begin{tabular}{|l|c|c|c|l|c|}
\hline
\textbf{Phase} & $k$ & \textbf{$a_k$} & \textbf{Sum} & \textbf{Condition Check} & \textbf{Result} \\ \hline
\multicolumn{6}{|c|}{\textbf{Phase 1: Scan non-positive tail ($a_k \le 0$)}} \\ \hline
Start & 7 & -0.3 & 0 & $a_7 \le 0$ & True \\
Iter 1 & 6 & -0.2 & -0.3 & $a_6 \le 0$ & True \\
Iter 2 & 5 & -0.1 & -0.5 & $a_5 \le 0$ & True \\
Iter 3 & 4 & 0.1 & -0.6 & $a_4 \le 0$ & \textbf{False} \\ \hline
\multicolumn{6}{|c|}{\textbf{Phase 2: Find truncation index ($k \cdot a_k + \text{Sum} < 0$)}} \\ \hline
Start & 4 & 0.1 & -0.6 & $4(0.1)+(-0.6) = -0.2$ & True \\
Iter 1 & 3 & 0.1 & -0.5 & $3(0.1)+(-0.5) = -0.2$ & True \\
Iter 2 & 2 & 0.3 & -0.4 & $2(0.3)+(-0.4) = 0.2$ & \textbf{False} \\ \hline
\multicolumn{6}{|c|}{\textbf{Final Calculation}} \\ \hline
\multicolumn{6}{|p{12cm}|}{ \textbf{Loop terminates with} $k=2$ and $\texttt{Sum}=-0.4$. \newline \textbf{Calculate Shift:} $\texttt{Shift} = \texttt{Sum} / k = -0.2$. \newline 
\textbf{Construct $\m{b}$:} $b_1 = 1.1 - 0.2 = 0.9$, $b_2 = 0.3 - 0.2 = 0.1$. The rest are 0. \newline 
\textbf{Result:} $\m{b} = (0.9, 0.1, 0, 0, 0, 0, 0)$. } \\ \hline
\end{tabular}
\end{table}

\begin{theorem}[Trace-preserving projection]
Let $\m{S} = \m{U} \Diag[\m{a}] \m{U}^{*} \in \b{F}_{sa}^{q \times q}$ be the spectral decomposition of a self-adjoint matrix, where $\m{U}$ is unitary (orthogonal if $\b{F} = \b{R}$) matrix and $\m{a} = (a_1, \dots, a_q)$ is the vector of its real eigenvalues, sorted in non-increasing order.
Then the optimization problem
\begin{equation*}
    \argmin_{\m{T} \in \b{F}_{sa}^{q \times q}} \vertiii{\m{S} - \m{T}}_{2}^{2} \quad \text{ w.r.t. } \quad \m{T} \succeq \m{0} \quad \text{and} \quad \tr{\m{T}} = \tr{\m{S}},
\end{equation*}
has a unique solution given by $\m{U} \Diag[\mathrm{SPTA}(\m{a})] \m{U}^{*}$.
\end{theorem}
\begin{proof}
Using spectral decomposition, we solve the optimization problem in two steps:
\begin{equation*}
    \min_{\m{D} = \Diag[\m{b}]} \left( \min_{\m{V} \in SU(q)} \vertiii{\m{V} \m{D} \m{V}^{*} - \m{S}}_{2}^{2} \right) \quad \text{ w.r.t. } \quad \m{D} \succeq \m{0} \text{ and } \tr{\m{D}} = \tr{\m{S}},
\end{equation*}
where $\m{b}$ is also sorted in non-increasing order. For a fixed diagonal matrix $\m{D} \succeq \m{0}$ with $\tr{\m{D}} = \tr{\m{S}}$, the inner optimization problem becomes
\begin{align*}
    \min_{\m{V} \in SU(q)} \vertiii{\m{V} \m{D} \m{V}^{*} - \m{S}}_{2}^{2} &= \vertiii{\m{D}}_{2}^{2} + \vertiii{\m{S}}_{2}^{2} - 2 \max_{\m{V} \in SU(q)} \innpr{\m{V} \m{D} \m{V}^{*}}{\m{S}}_{2}
\end{align*}
by the von-Neumann trace inequality \cite{horn2012matrix}, and equality holds if and only if $\m{U} = \m{V}$. Consequently, the outer optimization problem simplifies to \cref{thm:SPTA}.
\end{proof}

\subsection{Tomography with Mutually Unbiased Bases}\label{ssec:mub}
Let $GF(2^{k})$ be the Galois Field of order $q=2^{k}$ equipped with operations $(\bullet, +)$. We view elements of $GF(2^k)$ strictly as polynomials modulo an irreducible polynomial $p_{k}(x)$ of degree $k$. To identify field operations with vector operations, we assume the choice of a \emph{self-dual basis} for $GF(2^k)$. Under this basis, the field trace creates an isomorphism to the standard Euclidean dot product:
\begin{align*}
    \text{Tr}_{GF}(\m{a} \bullet \m{b}) = \m{a}^{\top}\m{b} = \sum_{j=1}^{k} a_j b_j \pmod 2, \quad \m{a}, \m{b}, \m{c}, \m{d} \in \{0, 1\}^{k} 
\end{align*}
For instance, $(011) \bullet (111) = (010)$, because $(x+1)(x^{2}+x+1) \equiv x \pmod{p_{3}}$, where the arithmetic is modulo the irreducible polynomial $p_{3}(x)=x^{3}+x+1$.
The phase space is the vector space $GF(2^k) \times GF(2^k)$, and its elements are denoted by $\m{u} = (\m{a} | \m{b})$. Under the self-dual basis, we can define the \textbf{symplectic inner product} consistently for both the field and the vector space: \begin{equation} \innpr{(\m{a}|\m{b})}{(\m{c}|\m{d})}_{sp} = \m{a}^{\top} \m{d} + \m{b}^{\top} \m{c} \pmod 2. \end{equation}
This phase space corresponds to the Pauli operators of the $k$-qubit system $\c{Q} = \b{C}^{q}$, having size $q^{2} = 2^{2k}$:
\begin{equation*}
    \m{I}_{2} = \begin{pmatrix}
        1 &0 \\ 0 &1
    \end{pmatrix}, \quad
    \m{X} = \begin{pmatrix}
        0 &1 \\ 1 &0
    \end{pmatrix}, \quad
    \m{Y}^{phy} = \begin{pmatrix}
        0 &-\imath \\ \imath &0
    \end{pmatrix}, \quad
    \m{Z} = \begin{pmatrix}
        1 &0 \\ 0 &-1
    \end{pmatrix} \in \b{C}_{sa, 0}^{2 \times 2}.
\end{equation*}
We define the \emph{mathematical} Pauli Y matrix by $ \m{Y}^{math} = \m{X} \m{Z}$ so that $\m{Y}^{phy} = \imath \m{Y}^{math}$. For the $k$-qubit system $\c{Q} = \b{C}^{q}$ with $q=2^{k}$, the \emph{mathematical} tensorized Pauli operators can be identified with the symplectic vector space $(\b{Z}_{2}^{2k}, \innpr{\cdot}{\cdot}_{sp})$:
\begin{align*}
    \m{P}_{\m{u}}^{math} = \bigotimes_{j=1}^{k} \m{X}_{j}^{a_{j}} \m{Z}_{j}^{b_{j}}, \quad \m{u} = (\m{a}|\m{b}).
\end{align*}
For instance, as $\m{I}_{2} \leftrightarrow (0 | 0)$, $\m{X} \leftrightarrow (1 | 0)$, $\m{Y}^{math} \leftrightarrow (1 | 1)$, $\m{Z} \leftrightarrow (0 | 1)$, the vector $\m{u} = (1010 \vert 0011)$ determines $\m{P}_{\m{u}}^{math} = \m{X} \m{I}_{2} \m{Y}^{math} \m{Z}$. In this association, the commutative relationship between Pauli operators can be written as
\begin{align}\label{eq:phase:corr}
    \m{P}_{\m{u}} \m{P}_{\m{v}} = \imath^{\innpr{\m{u}}{\m{v}}_{sp}} \m{P}_{\m{u} + \m{v}} = (-1)^{\innpr{\m{u}}{\m{v}}_{sp}} \m{P}_{\m{v}} \m{P}_{\m{u}} \, \Rightarrow \, \innpr{\m{u}}{\m{v}}_{sp} = \begin{cases}
        0 &, \quad \text{commute}, \\
        1 &, \quad \text{anti-commute},
    \end{cases}
\end{align}
and here the choice of physical or mathematical convention does not matter for commutativity.

A clique is a maximal subgroup of mutually commutative Pauli operators. There are $q+1$ cliques consisting of $q-1$ traceless Pauli operators and the identity $\m{I}_{q}$. Each clique is associated with a label $\mb{\lambda} \in \{0, 1\}^{k} \cup \{\infty\}$. The generators for these cliques are derived from the field structure. In the bitwise representation with $\m{e}_{0} = (100\cdots0), \m{e}_{1} = (010\cdots0), \cdots, \m{e}_{k-1} = (000\cdots01)$, these are given by:
\begin{itemize}[leftmargin = *]
\item \textbf{($\mb{\lambda} = \infty$)} The clique is given by $\c{C}_{\infty} = \{(\m{0} | \m{b}) : \m{b} \in \{0, 1\}^{k} \}$, generated by $\c{G}_{\infty} = \{(\m{0} | \m{e}_{j}) : 1 \le j \le k\}$. This is the computational basis, as the generator simply matches the Pauli Z operators for each qubit.
\item \textbf{($\mb{\lambda} \in \{0, 1\}^{k}$)} The clique is given by $\c{C}_{\mb{\lambda}} = \{(\m{a} | \lambda \bullet \m{a}) : \m{a} \in \{0, 1\}^{k} \}$, generated by $\c{G}_{\mb{\lambda}} = \{(\m{e}_{j} | \mb{\lambda} \bullet \m{e}_{j}) : 0 \le j \le k-1\}$.
\end{itemize}

Let us denote the \emph{physical} generators by $\m{G}_{\mb{\lambda}, j}^{phy} = \m{P}_{(\m{e}_{j} | \mb{\lambda} \bullet \m{e}_{j})}^{phy}$ (it is obvious for $\m{Z}_{j} = \m{G}_{\infty, j}^{phy}$). Because the operators in a clique commute, they share a common eigenbasis, which serves as an orthonormal basis of $\b{C}^{q}$. Each clique determined by $\mb{\lambda} \in \{0, 1\}^{k} \cup \{\infty\}$ and $\m{a} \in \{0, 1\}^{k}$ amounts to the \emph{physical} eigenspace:
\begin{align*}
    \m{\Pi}_{\mb{\lambda}, \m{a}}^{phy} = \prod_{j=1}^{k} \frac{\m{I}_{2} + (-1)^{a_{j}} \m{G}_{\mb{\lambda}, j}^{phy}}{2} = \frac{1}{2^{k}} \sum_{\m{b}\in GF(2^{k})} (-1)^{\m{a}^{\top}\m{b}} \phi_{\mb{\lambda}}(\m{b}) \m{P}_{(\m{b} | \mb{\lambda} \bullet \m{b})}^{phy}
\end{align*}
where $\phi_{\mb{\lambda}}(\m{b}) \in \{\pm 1\}$ is the phase correction arising from \eqref{eq:phase:corr}:
\begin{align*}
    \prod_{j: b_{j} = 1} \m{G}_{\mb{\lambda}, j}^{phy} = \phi_{\mb{\lambda}}(\m{b}) \m{P}_{(\m{b} | \mb{\lambda} \bullet \m{b})}^{phy}.
\end{align*}
The $q+1$ collection of orthonormal bases (with identifications via the eigenspace) $\{\m{\Pi}_{\mb{\lambda}, \m{a}}^{phy} : \m{a} \in \{0, 1\}^{k} \}_{\mb{\lambda}}$ is called the MUB.

Any self-adjoint matrix on the $k$-qubit system $\c{Q} = \b{C}^{q}$ can be decomposed into the Pauli basis
\begin{align*}
    \mb{\rho} = \frac{1}{2^{k}} \sum_{\m{u} \in \{0, 1\}^{2k}} \m{r}_{\m{u}} \m{P}_{\m{u}}^{phy}  \in \b{C}_{sa}^{q \times q}.
\end{align*}
Hence, $\m{r} \in \b{R}^{q^{2}}$ is called the generalized Bloch vector. If $\mb{\rho} \in \b{S}(\b{C}_{+}^{q \times q})$ is a density matrix, then $\m{r}_{(\m{0} \vert \m{0})} = \tr{\mb{\rho}}= 1$, and the coefficients give the expected probability of the Bernoulli distribution under the Pauli observable. In contrast, under the MUB tomography, the expected probabilities are given by the Walsh--Hadamard Transform (WHT) of the Bloch vector: for $\mb{\lambda} \in \{0, 1\}^{k} \cup \{\infty\},  \m{a} \in \{0, 1\}^{k}$,
\begin{align*}
    d_{\mb{\lambda}, \m{a}}[\mb{\rho}] = \text{Tr}(\mb{\rho} \m{\Pi}_{\mb{\lambda}, \m{a}}^{phy}) = \frac{1}{2^k} \sum_{\m{b} \in \{0,1\}^k} (-1)^{\m{a}^{\top}\m{b}} \phi_{\mb{\lambda}}(\m{b}) \m{r}_{(\m{b} | \mb{\lambda} \bullet \m{b})},
\end{align*}
which is the Walsh--Hadamard Transform of the Bloch vector restricted to the clique defined by $\mb{\lambda}$.
If $p_{\mb{\lambda}, \m{a}} \sim \text{Binom}(r, d_{\mb{\lambda}, \m{a}}[\mb{\rho}])/r$ is the estimate for $d_{\mb{\lambda}, \m{a}}[\mb{\rho}]$, then the LSE under the MUB tomography is given by
\begin{align*}
    \hat{\mb{\rho}}^{\text{LSE}} = \sum_{\mb{\lambda} \in \{0, 1\}^{k} \cup \{\infty\}} \sum_{\m{a} \in \{0, 1\}^{k}} p_{\mb{\lambda}, \m{a}} \m{\Pi}_{\mb{\lambda}, \m{a}} - \m{I}_{q}.
\end{align*}
To simply compute its generalized Bloch vector, we can use the inverse WHT: for $\mb{\lambda} \in \{0, 1\}^{k} \cup \{\infty\},  \m{b} \in \{0, 1\}^{k} \backslash \{\m{0}\}$,
\begin{align*}
    \hat{\m{r}}_{(\m{0} \vert \m{0})}^{\text{LSE}} = 1, \quad \hat{\m{r}}_{(\m{b} | \mb{\lambda} \bullet \m{b})}^{\text{LSE}} = \phi_{\mb{\lambda}}(\m{b}) \sum_{\m{a} \in \{0,1\}^k} (-1)^{\m{a}^{\top}\m{b}} p_{\mb{\lambda}, \m{a}}.
\end{align*}
Finally, similar to the FFT algorithm, we employ the fast WHT algorithm that reduces the computational complexity from $O(2^{2k})$ to $O(k \cdot 2^{k})$.

\section{Additional Examples}

\begin{example}[Generic Kernel]\label{ex:disc:kern}
Let $X = \{1, 2, \cdots, l \}$ to simplify the notation, and consider a POVM $\mb{\nu} = \{\mb{\nu}_1, \dots, \mb{\nu}_l\}$ on a quantum system $\c{Q}$. Given a positive semi-definite Gram matrix $\m{K} = [K(i, j)]_{i, j =1}^{l} \in \b{F}^{l \times l}$ defining a scalar-valued reproducing kernel $K$ on $X$, it is convenient to associate any function $f: X \to \b{F}$ with a vector via evaluation:
\begin{equation*}
    (f: X \to \b{F}) \, \cong \, \m{f} \in \b{F}^{l}, \quad \m{f}_{i} = f(i), \ i=1, 2, \cdots, l.
\end{equation*}
We can identify the RKHS $\c{R}(K)$ with the column space of the Gram matrix, $\mathrm{Ran}(\m{K}) \subseteq \b{F}^{l}$, equipped with the inner product $\innpr{\m{f}}{\m{g}}_{\c{R}(K)} = \m{g}^{*} \m{K}^{\dagger} \m{f}$. Under this vectorization, the point evaluation feature map $k_i = K(\cdot, i) \in \c{R}(K)$ simply corresponds to the $i$-th column of $\m{K}$, denoted by $\m{K}_{\cdot i}$. 
For any density matrix $\mb{\rho} \in \b{S}(\b{F}_{sa, 1}^{q \times q})$, the classical covariance embedding of the induced probability measure $\mb{\nu}_{\mb{\rho}}$ is given by:
\begin{align*}
    \m{T}_{K}^{\mb{\nu}_{\mb{\rho}}} = \sum_{i=1}^{l} \m{w}_{i} k_{i} k_{i}^{*} \in \c{B}_{\infty}^{+}(\c{R}(K)), \quad \m{w}_{i} = \tr{\mb{\rho} \mb{\nu}_{i}} 
\end{align*}
Because the rank-one operator $k_{i} k_{i}^{*}$ acts on $f$ via $\innpr{f}{k_i}_{\c{R}(K)} k_{i} = f(i) k_{i}$, its matrix representation on $\b{F}^{l}$ maps $\m{f} \mapsto \m{f}_{i} \m{K}_{\cdot i}$. Consequently, the full operator acts as $\sum_{i=1}^{l} \tr{\mb{\rho} \mb{\nu}_{i}} \m{f}_{i} \m{K}_{\cdot i}$, which precisely corresponds to the matrix multiplication $\m{K} \Diag(\m{w})$.
Lifting this to the Quantum Covariance Embedding (QCE) framework, the tensorized operator becomes:
\begin{align*}
    \m{T}_{K}^{\mb{\nu}} = \sum_{i=1}^{l} (k_{i} k_{i}^{*}) \otimes \mb{\nu}_{i} \in \c{B}_{\infty}^{+}(\c{R}(K) \otimes \c{Q}).
\end{align*}
Its action on an elementary tensor $f \otimes \phi \in \c{R}(K) \otimes \c{Q}$ yields:
\begin{align*}
    \m{T}_{K}^{\mb{\nu}} (f \otimes \phi) = \sum_{i=1}^{l} (k_{i} k_{i}^{*} f) \otimes (\mb{\nu}_{i} \phi) = \sum_{i=1}^{l} (f(i) k_{i}) \otimes (\mb{\nu}_{i} \phi).
\end{align*}
We can explicitly verify that this construction recovers the tensorized Bochner integral identity. For any $f, g \in \c{R}(K)$ and $\phi, \psi \in \c{Q}$:
\begin{align*}
    &\innpr{\m{T}_{K}^{\mb{\nu}} (f \otimes \phi)}{g \otimes \psi}_{\c{R}(K) \otimes \c{Q}} = \sum_{i=1}^{l} \innpr{(f(i) k_{i}) \otimes (\mb{\nu}_{i} \phi)}{g \otimes \psi}_{\c{R}(K) \otimes \c{Q}} \\
    &= \sum_{i=1}^{l} f(i) \innpr{k_{i}}{g}_{\c{R}(K)} \innpr{\mb{\nu}_{i} \phi}{\psi}_{\c{Q}} = \sum_{i=1}^{l} f(i) \overline{g(i)} \innpr{\mb{\nu}_{i} \phi}{\psi}_{\c{Q}} = \int_{X} f(x) \overline{g(x)} \ \rd \mb{\nu}_{\phi, \psi}(x).
\end{align*}
\end{example}

\begin{example}[Complex Qubit System]\label{ex:comp:qubit:LSE}
For a standard complex qubit system $\c{Q} = \b{C}^{2}$, the tensorized linear regression for density matrix estimation geometrically reduces to classical linear regression on the Bloch vector. Consider a design composed of $n$ observables $\{\m{O}_{i} = \m{u}_{i} \cdot \vec{\mb{\sigma}} \in \b{F}_{sa, 0}^{2 \times 2} : \m{u}_{i} \in \b{S}^{2} \}_{i=1}^{n}$. Recall from \cref{thm:supop:qubit} that it is a complete design if and only if $\spann \{\m{u}_{i}\}_{i=1}^{n} = \b{R}^{3}$.
Given an unknown density $\mb{\rho} = \frac{1}{2} (\m{I}_{2} + \m{a} \cdot \vec{\mb{\sigma}}) \in \b{S}(\b{C}_{+}^{2 \times 2})$ with Bloch vector $\m{a} \in \b{B}^{3}$, the expected outcome probabilities are:
\begin{align*}
    d_{i, \pm}[\mb{\rho}] = \tr{\mb{\rho} \frac{\m{I}_{2} \pm \m{u}_{i} \cdot \vec{\mb{\sigma}}}{2}} = \frac{1 \pm \innpr{\m{a}}{\m{u}_{i}}_{\b{R}^{3}}}{2}.
\end{align*}
Then, the estimated Bloch vector $\hat{\m{a}}^{\text{LSE}}$ for $\hat{\mb{\rho}}^{\text{LSE}}$ is obtained by the standard OLS inversion:
\begin{align*}
    \hat{\m{a}}^{\text{LSE}} = \bar{\m{D}}^{-1} \bar{\m{p}} \in \b{R}^{3}, \quad \bar{\m{D}} = \frac{1}{n} \sum_{i=1}^{n} \m{u}_{i} \m{u}_{i}^{\top}, \quad \bar{\m{p}} := \frac{1}{n} \sum_{i=1}^{n} (2 p_{i, +}-1) \m{u}_{i} \in \b{B}^{3},
\end{align*}
which follows from the geometric action of the Gram superoperator derived in \cref{thm:supop:qubit}:
\begin{align*}
    \hat{\mb{\rho}}^{\text{LSE}} = \left( \frac{\f{D}^{*} \f{D}}{n} \right)^{-1} \left(\frac{\f{D}^{*} \m{P}}{n} \right) 
    = \left( \frac{\f{D}^{*} \f{D}}{n} \right)^{-1} \frac{\m{I}_{2} + \bar{\m{p}} \cdot \vec{\mb{\sigma}}}{2} 
    = \frac{\m{I}_{2} + \hat{\m{a}}^{\text{LSE}} \cdot \vec{\mb{\sigma}}}{2}.
\end{align*}
Because $\mathrm{MSE}[\hat{\mb{\rho}}^{\text{LSE}} \vert \mb{\rho}] = \frac{1}{2} \b{E} \left[\|\hat{\m{a}}^{\text{LSE}} - \m{a}\|_{\b{R}^{3}}^{2} \vert \m{a} \right]$, the operator MSE can be explicitly mapped to the design geometry:
\begin{align*}
    \mathrm{MSE}[\hat{\mb{\rho}}^{\text{LSE}} \vert \mb{\rho}] 
    = \frac{1}{2nr} \left( \tr{ \bar{\m{D}}^{-1}} - \tr{ \bar{\m{D}}^{-2} \left(\frac{1}{n} \sum_{i=1}^{n} (\m{u}_{i}^{\top} \m{a})^{2} \m{u}_{i} \m{u}_{i}^{\top} \right)} \right).
\end{align*}
In the special case where $\{\m{O}_{i} \}_{i=1}^{n}$ forms a $1/3$-unitary design (i.e., $\bar{\m{D}} = \m{I}_{3}/3$), the estimator simplifies directly to $\hat{\m{a}}^{\text{LSE}} = 3 \bar{\m{p}}$, and the MSE elegantly reduces to the isotropic bound derived in \cref{cor:unit:mse}:
\begin{align*}
    \mathrm{MSE}[\hat{\mb{\rho}}^{\text{LSE}} \vert \mb{\rho}] = \frac{3(3- \|\m{a}\|_{\b{R}^{3}}^{2})}{2nr}.
\end{align*}
Analogous to the rebit case ($|\hat{z}^{\text{LSE}}|_{\b{C}} > 1$), if $\|\hat{\m{a}}^{\text{LSE}}\|_{\b{R}^{3}} > 1$, the constrained optimal solution is strictly obtained by mapping the vector back to the physical boundary: $\hat{\m{a}}_{+}^{\text{LSE}} = \hat{\m{a}}^{\text{LSE}} / \|\hat{\m{a}}^{\text{LSE}}\|_{\b{R}^{3}} \in \b{S}^{2}$, amounting to a pure state.

Finally, incorporating the classical bit-flip noise defined in \cref{ssec:bit:flip}, the attenuated expectation becomes $\b{E}[\hat{\m{a}}^{\text{LSE}} \vert \m{a}] = (1-2\eta)\m{a}$, and the corresponding corrupted $\mathrm{MSE}[\hat{\mb{\rho}}^{\text{LSE}} \vert \mb{\rho}]$ is:
\begin{align*}
    \frac{1}{2nr} \left( \tr{ \bar{\m{D}}^{-1}} - (1-2\eta)^{2} \tr{ \bar{\m{D}}^{-2} \left(\frac{1}{n} \sum_{i=1}^{n} (\m{u}_{i}^{\top} \m{a})^{2} \m{u}_{i} \m{u}_{i}^{\top} \right)} \right) + 4\eta^{2} \|\m{a}\|_{\b{R}^{3}}^{2}.
\end{align*}

\begin{figure}[h!]
    \centering
    \includegraphics[width=\textwidth]{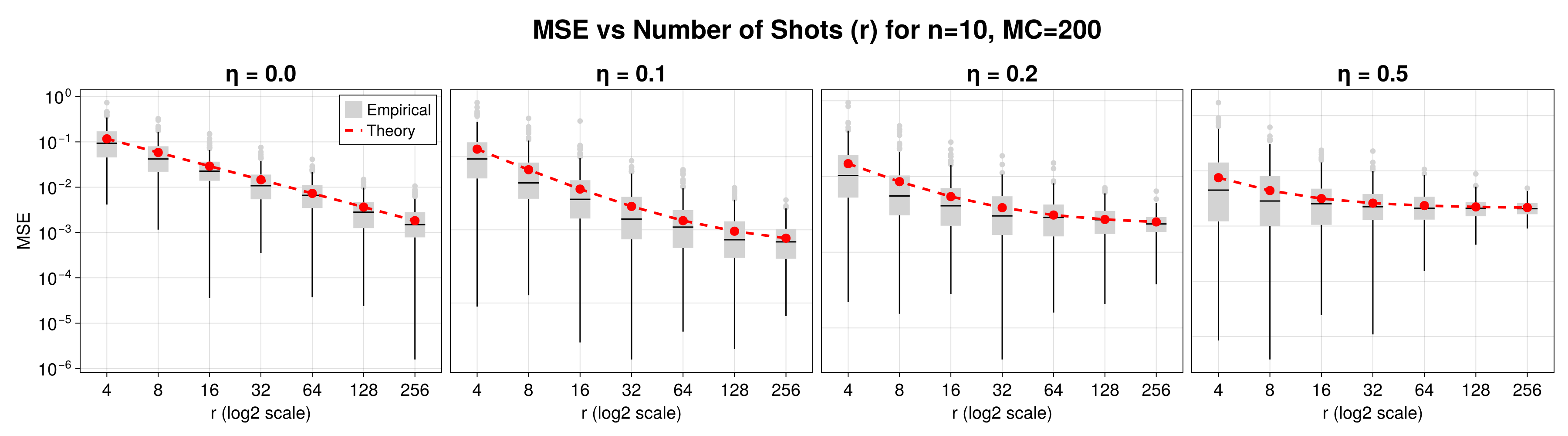}
    
    \caption{Operator MSE after performing Monte Carlo Simulation with $200$ repetitions for each setup. Here, we have used the approximate unitary design, i.e., $\m{u}_{i}$'s are randomly generated on the unit sphere $\b{S}^{2}$. The theoretical trend lines amounts for that under the exact unitary design.}
    \label{fig:mse:eta}
\end{figure}

Observe that as the noise rate $\eta$ approaches the totally mixed channel $0.5$, variance convergence happens faster, but the asymptotic limiting MSE sharply increases.
\end{example}

\subsection{Discrete Modular Error in Qudits}\label{ssec:qudit:povm} 
Consider a qudit system $\c{Q} = \b{C}^{q}$ with the orthonormal basis $\{\phi_{0}, \cdots, \phi_{q-1}\}$. Let $\m{N}$ be the number operator defined by: 
\begin{align}\label{eq:numb:op}
    \m{N} = \sum_{k=0}^{q-1} k \ \m{\Pi}_k, \quad \m{\Pi}_k = \phi_{k} \phi_{k}^{*}, \ X = \sigma(\m{N}) = \{0, 1, \dots, q-1\} \cong \b{Z}_{q}.
\end{align}
While the ideal PVM is $\mb{\nu}^{\m{N}} = \sum_{k} \delta_k \m{\Pi}_k$, consider a scenario standard in high-dimensional quantum communication where detector outcomes are subject to \emph{modular additive noise} \cite{gottesman2001encoding}:
\begin{align*}
    Y_{\mb{\eta}} := (Y + \varepsilon) \pmod q, \quad \b P(Y=\lambda_k\mid\mb\rho)=\tr{\mb{\rho} \m{\Pi}_{k}}, \quad \varepsilon \sim \text{Categorical}(\mb{\eta}) ,
\end{align*} 
where $\mb{\eta} = (\eta_0, \dots, \eta_{q-1})$ represents the noise profile satisfying $\eta_{k} \ge 0$ and $\sum_{k=0}^{q-1} \eta_{k} = 1$. The resulting POVM $\mb{\nu}^{\m{N}, \mb{\eta}}$ is the convolution of the ideal PVM with the noise: 
\begin{align*}
    \mb{\nu}^{\m{N}, \mb{\eta}}(\{j\}) = \sum_{k=0}^{q-1} \b{P}(Y_{\mb{\eta}}=j | Y=k) \m{\Pi}_{k} = \sum_{k \in \b{Z}_{q}} \eta_{(j-k)}  \m{\Pi}_{k}, \quad j \in \b{Z}_{q}.
\end{align*}

\begin{proposition}[QMD of Modular Noise]\label{QMD_modular_noise}
Let $\m{K} \in \b{C}_{+}^{q \times q}$ be a Gram matrix associated to a kernel $K: \b{Z}_q \times \b{Z}_{q} \to \b{C}$. For each basis index $k \in \b{Z}_{q}$, define the diagonal discrepancy matrix 
$\m{E}_{k} \in \b{R}^{q \times q}$ with entries $[\m{E}_{k}]_{jj} = \delta_{jk} - \eta_{(j-k) \pmod q} $.
The conditional QMD between the ideal PVM $\mb{\nu}^{\m{N}}$ and the noisy POVM in \eqref{eq:numb:op} is given by:
\begin{align*}
    \QMD_{K}^{\mb{\rho}}(\mb{\nu}^{\m{N}}, \mb{\nu}^{\m{N}, \mb{\eta}}) = 
    \vertiii{\sum_{k=0}^{q-1} \tr{\mb{\rho} \m{\Pi}_k} (\m{K}^{1/2} \m{E}_{k} \m{K}^{1/2})}_{\infty} \le \max_{k \in \b{Z}_{q}} \vertiii{\m{K}^{1/2} \m{E}_{k} \m{K}^{1/2}}_{\infty}.
\end{align*}
The upper bound is achieved by the pure state $\m{\Pi}_{j}$ where 
\begin{align*}
    j = \argmax\limits_{k \in \b{Z}_{q}} \vertiii{\m{K}^{1/2} \m{E}_{k} \m{K}^{1/2}}_{\infty}.
\end{align*}
In particular, if $\m{K} = \m{I}_{q}$, the maximum is attained by \emph{any} basis state $\m{\Pi}_{k}$ for any $k \in \b{Z}_{q}$.
\end{proposition}
\begin{proof}[Proof of \cref{QMD_modular_noise}]
For each $j \in \b{Z}_{q}$, let $\m{\Pi}_{j} = \phi_j \phi_j^*$ be the ideal projections. The effective operators for the noisy POVM are $\m{\Pi}_{\mb{\eta}, j} = \sum_{k=0}^{q-1} \eta_{(j-k)} \m{\Pi}_k$.
Given a density $\mb{\rho}$, let $\m{p} \in \b{R}^q$ be the probability vector where $p_k = \tr{\mb{\rho} \m{\Pi}_k}$. The distribution of the ideal outcome is $\mb{\nu}_{\mb{\rho}}(\{j\}) = p_j$, and the noisy outcome is the convolution $\mb{\nu}^{\m{N}, \mb{\eta}}_{\mb{\rho}}(\{j\}) = (\mb{\eta} \ast \m{p})_j$. For any function $f \in \c{R}(K)$, the inner product is:
\begin{align*}
    &\innpr{(\m{T}_{K}^{\mb{\nu}^{\m{N}}} - \m{T}_{K}^{\mb{\nu}^{\m{N}, \mb{\eta}}}) f}{f}_{\c{R}(K)} 
    = \sum_{j=0}^{q-1} |f(j)|^2 (\mb{\nu}^{\m{N}}(\{j\}) - \mb{\nu}^{\m{N}, \mb{\eta}}(\{j\})) \\
    &= \sum_{j=0}^{q-1} |f(j)|^2 \left( p_j - \sum_{k=0}^{q-1} \eta_{(j-k)} p_k \right) = \sum_{k=0}^{q-1} p_k \left[ |f(k)|^2 - \sum_{j=0}^{q-1} \eta_{(j-k)} |f(j)|^2 \right] \\ &= \m{f}^{*} \left(\sum_{k=0}^{q-1} p_k  \m{E}_{k} \right) \m{f},
\end{align*}
where $\m{f} = [f(0), \dots, f(q-1)]^\top \in \b{C}^{q}$. Using the Hilbert space isomorphism $\m{f} = \m{K}^{1/2}\m{u}$ with $\|\m{u}\|_2 \le 1$ as in the proof of \cref{thm:pauli:QMD}, the conditional QMD for a fixed $\mb{\rho}$ is:
\begin{align*}
    \QMD_{K}^{\mb{\rho}}(\mb{\nu}, \mb{\nu}_{\mb{\eta}}) &= 
    \vertiii{\sum_{k=0}^{q-1} p_{k} (\m{K}^{1/2} \m{E}_{k} \m{K}^{1/2})}_{\infty}.
\end{align*}
To find the maximally distinguishing state, we maximize this norm over the probability simplex $\m{p}$. The spectral norm is a convex function; thus, its maximum over the simplex is attained at one of the vertices, and the bound follows immediately.
\end{proof}

Recall from \cref{thm:pauli:QMD} that for qubits ($\c{Q} = \b{C}^{2}$), the maximally distinguishing states are independent of the kernel choice. However, for higher-dimensional qudits $(q>2)$, this invariance breaks down, and the optimal state index becomes coupled to the geometry of the kernel. Specifically, the kernel imposes a non-uniform metric over the outcome space $\b{Z}_{q}$, assigning different statistical weights to the discrepancy of each eigenstate.

\section{Proofs}\label{sec:proof}

\begin{proof}[Proof of \cref{thm:quant:cov:scalar}]
Define a bivariate function on $\c{R}(K) \otimes \c{Q}$ by
\begin{equation}\label{eq:sesqui:embed}
    \m{L}_{K}^{\mb{\nu}}(f \otimes \phi, g \otimes \psi) := \int_{X} f(x) \overline{g(x)} \ \rd \mb{\nu}_{\phi, \psi}(x), \quad f, g \in \c{R}(K), \, \phi,\psi \in \c{Q}.
\end{equation}
As the right-hand side is linear in $f$ and $\phi$ and conjugate linear in $g$ and $\psi$, \eqref{eq:sesqui:embed} is a well-defined sesquilinear form due to the universal property \cite{kadison1986fundamentals}.
We claim that $\m{L}_{K}^{\mb{\nu}}$ in \eqref{eq:sesqui:embed} is bounded. Without loss of generality, we assume $\|\phi\|_{\c{Q}} = 1$ so that $\mb{\nu}_{\phi,\phi} \in \s{P}(X)$, and let $c := \sup_{x \in X} K(x, x) < \infty$. Then, for any $f \in \c{R}(K), \, \phi \in \c{Q}$,
\begin{equation*}
    \m{L}_{K}^{\mb{\nu}}(f \otimes \phi, f \otimes \phi) = \int_{X} |f(x)|^{2} \ \rd \mb{\nu}_{\phi,\phi}(x) \le c \|f\|_{\c{R}(K)}^{2} \|\phi\|_{\c{Q}}^{2} =  c \|f \otimes \phi\|_{\c{R}(K) \otimes \c{Q}}^{2}, 
\end{equation*}
thus $\m{L}_{K}^{\mb{\nu}}$ is bounded due to \cite[Proposition A.61]{hall2013quantum}. Therefore, from \cite[Proposition A.63]{hall2013quantum}, there is a unique $\m{T}_{K}^{\mb{\nu}} \in \c{B}_{\infty}(\c{R}(K) \otimes \c{Q})$ with $\vertiii{\m{T}_{K}^{\mb{\nu}}}_{\infty} \le c$ such that 
\begin{equation*}
    \m{L}_{K}^{\mb{\nu}}(f \otimes \phi, g \otimes \psi) = \innpr{\m{T}_{K}^{\mb{\nu}} (f \otimes \phi)}{g \otimes \psi}_{\c{R}(K) \otimes \c{Q}}.
\end{equation*}
Additionally, it is trivial that $\m{T}_{K}^{\mb{\nu}} \succeq \m{0}$ since $\mb{\nu}_{\phi,\phi}$ is a non-negative measure for any $\phi \in \c{Q}$. 

Let $\{f_{i}: i \in \b{N} \}$ and $\{\phi_{j}: j \in \b{N}\}$ be countable orthonormal bases for $\c{R}(K)$ and $\c{Q}$, respectively. Then, we obtain
\begin{align*}
    \tr{\m{T}_{K}^{\mb{\nu}}} &= \sum_{j = 1}^{\infty} \sum_{i=1}^{\infty} \innpr{\m{T}_{K}^{\mb{\nu}} (f_{i} \otimes \phi_{j})}{f_{i} \otimes \phi_{j}}_{\c{R}(K) \otimes \c{Q}} \\
    &= \sum_{j = 1}^{\infty} \int_{X} \left( \sum_{i=1}^{\infty} |f_{i}(x)|^{2} \right) \ \rd \mb{\nu}_{\phi_{j}, \phi_{j}}(x) 
    = \sum_{j = 1}^{\infty} \int_{X} K(x, x) \ \rd \mb{\nu}_{\phi_{j}, \phi_{j}}(x) \\
    &= \sum_{j = 1}^{\infty} \innpr{\left(\int_{X} K(x, x) \ \rd \mb{\nu}(x) \right) \phi_{j}}{\phi_{j}}_{\c{Q}} = \vertiii{\int_{X} K(x, x) \ \rd \mb{\nu}(x)}_{1},
\end{align*}
provided that either quantity is finite. This completes the proof.
\end{proof}

\begin{proof}[Proof of \cref{thm:char:ker}]
\begin{enumerate}[leftmargin = *]
\item Let $\mb{\mu}, \mb{\nu}  \in \s{P}(X, \c{Q})$ be two POVMs. Then, \eqref{eq:embed:cvx:combi}, the Krein--Milman theorem, and \cite[Proposition A.63]{hall2013quantum} establish the following equivalence:
\begin{align*}
    &\m{T}_{K}^{\mb{\mu}_{\mb{\rho}}} = \m{T}_{K}^{\mb{\nu}_{\mb{\rho}}} \in \c{B}_{\infty}^{+}(\c{R}(K)), \quad \forall \mb{\rho} \in \b{S}(\c{B}_{1}^{+}(\c{Q})) \\ 
    &\Leftrightarrow \quad \m{T}_{K}^{\mb{\mu}_{\phi \phi^{*}}} = \m{T}_{K}^{\mb{\nu}_{\phi \phi^{*}}}, \quad \forall \phi \in \c{Q}, \, \|\phi\|_{\c{Q}} = 1 \\
    &\Leftrightarrow \quad \innpr{(\m{T}_{K}^{\mb{\mu}_{\phi \phi^{*}}} - \m{T}_{K}^{\mb{\nu}_{\phi \phi^{*}}}) f}{f}_{\c{R}(K)} = 0, \quad \forall f \in \c{R}(K), \, \|\phi\|_{\c{Q}} = 1 \\
    &\Leftrightarrow \quad \innpr{(\m{T}_{K}^{\mb{\mu}} - \m{T}_{K}^{\mb{\nu}}) (f \otimes \phi)}{f \otimes \phi}_{\c{R}(K) \otimes \c{Q}} = 0, \quad \forall f \in \c{R}(K), \, \phi \in \c{Q} \quad \\ 
    &\Leftrightarrow \quad \m{T}_{K}^{\mb{\mu}} = \m{T}_{K}^{\mb{\nu}}.
\end{align*}
Note that the restriction $\|\phi\|_{\c{Q}} = 1$ can be dropped in the penultimate step due to the scale-invariance of the zero element in the sesquilinear form.
\item Assume $K$ is characteristic for scalar probability measures in $\s{P}(X)$. If $\m{T}_{K}^{\mb{\mu}} = \m{T}_{K}^{\mb{\nu}}$, then by Part 1, $\m{T}_{K}^{\mb{\mu}_{\mb{\rho}}} = \m{T}_{K}^{\mb{\nu}_{\mb{\rho}}}$ for all density operators $\mb{\rho} \in \b{S}(\c{B}_{1}^{+}(\c{Q}))$. Because $\mb{\mu}_{\mb{\rho}}, \mb{\nu}_{\mb{\rho}} \in \s{P}(X)$ are scalar probability measures, the characteristic property of $K$ implies $\mb{\mu}_{\mb{\rho}} = \mb{\nu}_{\mb{\rho}}$ for all $\mb{\rho}$. This means $\tr{\mb{\mu}(B) \mb{\rho}} = \tr{\mb{\nu}(B) \mb{\rho}}$ for all Borel sets $B \in \s{B}(X)$ and all density matrices $\mb{\rho}$. Consequently, $\mb{\mu}(B) = \mb{\nu}(B)$ identically, meaning $\mb{\mu} = \mb{\nu}$. Thus, $K$ is characteristic for $\s{P}(X, \c{Q})$.

Conversely, assume $K$ is characteristic for $\s{P}(X, \c{Q})$. Let $\mu, \nu \in \s{P}(X)$ be scalar probability measures such that their classical kernel embeddings match: $\m{T}_{K}^{\mu} = \m{T}_{K}^{\nu}$. We can construct the associated POVMs via $\mb{\mu}(\cdot) := \mu(\cdot) \m{I}_{\c{Q}}$ and $\mb{\nu}(\cdot) := \nu(\cdot) \m{I}_{\c{Q}}$. Their corresponding QCEs identically match because $\m{T}_{K}^{\mb{\mu}} = \m{T}_{K}^{\mu} \otimes \m{I}_{\c{Q}} = \m{T}_{K}^{\nu} \otimes \m{I}_{\c{Q}} = \m{T}_{K}^{\mb{\nu}}$. This immediately forces $\mu = \nu$. Thus, $K$ is characteristic for $\s{P}(X)$.
\end{enumerate}
\end{proof}

\begin{proof}[Proof of \cref{prop:ineq:QCE}]
For any bounded self-adjoint operator $\m{A} \in \c{B}_{\infty}(\c{H})$, it holds that $\vertiii{\m{A}}_{\infty} = \sup_{\|f\|_{\c{H}}=1} \left| \innpr{\m{A} f}{f}_{\c{H}} \right|$. Now, since $\m{T}_{K}^{\mb{\mu}} - \m{T}_{K}^{\mb{\nu}} \in \c{B}_{\infty}(\c{R}(K) \otimes \c{Q})$ is self-adjoint, we have
\begin{align*}
    \vertiii{\m{T}_{K}^{\mb{\mu}} - \m{T}_{K}^{\mb{\nu}}}_{\infty} &\ge \sup_{\|\phi\|_{\c{Q}}=1} \sup_{\|f\|_{\c{R}(K)}=1} \left| \innpr{(\m{T}_{K}^{\mb{\mu}} - \m{T}_{K}^{\mb{\nu}}) (f \otimes \phi)}{f \otimes \phi}_{\c{R}(K) \otimes \c{Q}} \right| \\
    &= \sup_{\|\phi\|_{\c{Q}}=1} \sup_{\|f\|_{\c{R}(K)}=1} \left| \innpr{(\m{T}_{K}^{\mb{\mu}_{\phi \phi^{*}}} - \m{T}_{K}^{\mb{\nu}_{\phi \phi^{*}}}) f}{f}_{\c{R}(K)} \right| \\
    &= \sup_{\|\phi\|_{\c{Q}}=1} \vertiii{\m{T}_{K}^{\mb{\mu}_{\phi \phi^{*}}} - \m{T}_{K}^{\mb{\nu}_{\phi \phi^{*}}}}_{\infty} = \sup_{\mb{\rho} \in \b{S}(\c{B}_{1}^{+}(\c{Q}))} \vertiii{\m{T}_{K}^{\mb{\mu}_{\mb{\rho}}} - \m{T}_{K}^{\mb{\nu}_{\mb{\rho}}}}_{\infty},
\end{align*}
where we used \eqref{eq:embed:cvx:combi} alongside the fact that pure states (rank-one projections $\phi\phi^*$) are the extreme points of the convex set of density operators $\b{S}(\c{B}_{1}^{+}(\c{Q}))$. 
\end{proof}

\begin{proof}[Proof of \cref{thm:pauli:QMD}]
The eigenprojections for the Pauli-X matrix $\mb{\sigma}_{x}$ corresponding to eigenvalues $\pm 1$ are given by $\m{\Pi}_{x, \pm} = (\m{I}_{2} \pm \mb{\sigma}_{x})/2$, with analogous expressions for $\mb{\sigma}_{y}$ and $\mb{\sigma}_{z}$.  
For any function $f \in \c{R}(K)$, we get
\begin{align*}
    &\innpr{(\m{T}_{K}^{\mb{\nu}_{i, \mb{\rho}}} - \m{T}_{K}^{\mb{\nu}_{j, \mb{\rho}}}) f}{f}_{\c{R}(K)} 
    = \int_{X} |f(x)|^{2} \ \rd (\mb{\nu}_{i, \mb{\rho}}-\mb{\nu}_{j, \mb{\rho}})(x) \\
    &= |f(+1)|^{2} \tr{\mb{\rho} (\m{\Pi}_{i, +}- \m{\Pi}_{j, +})} + |f(-1)|^{2} \tr{\mb{\rho} (\m{\Pi}_{i, -}- \m{\Pi}_{j, -})}.
\end{align*}
Using the relation $\m{\Pi}_{i, \pm} - \m{\Pi}_{j, \pm} = \pm (\mb{\sigma}_{i} - \mb{\sigma}_{j})/2$, this simplifies to:
\begin{align*}
    \innpr{(\m{T}_{K}^{\mb{\nu}_{i, \mb{\rho}}} - \m{T}_{K}^{\mb{\nu}_{j, \mb{\rho}}}) f}{f}_{\c{R}(K)} 
    &= \frac{\m{f}^{*} \mb{\sigma}_{z} \m{f}}{2} \tr{\mb{\rho} (\mb{\sigma}_{i} - \mb{\sigma}_{j})},
\end{align*}
where $\m{f} = [f(+1), f(-1)]^{\top} \in \b{C}^{2}$ is the vector of point evaluations. Thus
\begin{align*}
    \QMD_{K}^{\mb{\rho}}(\mb{\nu}_{i}, \mb{\nu}_{j}) 
    &= \frac{|\tr{\mb{\rho} (\mb{\sigma}_{i} - \mb{\sigma}_{j})}|}{2} \sup_{\|f\|_{\c{R}(K)} \le 1} \left| \m{f}^* \mb{\sigma}_z \m{f} \right|.
\end{align*}
Substituting the isomorphism $f \in \c{R}(K) \ \cong \ \m{u} = \m{K}^{\dagger/2} \m{f} \in (\text{Ran}(\m{K}), \innpr{\cdot}{\cdot})$ yields:
\begin{align*}
    \sup_{\|f\|_{\c{R}(K)} \le 1} \left| \m{f}^* \mb{\sigma}_z \m{f} \right| = \sup_{\|\m{u}\| \le 1} \left| \m{u}^* \m{K}^{1/2} \mb{\sigma}_{z} \m{K}^{1/2} \m{u} \right| 
    = \vertiii{\m{K}^{1/2} \mb{\sigma}_{z} \m{K}^{1/2}}_{\infty},
\end{align*}
where we utilized that the operator norm of a Hermitian matrix is its spectral radius.
Finally, consider the Bloch sphere representation \eqref{eq:Bloch:repre} to maximize the trace term. Using the orthogonality of Pauli matrices ($\tr{\mb{\sigma}_k \mb{\sigma}_l} = 2\delta_{kl}$), we obtain:
\begin{align*}
    \sup_{\mb{\rho} \in \b{S}(\b{C}_{+}^{2 \times 2})} |\tr{\mb{\rho} (\mb{\sigma}_{i} - \mb{\sigma}_{j})}| = \sup_{\m{a} \in \b{B}^{3}} |a_{i} - a_{j}| = \sup_{\m{a} \in \b{B}^{3}} |\innpr{\m{a}}{\m{e}_{i} - \m{e}_{j}}| = \sqrt{2}.
\end{align*}
The supremum is attained if and only if the Bloch vector is either $\m{a}_{\pm} = \pm (\m{e}_{i} - \m{e}_{j})/\sqrt{2} \in \b{S}^{2}$, which corresponds to $\mb{\rho}_\pm$.
\end{proof}

\begin{proof}[Proof of \cref{prop:QMD:flip}]
Let us define the effective operators for the noisy POVM as: $\m{\Pi}_{\eta, +} = (1-\eta) \m{\Pi}_{+} + \eta \m{\Pi}_{-}, \quad \m{\Pi}_{\eta, -} = \eta \m{\Pi}_{+} + (1-\eta) \m{\Pi}_{-}$.
Following the derivation in the proof of \cref{thm:pauli:QMD}:
\begin{align*}
    \innpr{(\m{T}_{K}^{\mb{\nu}^{\m{O}}_{\mb{\rho}}} - \m{T}_{K}^{\mb{\nu}^{\m{O}, \eta}_{\mb{\rho}}}) f}{f}_{\c{R}(K)} 
    = \frac{\m{f}^* \mb{\sigma}_{z} \m{f}}{2} \tr{\mb{\rho} (\m{O} - \m{O}_{\eta})} = \eta \tr{\mb{\rho} \m{O}} (\m{f}^* \mb{\sigma}_{z} \m{f}), \, f \in \c{R}(K).
\end{align*}
Taking the absolute value and the supremum over the unit ball $\|f\|_{\c{R}(K)} \le 1$ yields:
\begin{align*}
    \QMD_{K}^{\mb{\rho}}(\mb{\nu}^{\m{O}}, \mb{\nu}^{\m{O}, \eta}) 
    &= \eta \vertiii{\m{K}^{1/2} \mb{\sigma}_{z} \m{K}^{1/2}}_{\infty} |\tr{\mb{\rho} \m{O}}|.
\end{align*}
To maximize $|\tr{\mb{\rho} \m{O}}|$, let $\m{O} = \m{u} \cdot \vec{\mb{\sigma}}$ with $\m{u} \in \b{S}^{2}$. Then:
\begin{align*}
    \sup_{\mb{\rho} \in \b{S}(\b{C}_{+}^{2 \times 2})} |\tr{\mb{\rho} \m{O}}| = \sup_{\m{a} \in \b{B}^{3}} |\innpr{\m{a}}{\m{u}}| = 1.
\end{align*}
The supremum is attained if and only if $\m{a} = \pm \m{u}$, which are the pure eigenstates of $\m{O}$.
\end{proof}

\begin{proof}[Proof of \cref{thm:two:bit:flip}]
Consider the Hermitian matrices $\m{\Pi}_{\eta, \pm}$ and $\tilde{\m{\Pi}}_{\tilde{\eta}, \pm}$ defined as in the proof of \cref{prop:QMD:flip}. Specifically, we have:
\begin{align*}
    \tilde{\m{\Pi}}_{\tilde{\eta}, +} = (1-\tilde{\eta}) \tilde{\m{\Pi}}_{+} + \tilde{\eta} \tilde{\m{\Pi}}_{-} = (1-\tilde{\eta}) \frac{\m{I}_{2} + \tilde{\m{O}}}{2} + \tilde{\eta} \frac{\m{I}_{2} - \tilde{\m{O}}}{2} = \frac{\m{I}_{2} + (1-2 \tilde{\eta}) \tilde{\m{O}}}{2}. 
\end{align*}
The difference between the positive operators is:
\begin{align*}
    \m{\Pi}_{\eta, +} - \tilde{\m{\Pi}}_{\tilde{\eta}, +} &= \frac{(1-2 \eta) \m{O} - (1-2 \tilde{\eta}) \tilde{\m{O}}}{2} \\
    &= \frac{(1-2\eta) \m{u} - (1-2 \tilde{\eta}) \tilde{\m{u}}}{2} \cdot \vec{\mb{\sigma}} = - (\m{\Pi}_{\eta, -} - \tilde{\m{\Pi}}_{\tilde{\eta}, -}).
\end{align*}
Repeating the steps of \cref{thm:pauli:QMD} leads directly to the equality in \eqref{eq:QMD:flip:two}. 
To find the upper bound, let $\m{v} := (1-2\eta) \m{u} - (1-2 \tilde{\eta}) \tilde{\m{u}}$. Then:
\begin{align*}
    \sup_{\mb{\rho} \in \b{S}(\b{C}_{+}^{2 \times 2})} \left|\tr {\mb{\rho} \frac{(1-2 \eta) \m{O} - (1-2 \tilde{\eta}) \tilde{\m{O}}}{2}} \right| 
    = \frac{1}{2} \sup_{\m{a} \in \b{B}^{3}} |\innpr{\m{a}}{\m{v}}_{\b{R}^{3}}| 
    = \frac{\|\m{v}\|_{\b{R}^{3}}}{2}.
\end{align*}
The supremum is attained if and only if the Bloch vector $\m{a} = \pm \m{v}/ \|\m{v}\|_{\b{R}^{3}}$.
\end{proof}

\begin{proof}[Proof of \cref{prop:identifi}] ($1 \Rightarrow 2$) is trivial since $(\mb{\rho} - \tilde{\mb{\rho}}) \in \b{F}_{sa, 0}^{q \times q}$. We prove ($2 \Rightarrow 1$) by contradiction. Suppose that there exists a non-zero $\mb{\Delta} \in \b{F}_{sa, 0}^{q \times q}$ such that $d_{ik}[\mb{\Delta}] = 0$ for all $i = 1, \cdots, n$ and $k = 1, \cdots, q_{i}$.
Let $\mb{\rho} = \m{I}_{q}/q$ and consider 
\begin{align*}
    \tilde{\mb{\rho}} := (\m{I}_{q} - \varepsilon \mb{\Delta})/q, \quad 0 < \varepsilon < 1/ \vertiii{\mb{\Delta}}_{\infty},
\end{align*}
which is a valid density matrix. Then $\mb{\rho} - \tilde{\mb{\rho}} = \varepsilon \mb{\Delta}/q$, hence
\begin{align*}
    d_{ik}[\mb{\rho}] - d_{ik}[\tilde{\mb{\rho}}] = \frac{\varepsilon}{q} d_{ik}[\mb{\Delta}] = 0, \quad i = 1, \cdots, n, \, k = 1, \cdots, q_{i}.
\end{align*} 
However, this implies that $\mb{\rho} = \tilde{\mb{\rho}}$, i.e., $\mb{\Delta} = \m{0}$, which is a contradiction.
\end{proof}

\begin{proof}[Proof of \cref{prop:sphe:design}]
For any $\m{S} \in \b{F}_{sa}^{q \times q}$ and $\m{A} \in \b{R}^{n \times q}$, we get
\begin{equation*}
    \innpr{\f{D} \m{S}}{\m{A}} = \sum_{i=1}^{n} \sum_{k=1}^{q_{i}} \frac{\tr{\m{S} \m{\Pi}_{ik}}}{m_{ik}} \sum_{\{\ell : k_{i}(\ell) = k\}} a_{i\ell} = \tr{\m{S} \left( \sum_{i=1}^{n} \sum_{k=1}^{q_{i}} \bar{a}_{ik} \m{\Pi}_{ik} \right)} = \innpr{\m{S}}{\f{D}^{*} \m{A}}, 
\end{equation*}
which establishes the adjoint.
We now show the equivalence of the three conditions:
\begin{itemize}
\item [$(2 \Leftrightarrow 3)$] Since $\text{Ran}(\f{D}^{*}) = \spann \{\m{\Pi}_{ik} : i = 1, \cdots, n, \, k = 1, \cdots, q_{i} \}$, (3) is equivalent to the fact that $\f{D}^{*}$ is surjective, i.e., $\f{D}$ is injective since $\ker(\f{D}) = \text{Ran}(\f{D}^{*})^{\perp}$.
\item [$(3 \Rightarrow 1)$] Trivial.
\item [$(1 \Rightarrow 3)$] Suppose $\f{D}$ is not injective. Then, there exists a non-zero self-adjoint matrix $\mb{\Delta} \in \ker(\f{D}) = \ker(\f{D}^{*} \f{D})$. First, note that $\mb{\Delta} \in \b{F}_{sa, 0}^{q \times q}$  since 
\begin{equation*}
    0 = \tr{\f{D}^{*} \f{D}(\mb{\Delta})} = \tr{\sum_{i=1}^{n} \sum_{k=1}^{q_{i}} \frac{d_{ik}[\mb{\Delta}]}{m_{ik}} \m{\Pi}_{ik}} = \sum_{i=1}^{n} \sum_{k=1}^{q_{i}} d_{ik}[\mb{\Delta}] = n \tr{\mb{\Delta}}.
\end{equation*}
Meanwhile, we have $d_{ik}[\mb{\Delta}] = 0$ for any $i = 1, \cdots, n$ and $k = 1, \cdots, q_{i}$. Thus, the identifiability of $\{\m{O}_{i}: i = 1, \cdots, n\}$ implies that $\mb{\Delta} = \m{0}$, which is a contradiction.
\end{itemize}
\end{proof}

\begin{lemma}[Haar Measure Expectations on the Stiefel Manifold]\label{lem:haar:expec} 
Let $\phi, \psi \in \b{F}^{q}$ be mutually orthogonal unit vectors ($\innpr{\phi}{\psi}_{\b{F}^{q}}=0$) and let $\m{S} \in \b{F}_{sa, 0}^{q \times q}$.
\begin{enumerate}
\item If $\b{F} = \b{R}$ and $\m{U}$ follows a Haar measure on $O(q)$, then
\begin{align*}
    \b{E}\left[\innpr{\m{S} (\m{U}^{\top} \phi)}{(\m{U}^{\top} \phi)}_{\b{R}^{q}} (\m{U}^{\top} \phi) (\m{U}^{\top} \phi)^{\top}\right] &= \frac{2 \m{S}}{q(q+2)}, \\
    \b{E}\left[\innpr{\m{S} (\m{U}^{\top} \phi)}{(\m{U}^{\top} \phi)}_{\b{R}^{q}} (\m{U}^{\top} \psi) (\m{U}^{\top} \psi)^{\top}\right] &= - \frac{2 \m{S}}{q(q+2)(q-1)}.
\end{align*}

\item If $\b{F} = \b{C}$ and $\m{U}$ follows a Haar measure on $U(q)$, then
\begin{align*}
    \b{E}\left[\innpr{\m{S} (\m{U}^{*} \phi)}{(\m{U}^{*} \phi)}_{\b{C}^{q}} (\m{U}^{*} \phi) (\m{U}^{*} \phi)^{*}\right] &= \frac{\m{S}}{q(q+1)}, \\
    \b{E}\left[\innpr{\m{S} (\m{U}^{*} \phi)}{(\m{U}^{*} \phi)}_{\b{C}^{q}} (\m{U}^{*} \psi) (\m{U}^{*} \psi)^{*}\right] &= - \frac{\m{S}}{q(q+1)(q-1)}.
\end{align*}
\end{enumerate}
\end{lemma}

\begin{proof}
Let $\m{v} := \m{U} \phi$ and $\m{w} := \m{U} \psi$. Since $\m{U}$ is Haar-distributed (and thus $\m{U}$, $\m{U}^\top$, and $\m{U}^*$ share the same distribution) and $\innpr{\phi}{\psi}=0$, the pair $(\m{v}, \m{w})$ is uniformly distributed on the Stiefel manifold. We compute the expectations component-wise using the moment formulas for the Stiefel manifold \cite{collins2006integration}.

\begin{enumerate}[leftmargin = *]
\item \textbf{Real Case ($\b{F} = \b{R}$):}
The $(k,l)$-th entry of the target matrix is $\sum_{i,j} \m{S}_{ij} \b{E}[v_i v_j v_k v_l]$.
Using the fourth-order moment formula for the unit sphere:
\begin{equation*}
    \b{E} [v_{i} v_{j} v_{k} v_{l}] = \frac{\delta_{ij}\delta_{kl} + \delta_{ik}\delta_{jl} + \delta_{il}\delta_{jk}}{q(q+2)}.
\end{equation*}
Substituting this into the summation:
\begin{align*}
    \left( \b{E} [\innpr{\m{S}\m{v}}{\m{v}} \m{v} \m{v}^{\top}] \right)_{kl} = \sum_{i,j=1}^{q} \frac{\m{S}_{ij} (\delta_{ij}\delta_{kl} + \delta_{ik}\delta_{jl} + \delta_{il}\delta_{jk})}{q(q+2)} 
    = \frac{\tr{\m{S}}\delta_{kl} + \m{S}_{kl} + \m{S}_{lk}}{q(q+2)}.
\end{align*}
Since $\tr{\m{S}}=0$ and $\m{S}$ is self-adjoint ($\m{S}_{lk}=\m{S}_{kl}$), this simplifies to $\frac{2 \m{S}_{kl}}{q(q+2)}$.
For the orthogonal vectors $\m{v} \perp \m{w}$, we require the joint moments. By $O(q)$-invariance, the general form for the tensor is:
\begin{equation*}
    \b{E}[v_i v_j w_k w_l] = C_1 \delta_{ij}\delta_{kl} + C_2 (\delta_{ik}\delta_{jl} + \delta_{il}\delta_{jk}).
\end{equation*}
The constants are determined by conditioning:
\begin{align*}
    &\frac{1}{q} \delta_{ij} = \b{E}[v_i v_j \innpr{\m{w}}{\m{w}}] =  \sum_{k=1}^{q} [C_{1} \delta_{ij} + 2 C_2 (\delta_{ik}\delta_{jk})] = (q C_{1} + 2 C_2) \delta_{ij}, \\ 
    &0 = \b{E}[v_i w_l \innpr{\m{v}}{\m{w}}] = \sum_{k=1}^{q} [C_1 \delta_{ik}\delta_{kl} + C_2 (\delta_{ik}\delta_{kl} + \delta_{il})] = (C_{1} + (q+1) C_2) \delta_{il},
\end{align*}
yielding $C_2 = - [q(q-1)(q+2)]^{-1}$ and $C_1 = -(q+1)C_2$.
The expectation is:
\begin{align*}
    \left( \b{E} [\innpr{\m{S}\m{v}}{\m{v}} \m{w} \m{w}^{\top}] \right)_{kl} &= \sum_{i,j=1}^{q} \m{S}_{ij} [ C_1 \delta_{ij}\delta_{kl} + C_2 (\delta_{ik}\delta_{jl} + \delta_{il}\delta_{jk}) ] \\
    &= C_1 \tr{\m{S}} \delta_{kl} + C_2 (\m{S}_{kl} + \m{S}_{lk}) 
    = 2 C_2 \m{S}_{kl} = - \frac{2 \m{S}_{kl}}{q(q+2)(q-1)}.
\end{align*}

\item \textbf{Complex Case ($\b{F} = \b{C}$):}
The $(k,l)$-th entry is $\sum_{i,j} \m{S}_{ij} \b{E}[\bar{v}_i v_j v_k \bar{v}_l]$.
Using the Weingarten formula for $U(q)$ \cite{collins2006integration}:
\begin{equation*}
    \b{E} [\bar{v}_{i} v_{j} v_{k} \bar{v}_{l}] = \frac{\delta_{ij}\delta_{kl} + \delta_{ik}\delta_{jl}}{q(q+1)}.
\end{equation*}
Substituting this into the summation:
\begin{align*}
    \left( \b{E} [\innpr{\m{S}\m{v}}{\m{v}} \m{v} \m{v}^{*}] \right)_{kl} = \sum_{i,j=1}^{q} \frac{\m{S}_{ij} (\delta_{ij}\delta_{kl} + \delta_{ik}\delta_{jl})}{q(q+1)} 
    = \frac{\tr{\m{S}}\delta_{kl} + \m{S}_{kl}}{q(q+1)}.
\end{align*}
Since $\tr{\m{S}}=0$, we obtain $\frac{\m{S}_{kl}}{q(q+1)}$.
Again, for the orthogonal vectors $\m{v} \perp \m{w}$, the joint moment takes the form: $\b{E}[\bar{v}_i v_j w_k \bar{w}_l] = C_1 \delta_{ij}\delta_{kl} + C_2 \delta_{ik}\delta_{jl}$.
Using $\b{E}[\bar{v}_i v_j \innpr{\m{w}}{\m{w}}] = q^{-1}\delta_{ij}$ and $\b{E}[\bar{v}_i w_k \innpr{\m{w}}{\m{v}}] = 0$, we find $C_2 = - \frac{1}{q(q^2-1)}$ and $C_1 = - q C_2$.
The expectation is:
\begin{align*}
    \left( \b{E} [\innpr{\m{S}\m{v}}{\m{v}} \m{w} \m{w}^{*}] \right)_{kl} &= \sum_{i,j} \m{S}_{ij} (C_1 \delta_{ij}\delta_{kl} + C_2 \delta_{ik}\delta_{jl}) \\
    &= C_1 \tr{\m{S}} \delta_{kl} + C_2 \m{S}_{kl}
    = C_2 \m{S}_{kl} = - \frac{\m{S}_{kl}}{q(q+1)(q-1)}.
\end{align*}
\end{enumerate}
\end{proof}

\begin{corollary}[Haar Measure Expectations for Projections]\label{lem:haar:expec:v2}
Let $\m{\Pi} \in \b{F}_{sa}^{q \times q}$ be a projection operator of rank $m$ and let $\m{S} \in \b{F}_{sa, 0}^{q \times q}$. Then, it holds that
\begin{align*}
    \frac{1}{m} \b{E}\left[\tr{\m{S}(\m{U} \m{\Pi} \m{U}^{*})} \m{U} \m{\Pi} \m{U}^{*} \right] =  \frac{q - m}{q-1} \frac{\alpha_{q} \m{S}}{q}, \quad \alpha_{q} := \begin{cases}
        (q/2+1)^{-1} &, \quad \b{F} = \b{R}, \\
        (q+1)^{-1} &, \quad \b{F} = \b{C}.
    \end{cases} 
\end{align*}
\end{corollary}
\begin{proof}
Write $\m{\Pi} = \sum_{k=1}^{m} \phi_{k} \phi_{k}^{*}$,
where $\{\phi_{k}: k = 1, \cdots, m \}$ are mutually orthogonal unit vectors. Then,
\begin{align*}
    \tr{\m{S}(\m{U} \m{\Pi} \m{U}^{*})} \m{U} \m{\Pi} \m{U}^{*} &= \sum_{k, k'=1}^{m} \tr{\m{S}(\m{U} \phi_{k} \phi_{k}^{*} \m{U}^{*})} \m{U} \phi_{k'} \phi_{k'}^{*} \m{U}^{*} \\
    &= \sum_{k, k'=1}^{m} \innpr{\m{S} (\m{U} \phi_{k})}{(\m{U} \phi_{k})} (\m{U} \phi_{k'}) (\m{U} \phi_{k'})^{*},
\end{align*}
so the result follows from \cref{lem:haar:expec}:
\begin{align*}
    \b{E}\left[\tr{\m{S}(\m{U} \m{\Pi} \m{U}^{*})} \m{U} \m{\Pi} \m{U}^{*} \right] =   \frac{\alpha_{q} \m{S}}{q} \left[ m -  \frac{m(m-1)}{q-1} \right] = \frac{m \alpha_{q} \m{S}}{q} \frac{q - m}{q-1}.
\end{align*}
\end{proof}

\begin{proof}[Proof of \cref{thm:const:unit:dsgn}]
\quad
\begin{enumerate}[leftmargin = *]
\item Assume that $\s{O}$ is $\alpha$-unitary. Given $\m{V} \in \b{F}^{q \times q}$ with $\m{V} \m{V}^{*} = \m{I}_{q}$, the eigenprojection with respect to the new design $\{\tilde{\m{O}}_{i} := \m{V} \m{O}_{i} \m{V}^{*}: i = 1, \cdots, n\}$ becomes $\tilde{\m{\Pi}}_{ik} = \m{V} \m{\Pi}_{ik} \m{V}^{*}$ with the same multiplicities. Because $\m{V}^{*} \mb{\Delta} \m{V} \in \b{F}_{sa, 0}^{q \times q}$ if and only if $\mb{\Delta} \in \b{F}_{sa, 0}^{q \times q}$, we obtain
\begin{align*}
    \frac{\tilde{\f{D}}^{*} \tilde{\f{D}}}{n} \mb{\Delta} &= \frac{1}{n} \sum_{i=1}^{n} \sum_{k=1}^{q_{i}} \frac{\tr{\mb{\Delta} \tilde{\m{\Pi}}_{ik}}}{m_{ik}} \tilde{\m{\Pi}}_{ik} = \m{V} \left[ \frac{1}{n} \sum_{i=1}^{n} \sum_{k=1}^{q_{i}} \frac{\tr{(\m{V}^{*} \mb{\Delta} \m{V}) \m{\Pi}_{ik}}}{m_{ik}} \m{\Pi}_{ik} \right] \m{V}^{*} \\
    &= \m{V} \left[ \frac{\f{D}^{*} \f{D}}{n} (\m{V}^{*} \mb{\Delta} \m{V}) \right] \m{V}^{*} = \alpha \m{V} (\m{V}^{*} \mb{\Delta} \m{V}) \m{V}^{*} = \alpha \mb{\Delta}.
\end{align*}
The other direction is trivial since $\m{O}_{i} = \m{V}^{*} \tilde{\m{O}}_{i} \m{V}$.

\item Assume that $\s{O}$ is $\alpha$-unitary. Let $\m{U}$ be a random matrix distributed according to the Haar measure on $O(q)$ (if $\b{F} = \b{R}$) or $U(q)$ (if $\b{F} = \b{C}$). Then, \cref{lem:haar:expec:v2} leads to 
\begin{align*}
    \alpha \mb{\Delta} &= \frac{1}{n} \sum_{i=1}^{n} \sum_{k=1}^{q_{i}} \int \tr{\mb{\Delta} \m{U} \m{\Pi}_{ik} \m{U}^{*}} \frac{\m{U} \m{\Pi}_{ik} \m{U}^{*}}{m_{ik}} \rd \m{U} \\
    &= \left( \frac{1}{n} \sum_{i=1}^{n} \sum_{k=1}^{q_{i}} \frac{q - m_{ik}}{q-1}  \right) \frac{\alpha_{q} {\mb{\Delta}}}{q} 
    = \frac{\bar{q} - 1}{q-1} \alpha_{q} \mb{\Delta}.
\end{align*}
\end{enumerate}
\end{proof}

\begin{proof}[Proof of \cref{thm:supop:qubit}]
If $\m{v}_{i} = \m{0}$, we have $\tr{\mb{\Delta} \m{O}_{i}} \m{O}_{i} = \m{0}$ for any $\mb{\Delta} \in \b{C}_{sa, 0}^{2 \times 2}$.
When $\m{v}_{i} \ne \m{0}$, the eigenvalues of $\m{O}_{i}$ are given by $w_{i} \pm \|\m{v}_{i}\|$, with the corresponding eigenprojections $\m{\Pi}_{i,\pm} = \tfrac{1}{2} (\m{I}_{2} \pm \m{u}_{i} \cdot \vec{\mb{\sigma}})$.
Hence, for any $\mb{\Delta} = \m{a} \cdot \vec{\mb{\sigma}} \in \b{C}_{sa, 0}^{2 \times 2}$ with $\m{a} \in \b{R}^{3}$:
\begin{align*}
    \tr{\mb{\Delta} \m{\Pi}_{i, \pm}} \m{\Pi}_{i, \pm} &= \frac{1}{2} \tr{\pm (\m{a} \cdot \vec{\mb{\sigma}}) (\m{u}_{i} \cdot \vec{\mb{\sigma}})} \frac{\m{I}_{2} \pm \m{u}_{i} \cdot \vec{\mb{\sigma}}}{2} \\
    &= \pm (\m{a} \cdot \m{u}_{i}) \frac{\m{I}_{2} \pm \m{u}_{i} \cdot \vec{\mb{\sigma}}}{2},
\end{align*}
which leads to $\tr{\mb{\Delta} \m{\Pi}_{i, +}} \m{\Pi}_{i, +} + \tr{\mb{\Delta} \m{\Pi}_{i, -}} \m{\Pi}_{i, -} = (\m{u}_{i} \m{u}_{i}^{\top} \m{a}) \cdot \vec{\mb{\sigma}} \in \b{C}_{sa, 0}^{2 \times 2}$,
demonstrating \eqref{eq:supop:qubit}. Finally, noting that $\text{Ran} \big[ \sum_{i=1}^{n} \m{u}_{i} \m{u}_{i}^{\top} \big] = \spann \{\m{v}_{i}\}_{i=1}^{n}$, it is clear that $\{\m{O}_{i}\}_{i=1}^{n}$ is complete if and only if $\spann \{\m{v}_{i}\}_{i=1}^{n} = \b{R}^{3}$. Also, the only possible value for an $\alpha$-unitary design is $3\alpha = \tfrac{1}{n} \sum_{i=1}^{n} \|\m{u}_{i}\|^{2} = 1$.
\end{proof}

\begin{proof}[Proof of \cref{prop:unit:design}]
The sufficiency is immediate by restricting the identity to $\b{F}_{sa, 0}^{q \times q}$. Conversely, the trace-centering decomposition in \eqref{eq:tr:center} establishes the necessity:
\begin{align*}
    \frac{\f{D}^{*} \f{D}}{n} \m{S} = \frac{\f{D}^{*} \f{D}}{n} \left[\mb{\Delta} + \tr{\m{S}} \frac{\m{I}_{q}}{q}  \right] = \alpha \mb{\Delta} + \tr{\m{S}} \frac{\m{I}_{q}}{q} = \alpha \m{S} + (1-\alpha) \tr{\m{S}} \frac{\m{I}_{q}}{q}.
\end{align*}
Consider the following linear map:
\begin{align*}
    \f{T}: \m{S} \in \b{F}_{sa}^{q \times q} \mapsto \frac{1}{\alpha} \left[ \m{S} - \frac{1-\alpha}{q} \tr{\m{S}} \m{I}_{q} \right] \in \b{F}_{sa}^{q \times q}.
\end{align*}
Since $\tr{\f{T} [\m{S}]} = \tr{\m{S}}$, we have
\begin{align*}
    \frac{\f{D}^{*} \f{D}}{n} \f{T} \m{S} = \alpha \f{T} [\m{S}] + (1-\alpha) \tr{\m{S}} \frac{\m{I}_{q}}{q} 
    = \m{S}.
\end{align*}
This establishes that $(\f{D}^{*} \f{D}/n)$ is invertible with $\f{T} = (\f{D}^{*} \f{D}/n)^{-1}$.
\end{proof}

\begin{proof}[Proof of \cref{prop:sq:sum:proj}]
Due to \eqref{eq:design2spd} and \cref{prop:unit:design}, we have
\begin{align*}  
    \frac{1}{n} \sum_{i=1}^{n} \sum_{k=1}^{q_{i}} \frac{d_{ik}[\m{S}]^{2}}{m_{ik}} = \innpr{\frac{\f{D}^{*} \f{D}}{n} \m{S}}{\m{S}}_{2}
    = \alpha \tr{\m{S}^{2}} + \frac{1-\alpha}{q} \tr{\m{S}}^{2}.
\end{align*}
\end{proof}

\begin{proof}[Proof of \cref{thm:est:normal:eq}]
We first show that the unconstrained minimizer over $\b{F}_{sa}^{q \times q}$ automatically satisfies the unit trace condition. The normal equation is $(\f{D}^{*} \f{D}) \hat{\mb{\rho}} = \f{D}^{*} \m{P}$. Since the frame superoperator $(\f{D}^{*} \f{D}/n)$ preserves the trace, we have:
\begin{align*}
    n \tr{\hat{\mb{\rho}}} = \tr{(\f{D}^{*} \f{D}) \hat{\mb{\rho}}} = \tr{\f{D}^{*} \m{P}} = \innpr{\f{D}^{*} \m{P}}{\m{I}_{q}}_{2} = \innpr{\m{P}}{\f{D} \m{I}_{q}}_{2}.
\end{align*}
Recall that $\f{D}\m{I}_q$ is the matrix of all ones. Thus, $\innpr{\m{P}}{\f{D} \m{I}_{q}}_{2} = \sum_{i,k} p_{ik} = \sum_i 1 = n$. Hence, $\tr{\hat{\mb{\rho}}} = 1$.
The structure of the solution set follows immediately, and the null space vanishes if and only if $\{\m{O}_{i}\}_{i=1}^{n}$ is complete due to \cref{prop:sphe:design}.
\end{proof}

\begin{proof}[Proof of \cref{cor:equiv:proj}]
For any $\mb{\rho} \in \b{S}(\b{F}_{+}^{q \times q})$, since $\mb{\Delta} = \mb{\rho} - \hat{\mb{\rho}}^{\text{LSE}} \in \b{F}_{sa, 0}^{q \times q}$, the cross-term vanishes:
\begin{align*}
    \mathrm{Loss}[\mb{\rho}] = \vertiii{\f{D}[\mb{\Delta}] + (\f{D}[\hat{\mb{\rho}}^{\text{LSE}}] - \m{P})}_{2}^{2} &= \vertiii{\f{D}[\mb{\Delta}]}_{2}^{2} + \mathrm{Loss}[\hat{\mb{\rho}}^{\text{LSE}}] = \alpha \vertiii{\mb{\Delta}}_{2}^{2} + \mathrm{Loss}[\hat{\mb{\rho}}^{\text{LSE}}],
\end{align*}
hence the result follows.
\end{proof}

\begin{proof}[Proof of \cref{cor:unit:mse}]
Using \eqref{eq:unit:des:LSE}, $\tr{\f{D}^{*} \m{P}} = n$, and \cref{prop:unit:design}, we get
\begin{align*}
    \hat{\mb{\rho}}^{\text{LSE}} = (\f{D}^{*} \f{D})^{-1} \f{D}^{*} \m{P} = \frac{1}{n \alpha} \left[ \f{D}^{*} \m{P} - n (1-\alpha) \frac{\m{I}_{q}}{q} \right] = \frac{\f{D}^{*} \m{P}}{n \alpha} - \frac{1-\alpha}{\alpha} \frac{\m{I}_{q}}{q}{\color{red}.}
\end{align*}
For the MSE, observe that the error matrix lies in the traceless subspace: $\f{D}^{*} (\m{P} - \f{D} \mb{\rho}) \in \b{F}_{sa, 0}^{q \times q}$.
Consequently, the estimation error is:
\begin{align*}
    &\hat{\mb{\rho}}^{\text{LSE}} - \mb{\rho} = (\f{D}^{*} \f{D})^{-1} \f{D}^{*} (\m{P} - \f{D} \mb{\rho}) = \frac{\f{D}^{*} (\m{P} - \f{D} \mb{\rho})}{n\alpha}, \\
    &\vertiii{\hat{\mb{\rho}}^{\text{LSE}} - \mb{\rho}}_{2}^{2} =
    \frac{1}{(n \alpha)^{2}} \vertiii{ \sum_{i=1}^{n} \sum_{k=1}^{q_{i}} \frac{p_{ik} - d_{ik}[\mb{\rho}]}{m_{ik}} \m{\Pi}_{ik} }_{2}^{2}.
\end{align*}
Since $\Cov (p_{ik}, p_{i'k'} \vert \mb{\rho})$ in \eqref{eq:multinom:cov} whenever $i \ne i'$, taking the expectation yields:
\begin{align*}
    \mathrm{MSE}[\hat{\mb{\rho}}^{\text{LSE}} \vert \mb{\rho}] 
    &= \frac{1}{(n \alpha)^{2}} \sum_{i=1}^{n} \sum_{k, k'=1}^{q_{i}} \frac{\Cov (p_{ik}, p_{ik'} \vert \mb{\rho})}{m_{ik} m_{ik'}} \tr{\m{\Pi}_{ik} \m{\Pi}_{ik'}} \\
    &= \frac{1}{(n \alpha)^{2}} \sum_{i=1}^{n} \sum_{k=1}^{q_{i}} \frac{\Cov (p_{ik}, p_{ik} \vert \mb{\rho})}{m_{ik}} 
    = \frac{1}{(n \alpha)^{2} r} \sum_{i=1}^{n} \sum_{k=1}^{q_{i}} \frac{(d_{ik}[\mb{\rho}] - d_{ik}[\mb{\rho}]^{2})}{m_{ik}}.
\end{align*}
Observe the bound for the first-order term and the identity from \cref{prop:sq:sum:proj}:
\begin{align*}
    &\frac{1}{n} \sum_{i=1}^{n} \sum_{k=1}^{q_{i}} \frac{d_{ik}[\mb{\rho}]}{m_{ik}} \le \frac{1}{n} \sum_{i=1}^{n} \sum_{k=1}^{q_{i}} d_{ik}[\mb{\rho}] = 1, \quad \text{with equality if } m_{ik} \equiv 1, \\
    &\frac{1}{n} \sum_{i=1}^{n} \sum_{k=1}^{q_{i}} \frac{d_{ik}[\mb{\rho}]^{2}}{m_{ik}} = \alpha \tr{\mb{\rho}^{2}} + \frac{1-\alpha}{q} = \alpha \left(\tr{\mb{\rho}^{2}} - \frac{1}{q} \right) + \frac{1}{q},
\end{align*}
Substitution into the MSE expression yields the result. 
\end{proof}

\begin{proof}[Proof of \cref{prop:fixed:bud}]
Let $\{\m{E}_{\ell} : 1 \le \ell \le \dim(\b{F}_{sa}^{q \times q})\}$ be an orthonormal basis for $\b{F}_{sa}^{q \times q}$. Using \eqref{eq:design2spd}, we obtain
\begin{align*}
    \innpr{\left(\frac{\f{D}^{*} \f{D}}{n}\right) \m{E}_{\ell}}{\m{E}_{\ell}}_{2} = \frac{1}{n} \sum_{i=1}^{n} \sum_{k=1}^{q_{i}} \frac{\tr{\m{E}_{\ell} \m{\Pi}_{ik}}^{2}}{m_{ik}}.
\end{align*}
By Parseval's identity, we get
\begin{align*}
    \vertiii{\frac{\f{D}^{*} \f{D}}{n}}_{1} &= 
    \sum_{\ell=1}^{\dim(\b{F}_{sa}^{q \times q})}
    \innpr{\left(\frac{\f{D}^{*} \f{D}}{n}\right) \m{E}_{\ell}}{\m{E}_{\ell}}_{2} = \frac{1}{n} \sum_{i=1}^{n} \sum_{k=1}^{q_{i}} \frac{1}{m_{ik}} \left( \sum_{\ell=1}^{\dim(\b{F}_{sa}^{q \times q})} \tr{\m{E}_{\ell} \m{\Pi}_{ik}}^{2}\right) \\
    &= \frac{1}{n} \sum_{i=1}^{n} \sum_{k=1}^{q_{i}} \frac{1}{m_{ik}} \tr{\m{\Pi}_{ik}^{2}} = \frac{1}{n} \sum_{i=1}^{n} q_i = \bar{q}.
\end{align*}
The superoperator always has eigenvalue $1$ on $\spann\{\m{I}_{q}\}$, thus the remaining trace $\bar{q}-1$ must be distributed on $\b{F}_{sa, 0}^{q \times q}$. Thus, the smallest eigenvalue is maximized when the trace is distributed uniformly, i.e., when $\c{O}$ is a unitary design:
\begin{align*}
    \tau_{\min} \le \frac{1}{\dim(\b{F}_{sa}^{q \times q})} \vertiii{\left. \frac{\f{D}^{*} \f{D}}{n} \right\vert_{\b{F}_{sa, 0}^{q \times q}} }_{1} = \frac{\bar{q} - 1}{\dim(\b{F}_{sa}^{q \times q}) -1} = \alpha.
\end{align*}
\end{proof}

\begin{proof}[Proof of \cref{thm:minimax}]
We employ Fano's method \cite{tsybakov2009introduction} with a local packing around the maximally mixed state $\mb{\rho}^{(0)}= \m{I}_{q}/q$. Let $D = \dim (\b{F}_{sa, 0}^{q \times q}) \asymp q^{2}$. Let $\{\m{E}_{1}, \cdots, \m{E}_{D}\} \subset \b{F}_{sa, 0}^{q \times q}$ be the eigenmatrices of the Gram superoperator. From \cref{prop:fixed:bud}, their eigenvalues satisfy $\sum_{j=1}^{D} \tau_{j} = \bar{q} -1$. For a radius parameter $\varepsilon > 0$, we define the perturbations:
\begin{align*}
    \mb{\rho}^{(\omega)} = \mb{\rho}^{(0)} \oplus \underbrace{\varepsilon \sum_{j=1}^{D} \frac{\omega_j}{\sqrt{D \tau_{j}}} \m{E}_{j}}_{=: \mb{\Delta}^{(\omega)}} \in \spann\{\m{I}_{q}\} \oplus \b{F}_{sa, 0}^{q \times q}, \quad \omega \in \{+1, -1\}^{D}.
\end{align*}

By the Varshamov--Gilbert lemma, there exists a subset of the hypercube $\Omega \subset \{+1, -1\}^{D}$ with cardinality $|\Omega| \ge 2^{D/8}$ such that for any distinct $\omega, \omega' \in \Omega$, the Hamming distance satisfies $d_{H}(\omega, \omega') \ge D/8$. Then, $\c{M} = \{\mb{\rho}^{(\omega)} : \omega \in \Omega \} \subset \b{S}(\b{F}_{+}^{q \times q})$ satisfies $\delta$-separation:
\begin{align*}
    &\vertiii{\mb{\rho}^{(\omega)} - \mb{\rho}^{(\omega')}}_{2}^{2} = \vertiii{\mb{\Delta}^{(\omega)} - \mb{\Delta}^{(\omega')}}_{2}^{2} = \frac{4 \varepsilon^{2}}{D} \left( \sum_{j=1}^{D} \frac{\ds{1}(\omega_{j} \neq \omega_{j}')}{ \tau_{j}} \right) \\
    &\ge \frac{4 \varepsilon^{2}}{D} \frac{ \left( \sum_{j=1}^{D} \ds{1}(\omega_{j} \neq \omega_{j}') \right)^{2}}{\sum_{j=1}^{D} \tau_{j}} = \frac{4 \varepsilon^2 d_{H}(\omega, \omega')^{2}}{D (\bar{q}-1)} \ge  \frac{D \varepsilon^2}{16 (\bar{q}-1)} = \frac{\varepsilon^2}{16 \alpha} =: \delta^{2},
\end{align*}
where $\alpha = \frac{\bar{q} - 1}{q-1} \cdot \alpha_{\c{Q}}$.
We now bound the KL divergence by the $\chi^{2}$-divergence. Using the fact that $d_{ik}[\mb{\rho}^{(0)}] = m_{ik} / q$, we obtain for any $\omega \in \Omega$ that:
\begin{align*}
    KL(P_{\mb{\rho}^{(\omega)}} \Vert P_{\mb{\rho}^{(0)}}) &\le \sum_{i=1}^{n} \sum_{j=1}^{r} KL( \mathrm{Categorical}([d_{ik}[\mb{\rho}^{(\omega)}]]) \Vert \mathrm{Categorical}([d_{ik}[\mb{\rho}^{(0)}]]) ) \\
    &\le r \sum_{i=1}^{n} \chi^{2}( \mathrm{Categorical}([d_{ik}[\mb{\rho}^{(\omega)}]]) \Vert \mathrm{Categorical}([d_{ik}[\mb{\rho}^{(0)}]]) ) \\
    &= rq \sum_{i=1}^{n} \sum_{k=1}^{q_{i}} \frac{d_{ik}[\mb{\Delta}^{(\omega)}]^{2}}{m_{ik}} = (n rq ) \innpr{\frac{\f{D}^{*} \f{D}}{n} \mb{\Delta}^{(\omega)}}{\mb{\Delta}^{(\omega)}}_{2} = (nrq) \varepsilon^{2}.
\end{align*}
Then, the Yang--Barron method \cite{yang1999information} gives:
\begin{align*}
    I(V ; Y) \le \frac{1}{|\Omega|} \sum_{\omega \in \Omega} KL(P_{\mb{\rho}^{(\omega)}} \Vert P_{\mb{\rho}^{(0)}}) \le (nrq) \varepsilon^{2},
\end{align*}
and the Fano method yields
\begin{align*}
    \inf_{\hat{\mb{\rho}}} \sup_{\mb{\rho} \in \b{S}(\b{F}_{+}^{q \times q})} \b{E} \left[\vertiii{\hat{\mb{\rho}} - \mb{\rho}}_{2}^{2} \vert \mb{\rho} \right] &\ge \frac{\delta^{2}}{2} \left(1 - \frac{I(V; Y) + \log 2}{\log |\Omega|} \right) \\
    &\ge \frac{\varepsilon^{2}}{32 \alpha} \left(1 - \frac{(nrq) \varepsilon^{2} + \log 2}{(\log 2) D/8 } \right).
\end{align*}
Finally, taking $\varepsilon^{2} = c_{0} \frac{q}{3nr} \asymp D/[(n r) q]$ for some universal constant $c_{0} > 0$ establishes the lower bound, provided that $r \ge q c_{0} \tr{\left( \f{D}^{*} \f{D} \right)^{-1}|_{\b{F}_{sa, 0}^{q \times q}}}$ so that the perturbation $\mb{\rho}^{(\omega)}$ is a valid density matrix because:
\begin{align*}
    \vertiii{\mb{\Delta}^{(\omega)}}_{\infty}^{2} \le \vertiii{\mb{\Delta}^{(\omega)}}_{2}^{2} = \varepsilon^{2} \left( \sum_{j=1}^{D} \frac{1}{D \tau_{j}} \right) &= c_{0} \frac{q}{r D} \tr{\left(\f{D}^{*} \f{D} \right)^{-1}|_{\b{F}_{sa, 0}^{q \times q}}}  < \frac{1}{q^{2}}.
\end{align*}
\end{proof}

\begin{proof}[Proof of \cref{thm:loss:max:eigen}]
To handle multiplicities explicitly, we use the stretched indices $\ell \in \{1, \dots, q\}$ such that $\tilde{\lambda}_{i\ell} = \lambda_{i k_{i}(\ell)}$. The normalized coefficients are $\tilde{d}_{i\ell}[\m{S}] = d_{i k_{i}(\ell)}[\m{S}]/m_{i k_{i}(\ell)}$ and similarly for $\tilde{p}_{i\ell}$. The discrepancy operator acts on $\c{R}(K)$ as:
\begin{align*}
    \mb{\Delta}_{i} := \m{T}_{K}^{\mb{\nu}_{i, \m{S}}} - \m{T}_{K}^{\hat{\mb{\nu}}_{i, \mb{\rho}}} &= \b{E}[k_{Y} k_{Y}^{*} \mid \m{S}] - \frac{1}{r} \sum_{j=1}^{r} k_{Y_{ij}} k_{Y_{ij}}^{*} \\
    &= \sum_{k=1}^{q_{i}} (d_{ik}[\m{S}] - p_{ik}) k_{\lambda_{ik}} k_{\lambda_{ik}}^{*} = \sum_{\ell=1}^{q} (\tilde{d}_{i\ell}[\m{S}] - \tilde{p}_{i\ell}) k_{\tilde{\lambda}_{i\ell}} k_{\tilde{\lambda}_{i\ell}}^{*}.
\end{align*}
Since this operator has finite rank ($\le q_i$), its Schatten norms are well-defined. Let $\c{V}_{i} = \spann \{k_{\tilde{\lambda}_{i\ell}}\}_{\ell=1}^q \subset \c{R}(K)$. The operator $\mb{\Delta}_i$ vanishes on $\c{V}_{i}^{\perp}$, so we restrict our attention to $\c{V}_i$.
Consider the isometric isomorphism: 
\begin{equation*}
    \m{i} : (\c{V}_{i}, \innpr{\cdot}{\cdot}_{\c{R}(K)}) \to (\text{Ran}[\m{K}_{i}^{1/2}], \innpr{\cdot}{\cdot}_{\b{F}^{q}}), \, \sum_{\ell=1}^{q} \alpha_{\ell} k_{\tilde{\lambda}_{i\ell}} \mapsto \m{K}_{i}^{1/2} \mb{\alpha}, \quad \mb{\alpha} = [\alpha_{1}, \cdots, \alpha_{q}]^{\top}.
\end{equation*}

Under this isomorphism, the quadratic form of the discrepancy operator becomes:
\begin{align*}
    \innpr{\mb{\Delta}_{i} f}{g}_{\c{R}(K)} &= \sum_{\ell=1}^{q} (\tilde{d}_{i\ell}[\m{S}] - \tilde{p}_{i\ell}) f(\tilde{\lambda}_{i\ell}) \overline{g(\tilde{\lambda}_{i\ell})} \\
    &= \innpr{[\m{K}_{i}^{1/2} (\m{D}_{i}[\m{S}] - \m{P}_{i}) \m{K}_{i}^{1/2}] \m{i}(f)}{\m{i}(g)}_{\b{F}^{q}}, \quad f, g \in \c{V}_{i},
\end{align*}
since $[f(\tilde{\lambda}_{i1}), \cdots, f(\tilde{\lambda}_{iq})]^{\top} = \m{K}_{i} \mb{\alpha} = \m{K}_{i}^{1/2} [\m{i}(f)]$ and $[g(\tilde{\lambda}_{i1}), \cdots, g(\tilde{\lambda}_{iq})]^{\top} = \m{K}_{i}^{1/2} [\m{i}(g)]$.
Thus, $\mb{\Delta}_i$ is unitarily equivalent to the matrix $\m{i} \circ \mb{\Delta}_{i} \circ \m{i}^{-1} = \m{K}_{i}^{1/2} (\m{D}_{i}[\m{S}] - \m{P}_{i}) \m{K}_{i}^{1/2} \in \b{F}^{q \times q}$ restricted to its range, and their Schatten norms coincide.
\end{proof}

\begin{proof}[Proof of \cref{prop:superoperator_trace_norm}]
Let $\{\m{E}_{\ell} : 1 \le \ell \le D := \dim(\b{F}_{sa}^{q \times q})\}$ be an orthonormal basis for $\b{F}_{sa}^{q \times q}$. Then,
\begin{align*}
    &\innpr{\f{H}_{K} \m{E}_{\ell}}{\m{E}_{\ell}}_{2} = \frac{1}{n} \sum_{i=1}^{n} \sum_{k, k'=1}^{q_{i}} [\mb{\Omega}_i]_{kk'} \tr{\m{E}_{\ell} \m{\Pi}_{ik}} \tr{\m{E}_{\ell} \m{\Pi}_{ik'}}.
\end{align*}
For any $i, k, k'$, Plancherel's identity yields $\sum_{\ell=1}^{D} \tr{\m{E}_{\ell} \m{\Pi}_{ik}} \tr{\m{E}_{\ell} \m{\Pi}_{ik'}} = m_{ik} \delta_{kk'}$.
Consequently, the trace of the kernel Gram superoperator is given by
\begin{align*}
    \vertiii{\f{H}_{K}}_{1} &= 
    \sum_{\ell=1}^{D}
    \innpr{\f{H}_{K} \m{E}_{\ell}}{\m{E}_{\ell}}_{2} = \frac{1}{n} \sum_{i=1}^{n} \sum_{k=1}^{q_{i}} m_{ik} [\mb{\Omega}_i]_{kk} .
\end{align*}
\end{proof}

\begin{proof}[Proof of \cref{thm:quark:supop:pd}]
Using the cyclic property of the trace: 
\begin{align*} 
    &\vertiii{\m{K}_{i}^{1/2} (\m{D}_{i}[\m{S}] - \m{P}_{i}) \m{K}_{i}^{1/2}}_{2}^{2} = \tr{ (\m{D}_{i}[\m{S}] - \m{P}_{i}) \m{K}_{i} (\m{D}_{i}[\m{S}] - \m{P}_{i}) \m{K}_{i} } \\ 
    &= \sum_{\ell, \ell'=1}^{q} [(\m{D}_{i}[\m{S}] - \m{P}_{i})]_{\ell \ell} [\m{K}_{i}]_{\ell \ell'} [(\m{D}_{i}[\m{S}] - \m{P}_{i})]_{\ell' \ell'} [\m{K}_{i}]_{\ell' \ell} \\
    &= \sum_{k, k'=1}^{q_{i}} m_{ik} m_{ik'} \frac{d_{ik}[\m{S}] - p_{ik}}{m_{ik}} |K(\lambda_{ik}, \lambda_{ik'})|^{2} \frac{d_{ik'}[\m{S}] - p_{ik'}}{m_{ik'}} \\
    &= (\left[d_{i k}[\m{S}] - p_{ik}  \right]_{1 \le k \le q_{i}})^{\top} \mb{\Omega}_{i} \left[d_{i k}[\m{S}] - p_{ik}  \right]_{1 \le k \le q_{i}}.
\end{align*} 
Consequently, the loss function becomes $\tfrac{1}{n} \mathrm{Loss}_{K}[\m{S}] = \innpr{\f{H}_{K} [\m{S}]}{\m{S}}_{2} - 2 \innpr{\m{S}}{\m{P}_{K}}_{2} + \mathrm{Const.}$ 

Now, we assume that the squared kernel $|K|^2 : \b{R} \times \b{R} \to \b{R}$ is strictly p.d., and we claim that $\f{H}_{K}$ is strictly p.d.. It is trivial that $\f{H}_{K}$ is n.n.d., hence it suffices to show injectivity. Assume that $\f{H}_{K} [\m{S}] = \m{0}$ for some $\m{S} \in \b{F}_{sa}^{q \times q}$. Then, we have $d_{ik}[\m{S}] = 0$ for all $i, k$ as the Gram matrix $\mb{\Omega}_{i} \succ \m{0}$ is strictly p.d. for any $i$. This implies that $\m{S} \in \b{F}_{sa, 0}^{q \times q}$, and the completeness of the design implies that $\m{S} = \m{0}$. 
As a result, the loss function is strictly convex, and we have a unique solution.
To incorporate the unit-trace constraint, consider the Lagrangian:
\begin{align*} 
    \c{L}(\m{S}, \lambda) &= \frac{1}{2} \innpr{\f{H}_{K}[\m{S}]}{\m{S}}_{2} - \innpr{\m{S}}{\m{P}_{K}}_{2} + \lambda (\tr{\m{S}} - 1) \\
    &= \frac{1}{2} \innpr{\f{H}_{K}[\m{S}]}{\m{S}}_{2} - \innpr{\m{S}}{\m{P}_{K}}_{2} + \lambda (\innpr{\m{S}}{\m{I}_{q}}_{2} - 1).
\end{align*} 
Taking the gradient with respect to $\m{S}$ and setting it to zero: 
\begin{align*} 
    \nabla_{\m{S}} \c{L} = \f{H}_{K}[\m{S}] - \m{P}_K + \lambda \m{I}_q = \m{0} \quad \Longrightarrow \quad \hat{\mb{\rho}}^{K} = \f{H}_K^{-1}(\m{P}_K) - \lambda \f{H}_K^{-1}(\m{I}_{q}).
\end{align*} 
Then, solving for $\lambda \in \b{R}$ yields 
\begin{align*}
    1 = \tr{\f{H}_K^{-1}(\m{P}_K)} - \lambda \tr{\f{H}_K^{-1}(\m{I}_{q})} \quad \Longrightarrow \quad \lambda = \frac{\tr{\f{H}_K^{-1}(\m{P}_K)} - 1}{\tr{\f{H}_K^{-1}(\m{I}_{q})}}.
\end{align*}
\end{proof}

\begin{proof}[Proof of \cref{cor:unbiased}]
Straightforward since $\b{E}[\m{P}_{K} \mid \mb{\rho}] = \f{H}_K[\m{I}_{q}]$ and
\begin{align*}
    \hat{\mb{\rho}}^{K} - \mb{\rho} = \f{H}_K^{-1}[\m{P}_K - \f{H}_K[\m{I}_{q}] ] - \frac{ \tr{\f{H}_K^{-1}[\m{P}_K - \f{H}_K^{-1}[\m{I}_{q}] ]}}{\tr{\f{H}_K^{-1}[\m{I}_{q}]}} \f{H}_K^{-1}[\m{I}_{q}] \in \b{F}_{sa, 0}^{q \times q}.
\end{align*}
\end{proof}

\begin{proof}[Proof of \cref{prop:ker:tr:centre}]
Fix $\m{S}, \tilde{\m{S}} \in \b{F}_{sa}^{q \times q}$. Note that $\f{H}_K^{-1}$ is self-adjoint, hence $\innpr{\f{H}_K^{-1}[\m{I}_{q}]}{\tilde{\m{S}}}_{2} = \innpr{\f{H}_K^{-1}[\tilde{\m{S}}]}{\m{I}_{q}}_{2} = \tr{\f{H}_K^{-1}[\tilde{\m{S}}]}$.
This gives
\begin{align*}
    \innpr{\f{C}_{K}[\m{S}]}{\tilde{\m{S}}}_{2} &= \innpr{\m{S} - \frac{ \tr{\m{S}}}{\tr{\f{H}_K^{-1}[\m{I}_{q}]}} \f{H}_K^{-1}[\m{I}_{q}]}{\tilde{\m{S}}}_{2} = \innpr{\m{S}}{\tilde{\m{S}}}_{2} - \frac{ \tr{\m{S}} \tr{\f{H}_K^{-1}[\tilde{\m{S}}]}}{\tr{\f{H}_K^{-1}[\m{I}_{q}]}} \\
    &= \innpr{\m{S}}{\tilde{\m{S}} - \frac{ \tr{\f{H}_K^{-1}[\tilde{\m{S}}]}}{\tr{\f{H}_K^{-1}[\m{I}_{q}]}} \m{I}_{q}}_{2} = \innpr{\m{S}}{\f{C}_{K}^{*}[\tilde{\m{S}}]}_{2}.
\end{align*}
Observe that
\begin{align*}
    \f{H}_{K}^{-1} \f{C}_{K}^{*} [\m{S}] = \f{H}_{K}^{-1} \left[ \m{S} - \frac{ \tr{\f{H}_K^{-1}[\m{S}]}}{\tr{\f{H}_K^{-1}[\m{I}_{q}]}} \m{I}_{q} \right] 
    = \f{C}_{K} \f{H}_{K}^{-1} [\m{S}], 
\end{align*}
which yields the decomposition:
\begin{align*}
    \hat{\mb{\rho}}^{K} - \frac{\m{I}_{q}}{q} &= \left( \f{H}_K^{-1}[\m{P}_K] - \tr{\f{H}_K^{-1}[\m{P}_K]} \frac{\f{H}_K^{-1}[\m{I}_{q}]}{\tr{\f{H}_K^{-1}[\m{I}_{q}]}} \right) - \left( \frac{\m{I}_{q}}{q} - \tr{\frac{\m{I}_{q}}{q}} \frac{\f{H}_K^{-1}[\m{I}_{q}]}{\tr{\f{H}_K^{-1}[\m{I}_{q}]}} \right) \\
    &= \f{C}_{K} \big[ \f{H}_K^{-1}[\m{P}_K] \big] - \f{C}_{K} \left[ \frac{\m{I}_{q}}{q} \right] = \f{C}_{K} \left[ \f{H}_K^{-1}[\m{P}_K] - \frac{\m{I}_{q}}{q} \right] = \f{A}_{K} \left[\m{P}_K - \frac{\m{I}_{q}}{q} \right].
\end{align*}
\end{proof}

\begin{proof}[Proof of \cref{prop:eff:data:cov}]
From $\b{E}[\m{P}_K \mid \mb{\rho}] =  \f{H}_K [\mb{\rho}]$ and \cref{prop:ker:tr:centre}, we get
\begin{align*}
    \hat{\mb{\rho}}^{K} - \mb{\rho} &= \f{C}_{K} \left[ \f{H}_K^{-1}[\m{P}_K] - \frac{\m{I}_{q}}{q} \right] - \left[ \mb{\rho} - \frac{\m{I}_{q}}{q} \right] = \f{C}_{K} \left[ \f{H}_K^{-1}[\m{P}_K] - \mb{\rho} \right] \\
    &= \f{C}_{K} \f{H}_K^{-1} \left[ \m{P}_K - \f{H}_K [\mb{\rho}] \right] = \f{H}_K^{-1} \f{C}_{K}^{*}  \left[ \m{P}_K - \f{H}_K [\mb{\rho}] \right].
\end{align*}
Fix a self-adjoint matrix $\m{S} \in \b{F}_{sa}^{q \times q}$. Then
\begin{align*}
    \innpr{\m{P}_K - \f{H}_K [\mb{\rho}]}{\m{S}}_{2} = \frac{1}{n} \sum_{i=1}^{n} \sum_{k, k'=1}^{q_{i}} [\mb{\Omega}_i]_{kk'} (p_{ik} - d_{ik}[\mb{\rho}]) d_{ik'}[\m{S}].
\end{align*}
hence the action of the covariance is given by
\begin{align*}
    &\b{E}[\innpr{\m{S}}{\m{P}_{K} - \f{H}_K [\mb{\rho}]}_{2} (\m{P}_{K} - \f{H}_K [\mb{\rho}]) \mid \mb{\rho}] \\
    &= \frac{1}{n^{2}} \sum_{i=1}^{n} \sum_{k, k', k'', k''' =1}^{q_{i}} [\mb{\Omega}_i]_{kk'} \Cov (p_{ik}, p_{ik''}  \mid \mb{\rho}) [\mb{\Omega}_i]_{k''k'''} d_{ik'}[\m{S}] \m{\Pi}_{ik'''} = (A) - (B),
\end{align*}
due to \eqref{eq:multinom:cov}, where
\begin{align*}
    (A) &= \frac{1}{n^{2} r} \sum_{i=1}^{n} \sum_{k, k', k''' =1}^{q_{i}} d_{ik}[\mb{\rho}] [\mb{\Omega}_i]_{kk'}  d_{ik'}[\m{S}] [\mb{\Omega}_i]_{kk'''}  \m{\Pi}_{ik'''} \\
    &= \frac{1}{n^{2} r} \sum_{i=1}^{n} \sum_{k=1}^{q_{i}} d_{ik}[\mb{\rho}] \left( \sum_{k'=1}^{q_{i}} [\mb{\Omega}_i]_{kk'} d_{ik'}[\m{S}]  \right) \left( \sum_{k'=1}^{q_{i}}  [\mb{\Omega}_i]_{kk'''}  \m{\Pi}_{ik'''} \right) \\
    &= \frac{1}{n^{2} r} \sum_{i=1}^{n} \sum_{k=1}^{q_{i}} d_{ik}[\mb{\rho}] \innpr{\m{S}}{\m{M}_{ik}^{K}}_{2} \m{M}_{ik}^{K}, \\
    (B) &= \frac{1}{n^{2} r} \sum_{i=1}^{n} \sum_{k, k', k'', k''' =1}^{q_{i}} d_{ik}[\mb{\rho}] [\mb{\Omega}_i]_{kk'} d_{ik'}[\m{S}] d_{ik''}[\mb{\rho}] [\mb{\Omega}_i]_{k''k'''}  \m{\Pi}_{ik'''} \\
    &= \frac{1}{n^{2} r} \sum_{i=1}^{n} \left( \sum_{k, k' =1}^{q_{i}} d_{ik}[\mb{\rho}] [\mb{\Omega}_i]_{kk'} d_{ik'}[\m{S}] \right) \left( \sum_{k'', k''' =1}^{q_{i}} d_{ik''}[\mb{\rho}] [\mb{\Omega}_i]_{k''k'''}  \m{\Pi}_{ik'''} \right) \\
    &= \frac{1}{n^{2} r} \sum_{i=1}^{n} \innpr{\m{S}}{ \sum_{k=1}^{q_{i}} d_{ik}[\mb{\rho}] \sum_{k'=1}^{q_{i}} \m{M}_{ik'}^{K}}_{2} \left( \sum_{k''=1}^{q_{i}} d_{ik''}[\mb{\rho}] \sum_{k'''=1}^{q_{i}} \m{M}_{ik'''}^{K} \right).
\end{align*}
Substituting these results yields the expression for $\f{S}_{K, \mb{\rho}}$.
\end{proof}

\begin{proof}[Proof of \cref{thm:quark:clt}]
Following the proof of \cref{prop:eff:data:cov}, it is then a direct consequence that the CLT holds for the effective data matrix:
\begin{align*}
\sqrt{r}(\m{P}_K - \f{H}_K [\mb{\rho}]) \xrightarrow{d}  N\left(\m{0}, \f{S}_{K, \mb{\rho}} \right).
\end{align*}
Since $\f{A}_{K}$ is a bounded operator, the CLT for $\hat{\mb{\rho}}^{K} - \mb{\rho} = \f{A}_{K} \left[ \m{P}_K - \f{H}_K [\mb{\rho}] \right]$ follows by the continuous mapping theorem \cite{billingsley2013convergence}.
\end{proof}

\begin{proof}[Proof of \cref{thm:innpr:concen}]
Let $\m{T} = \f{A}_{K} [\m{S}]$ and consider the following random variable:
\begin{align*}
    Z :&= \innpr{\hat{\mb{\rho}}^{K} - \mb{\rho}}{\m{S}}_{2} = \innpr{\m{P}_K - \f{H}_K [\mb{\rho}]}{\m{T}}_{2} \\
    &= \frac{1}{n} \sum_{i=1}^{n} \sum_{k=1}^{q_{i}} (p_{ik} - d_{ik}[\mb{\rho}]) \underbrace{\innpr{\m{M}_{ik}^{K}}{\m{T}}_{2}}_{=:\beta_{ik}} 
    = \sum_{i=1}^{n} \sum_{j=1}^{r} \underbrace{\frac{1}{nr} \innpr{Y_{ij} - \m{d}_{i}}{\mb{\beta}_{i}}_{\b{R}^{q_{i}}}}_{=:Z_{ij}},
\end{align*}
where $Y_{ij} \sim \mathrm{Categorical}(\m{d}_{i})$, $\m{d}_{i} = [d_{i1}[\mb{\rho}], \cdots, d_{iq_{i}}[\mb{\rho}]]^{\top} \in \b{R}^{q_{i}}$, and $\mb{\beta}_{i} = [\beta_{i1}, \cdots, \beta_{iq_{i}}]^{\top} \in \b{R}^{q_{i}}$. We then use the Cram\'{e}r–Chernoff method to obtain the concentration inequality \cite{boucheron2003concentration}:
\begin{align*}
    \b{P} \left[Z \ge t \mid \mb{\rho} \right] \le \exp ( - g(t)), \quad g(t) := \sup_{\lambda > 0} (\lambda t - \log \b{E} \left[ e^{\lambda Z} \mid \mb{\rho} \right]).
\end{align*}

For each $1 \le i \le n$ and $1 \le j \le r$, note that $\b{E}[Z_{ij} \mid \mb{\rho}]= 0$, thus
\begin{align*}
    \log \b{E}\left[ e^{\lambda Z_{ij}} \mid \mb{\rho} \right] &= \log \left( 1 + \b{E}\left[ e^{\lambda Z_{ij}} - \lambda Z_{ij} - 1 \mid \mb{\rho} \right] \right) \\
    &\le \log \left( 1 + \b{E}[Z_{ij}^{2} \mid \mb{\rho}] \left( e^{\lambda} - \lambda - 1\right) \right) \le \b{E}[Z_{ij}^{2} \mid \mb{\rho}] \left( e^{\lambda} - \lambda - 1\right),
\end{align*}
and thus
\begin{align*}
    \log \b{E} \left[ e^{\lambda Z} \mid \mb{\rho} \right] = \sum_{i=1}^{n} \sum_{j=1}^{r} \b{E}\left[ e^{\lambda Z_{ij}} \mid \mb{\rho} \right] 
    \le \frac{m_{2}}{r} \left( e^{\lambda} - \lambda - 1\right),
\end{align*}
where
\begin{align*}
    m_{2} := r \sum_{i=1}^{n} \sum_{j=1}^{r} \b{E}[Z_{ij}^{2} \mid \mb{\rho}] = r\Cov(Z \mid \mb{\rho}) = \innpr{\f{S}_{K, \mb{\rho}} [\m{T}]}{\m{T}}_{2} = \innpr{\f{A}_{K} \f{S}_{K, \mb{\rho}} \f{A}_{K} [\m{S}]}{\m{S}}_{2}.
\end{align*}
Consequently, \cite[(2.10)]{boucheron2003concentration} yields the concentration inequality:
\begin{align*}
    g(t) \ge \sup_{\lambda > 0} \left( \lambda t - \frac{m_{2}}{r}(e^{\lambda} - \lambda - 1) \right) = \frac{m_{2}}{r} h \left( \frac{t}{m_{2}/r} \right) \ge \frac{t^{2}}{2(m_{2}/r + t/3)}.
\end{align*}
\end{proof}

\section{Additional Simulations}\label{sec:supp_material_simulation_study}
All simulations are performed in the real setting $\b{F}=\b{R}$.
As in \cref{ssec:emp:val}, we utilize the Gaussian kernel $K_{c}(x,y) := \exp (-c \|x-y\|^{2})$ for $c>0$, and write $K:=K_{1}$ for brevity. Here, the QHS is fixed to $\c{Q}=\b{R}^{8}$. We set the underlying density matrix as 
$\mb{\rho} = \Diag (\lambda_{1}, \dots, \lambda_{8})$, with eigenvalues 
$0.4 \geq 0.2 \geq 0.15 \geq 0.08 \geq 0.06 \geq 0.05 \geq 0.04 \geq 0.02$. We then sample a design of $n=100$ observables uniformly from $O(\b{R}^{8})$. 

\vspace{1em} \noindent
\textbf{Spectral Behavior.}
We investigate the spectral estimation errors based on $1000$ Monte Carlo repetitions, each with $r=50$ independent shots per observable.
\cref{fig:spectral_behavior} reports the spectral estimation errors for both the LSE and QUARK estimator with respect to $K=K_{1}$.
In contrast to \cref{fig:spectral_behavior_complex}, we observe in \cref{fig:spectral_behavior} that the LSE consistently achieves lower median estimation error for both eigenvalue and eigenvector estimation compared to the QUARK estimator.

\begin{figure}[H]
    \centering
    \begin{minipage}[t]{0.48\textwidth}
        \centering
        \includegraphics[width=\textwidth]{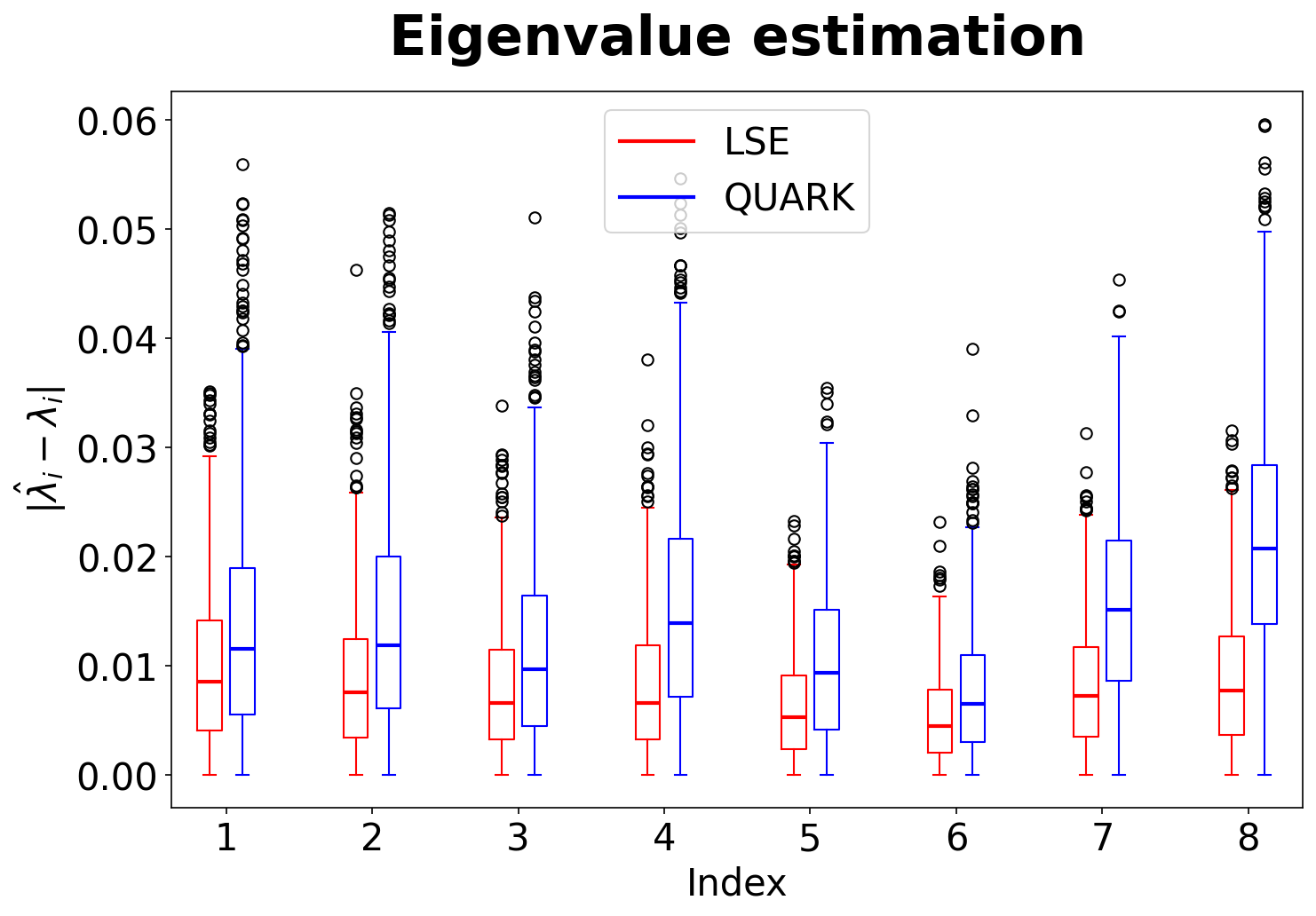}
    \end{minipage}
    \hfill
    \begin{minipage}[t]{0.48\textwidth}
        \centering
        \includegraphics[width=\textwidth]{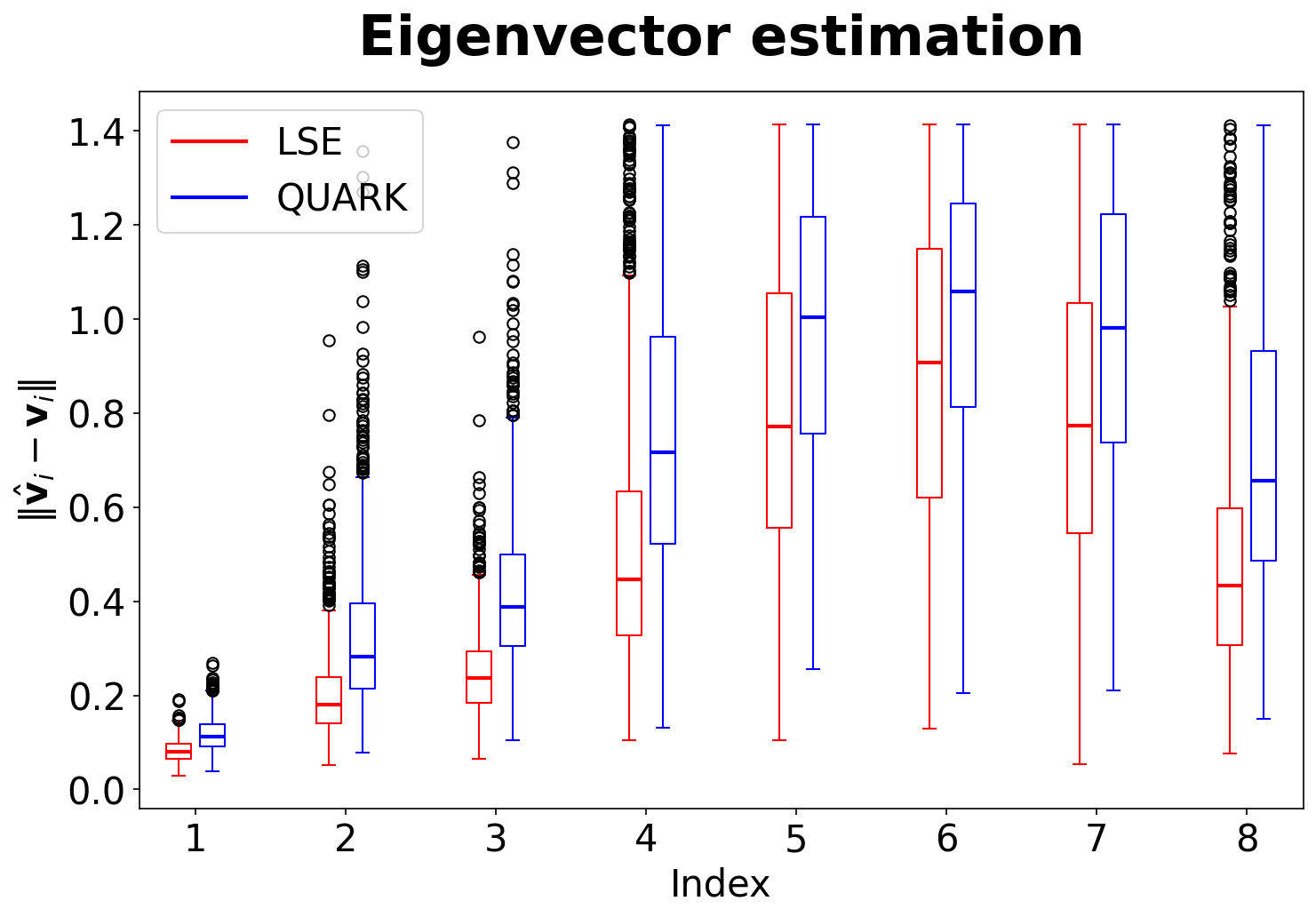}
    \end{minipage}
    \caption{Spectral behavior of the LSE and QUARK estimator, based on $1000$ Monte Carlo repetitions with $\b F=\b R$. The boxplots are ordered by decreasing eigenvalue magnitude. \textbf{Left:} Absolute eigenvalue errors $|\hat{\lambda}^{\text{LSE}}_{i}-\lambda_{i}|$ (red) and $|\hat{\lambda}^{K}_{i}-\lambda_{i}|$ (blue). \textbf{Right:} Absolute eigenvector errors $\|\hat{\m{v}}^{\text{LSE}}_{i}-\m{v}_{i}\|$ (red) and $\| \hat{\m{v}}_{i}^{K}-\m{v}_{i}\|$ (blue).}
    \label{fig:spectral_behavior}
\end{figure}

\noindent
\textbf{Asymptotic Performances.}
As in \cref{ssec:emp:val}, we vary the number of samples per observable $r$ increases. 
We perform the reconstructions for the LSE and various QUARK estimators. We then compare the results to the theoretical MSE under an exact unitary design. Note that since $\c{Q}=\b{R}^8$, the damping factor satisfies $\alpha_{\c Q}=1/5$. Similar phenomenon is observed as in \cref{ssec:emp:val}: The empirical behavior of the LSE closely follows the theoretical scaling. While all QUARK estimators exhibit larger errors asymptotically in comparison to the LSE estimator, the errors get closer to the theoretical LSE bound as $c$ increases. Nonetheless, \cref{fig:mse_vs_r} validates the asymptotic $O(r^{-1})$ decay predicted by the theory.

\begin{figure}[H]
    \centering
    \includegraphics[width=\textwidth]{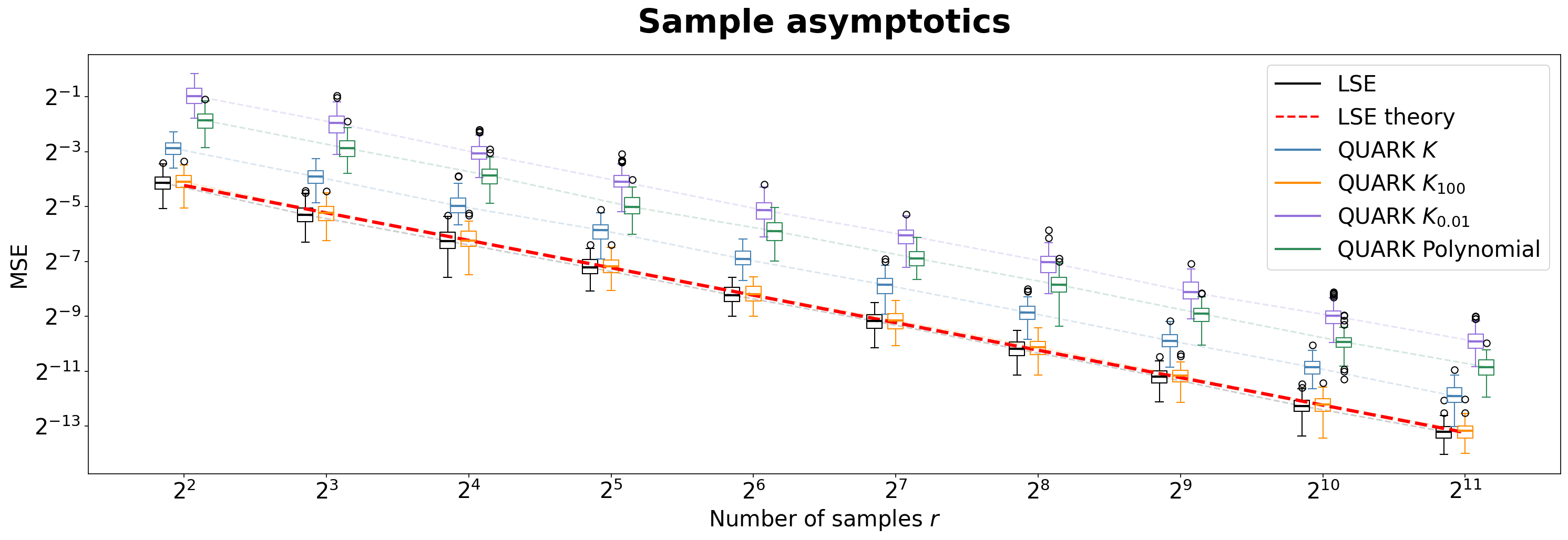}
    \caption{Asymptotic performances in terms of MSE of the LSE and various QUARK estimators with $\b{F}=\b{R}$.}
    \label{fig:mse_vs_r}
\end{figure}
\end{appendix}

\end{document}